\documentclass[12pt,a4paper]{amsart}

\usepackage{amsmath}
\usepackage{amsfonts}
\usepackage{latexsym}
\usepackage{graphicx}
\usepackage{amssymb}
\usepackage{amsthm}
\usepackage[margin=2cm]{geometry}
\usepackage{url}
\usepackage{color}
\usepackage{enumerate}
\usepackage[shortlabels]{enumitem} %
\usepackage[small]{caption}
\usepackage{cite}       %

\usepackage{todonotes}

\usepackage{algorithm}     %
\usepackage[noend]{algpseudocode}
\algrenewcommand\algorithmicrequire{\textbf{Precondition:}}
\algrenewcommand\algorithmicensure{\textbf{Postcondition:}}

\usepackage{array}           %
\usepackage{stmaryrd}

\usepackage{hyperref}

\usepackage{tikz}
\usepackage{tikz-cd}      %
\usetikzlibrary{decorations.pathreplacing} %
\usetikzlibrary{decorations.markings}      %
\usetikzlibrary{arrows.meta}

\usepackage[T1]{fontenc}     %
\usepackage{lmodern}         %
\usepackage[utf8]{inputenc}  %

\usepackage{sagetex}
\usepackage{accsupp}
\newcommand{\noncopynumber}[1]{
    \BeginAccSupp{method=escape,ActualText={}}
    #1
    \EndAccSupp{}
}

\lstdefinestyle{SageInput}{style=DefaultSageInput,basicstyle={\footnotesize\ttfamily}}
\newtheorem{definition}{Definition}[section]
\newtheorem{lemma}[definition]{Lemma}
\newtheorem{proposition}[definition]{Proposition}
\newtheorem{example}[definition]{Example}

\newtheorem{corollary}[definition]{Corollary}

\newtheorem{remark}[definition]{Remark}

\newtheorem{mainquestion}{Question}
\newtheorem{maintheorem}{Theorem}
\newtheorem{maincorollary}[maintheorem]{Corollary}

\newtheorem*{THEOREMI}{Theorem~\ref{thm:similar-to-itself}}
\newtheorem*{THEOREMII}{Theorem~\ref{thm:primitivity}}
\newtheorem*{THEOREMIII}{Theorem~\ref{thm:minimal}}
\newtheorem*{THEOREMIV}{Theorem~\ref{thm:n=1-equiv-Ammann}}

\numberwithin{equation}{section}

\newcommand{\N}{\mathbb{N}}
\newcommand{\Z}{\mathbb{Z}}
\newcommand{\Q}{\mathbb{Q}}
\newcommand{\R}{\mathbb{R}}

\newcommand{\ba}{{\boldsymbol{a}}}
\newcommand{\be}{{\boldsymbol{e}}}

\newcommand{\bk}{{\boldsymbol{k}}}
\newcommand{\bl}{{\boldsymbol{\ell}}}
\newcommand{\bn}{{\boldsymbol{n}}}
\newcommand{\bm}{{\boldsymbol{m}}}
\newcommand{\bp}{{\boldsymbol{p}}}

\newcommand{\bx}{{\boldsymbol{x}}}
\newcommand{\zero}{{\boldsymbol{0}}}
\newcommand{\UN}{{\boldsymbol{1}}}

\newcommand{\Acal}{\mathcal{A}}
\newcommand{\Bcal}{\mathcal{B}}
\newcommand{\Dcal}{\mathcal{D}}
\newcommand{\Fcal}{\mathcal{F}}
\newcommand{\Lcal}{\mathcal{L}}

\newcommand{\Tcal}{\mathcal{T}}
\newcommand{\Ucal}{\mathcal{U}}
\newcommand{\Xcal}{\mathcal{X}}

\newcommand{\ibar}{\overline{i}}
\newcommand{\jbar}{\overline{j}}
\newcommand{\nbar}{\overline{n}}
\newcommand{\nunder}{\underline{n}}
\newcommand{\iunder}{\underline{i}}

\newcommand{\sctop}{\textsc{Top}}
\newcommand{\scbottom}{\textsc{Bottom}}
\newcommand{\scright}{\textsc{Right}}
\newcommand{\scleft}{\textsc{Left}}

\newcommand{\shiftclosure}[1]{{\overline{#1}^{\sigma}}}
\newcommand{\scsize}{\textsc{size}}
\newcommand{\scarea}{\textsc{area}}
\newcommand{\height}{\textsc{height}}
\newcommand{\width}{\textsc{width}}

\newcommand{\RecurrentVertices}{\textsc{RecurrentVertices}}

\newcommand{\ZZrange}[2]{{\llbracket#1,#2\rrbracket}}

\newcommand{\dist}{\mathrm{dist}}

\newcommand\defn[1]{\textbf{#1}}

\newcommand\labelinside[6]{
\draw[draw] (#1,#2) rectangle (#1+\size,#2+\size);
\node[rotate=-90,black] at (#1+0.8, #2+0.5) {$#3$};
\node[rotate=0,  black] at (#1+0.5, #2+0.8) {$#4$};
\node[rotate=-90,black] at (#1+0.2, #2+0.5) {$#5$};
\node[rotate=0,  black] at (#1+0.5, #2+0.2) {$#6$};
}

\newcommand\labeloutside[6]{
\draw[draw] (#1+\size,#2) -- node[swap]{$#3$} (#1+\size,#2+\size)
                          -- node[swap]{$#4$} (#1,#2+\size)
                          -- node[swap]{$#5$} (#1,#2)
                          -- node[swap]{$#6$} (#1+\size,#2);
}

\def\dt{.25}

\newcommand\tile[7]{
\begin{scope}
\draw[draw=none,fill=#1] (#2,#3) rectangle (#2+\size,#3+\size);
\labeloutside{#2}{#3}{#4}{#5}{#6}{#7}
\end{scope}}

\newcommand\tileV[7]{
\begin{scope}[xshift=#2cm,yshift=#3cm]
\draw[draw=none,fill=#1] (#2+\dt,#3) rectangle (#2+\size-\dt,#3+\size);
\labeloutside{#2}{#3}{#4}{#5}{#6}{#7}
\end{scope}}

\newcommand\tileH[7]{
\begin{scope}[xshift=#2cm,yshift=#3cm]
\draw[draw=none,fill=#1] (#2,#3+\dt) rectangle (#2+\size,#3+\size-\dt);
\labeloutside{#2}{#3}{#4}{#5}{#6}{#7}
\end{scope}}

\def\size{1}
\def\widthYellow{.5}
\def\widthCyan{.5}
\def\greenEllipseRadius{.7}
\def\pinkEllipseRadius{.4}
\def\ourColorGreen{green!80}
\def\ourColorYellow{yellow}
\def\ourColorBlue{cyan!70}

\newcommand\tileHgreenBackground[2]{
\draw[fill=\ourColorYellow,draw=none]
({#1+\size},{#2+.5*\size+.5*\widthYellow}) arc [start angle=90, end angle=270,
x radius=\greenEllipseRadius, y radius={.5*\widthYellow}] -- cycle;
\draw[fill=\ourColorBlue,draw=none]
({#1},{#2+.5*\size+.5*\widthCyan}) arc [start angle=90, end angle=-90,
x radius=\greenEllipseRadius, y radius={.5*\widthCyan}] -- cycle;
\begin{scope}
\clip
({#1+\size},{#2+.5*\size+.5*\widthYellow}) arc [start angle=90, end angle=270,
x radius=\greenEllipseRadius, y radius={.5*\widthYellow}] -- cycle;
\draw[fill=\ourColorGreen,draw=none]
({#1},{#2+.5*\size+.5*\widthCyan}) arc [start angle=90, end angle=-90,
x radius=\greenEllipseRadius, y radius={.5*\widthCyan}] -- cycle;
\end{scope}
}

\newcommand\tileHpinkBackground[2]{
\draw[fill=\ourColorBlue,draw=none]
({#1+\size},{#2+.5*\size+.5*\widthCyan}) arc [start angle=90, end angle=270,
x radius=\pinkEllipseRadius, y radius={.5*\widthCyan}] -- cycle;
\draw[fill=\ourColorYellow,draw=none]
({#1},{#2+.5*\size+.5*\widthYellow}) arc [start angle=90, end angle=-90,
x radius=\pinkEllipseRadius, y radius={.5*\widthYellow}] -- cycle;
\begin{scope}
\clip
({#1+\size},{#2+.5*\size+.5*\widthCyan}) arc [start angle=90, end angle=270,
x radius=\pinkEllipseRadius, y radius={.5*\widthCyan}] -- cycle;
\draw[fill=pink,draw=none]
({#1},{#2+.5*\size+.5*\widthYellow}) arc [start angle=90, end angle=-90,
x radius=\pinkEllipseRadius, y radius={.5*\widthYellow}] -- cycle;
\end{scope}
}

\newcommand\tileVgreenBackground[2]{
\draw[fill=\ourColorYellow,draw=none]
({#1+.5*\size+.5*\widthYellow},{#2+\size}) arc [start angle=0, end angle=-180,
x radius={.5*\widthYellow}, y radius=\greenEllipseRadius] -- cycle;
\draw[fill=\ourColorBlue,draw=none]
({#1+.5*\size+.5*\widthCyan},{#2}) arc [start angle=0, end angle=180,
x radius={.5*\widthCyan}, y radius=\greenEllipseRadius] -- cycle;
\begin{scope}
\clip
({#1+.5*\size+.5*\widthYellow},{#2+\size}) arc [start angle=0, end angle=-180,
x radius={.5*\widthYellow}, y radius=\greenEllipseRadius] -- cycle;
\draw[fill=\ourColorGreen,draw=none]
({#1+.5*\size+.5*\widthCyan},{#2}) arc [start angle=0, end angle=180,
x radius={.5*\widthCyan}, y radius=\greenEllipseRadius] -- cycle;
\end{scope}
}

\newcommand\tileVpinkBackground[2]{
\draw[fill=\ourColorBlue,draw=none]
({#1+.5*\size+.5*\widthCyan},{#2+\size}) arc [start angle=0, end angle=-180,
x radius={.5*\widthCyan}, y radius=\pinkEllipseRadius] -- cycle;
\draw[fill=\ourColorYellow,draw=none]
({#1+.5*\size+.5*\widthYellow},{#2}) arc [start angle=0, end angle=180,
x radius={.5*\widthYellow}, y radius=\pinkEllipseRadius] -- cycle;
\begin{scope}
\clip
({#1+.5*\size+.5*\widthCyan},{#2+\size}) arc [start angle=0, end angle=-180,
x radius={.5*\widthCyan}, y radius=\pinkEllipseRadius] -- cycle;
\draw[fill=pink,draw=none]
({#1+.5*\size+.5*\widthYellow},{#2}) arc [start angle=0, end angle=180,
x radius={.5*\widthYellow}, y radius=\pinkEllipseRadius] -- cycle;
\end{scope}
}

\newcommand\tileJunctionBackground[6]{
\def\widthR{0.5}
\def\widthT{0.5}
\def\widthL{0.5}
\def\widthB{0.5}
\coordinate (R) at (#1+\size,{#2+(\size+\widthR)/2}) {};
\coordinate (T) at ({#1+(\size+\widthT)/2},#2+\size) {};
\coordinate (L) at (#1,{#2+(\size-\widthL)/2}) {};
\coordinate (B) at ({#1+(\size-\widthB)/2},#2) {};
\def\du{.02}
\draw [fill=#3,draw=none] (R) to [bend right=20] (#1+\size/2+\du,#2+\size/2-\du) -- (#1+\size,#2) -- cycle;
\draw [fill=#4,draw=none] (T) to [bend left=20]  (#1+\size/2-\du,#2+\size/2+\du) -- (#1,#2+\size) -- cycle;
\draw [fill=#5,draw=none] (L) to [bend right=20] (#1+\size/2-\du,#2+\size/2+\du) -- (#1,#2+\size) -- cycle;
\draw [fill=#6,draw=none] (B) to [bend left=20]  (#1+\size/2+\du,#2+\size/2-\du) -- (#1+\size,#2) -- cycle;
\draw[fill=white,draw=none] ({#1},{#2+\size})
-- ++ (0,{-(\size-\widthL)/2}) arc
[start angle=-90, end angle=0,
x radius={(\size-\widthT)/2}, y radius={(\size-\widthL)/2}];
\draw[fill=white,draw=none] ({#1+\size},{#2})
-- ++ (0,{(\size-\widthR)/2}) arc
[start angle=90, end angle=180,
x radius={(\size-\widthB)/2}, y radius={(\size-\widthR)/2}];
}

\newcommand\crossX[2]{
\draw[very thick] ({#1+\size*.4},{#2+\size*.4}) --
                    ({#1+\size*.6},{#2+\size*.6});
\draw[very thick] ({#1+\size*.4},{#2+\size*.6}) --
                    ({#1+\size*.6},{#2+\size*.4});
}

\newcommand\tileJunctionOOOO[6]{
\begin{scope}
\tileJunctionBackground{#1}{#2}{\ourColorBlue}{\ourColorBlue}{\ourColorBlue}{\ourColorBlue}
\labeloutside{#1}{#2}{#3}{#4}{#5}{#6}
\end{scope}}
\newcommand\tileJunctionOOOI[6]{
\begin{scope}
\tileJunctionBackground{#1}{#2}{\ourColorBlue}{\ourColorBlue}{\ourColorBlue}{\ourColorYellow}
\labeloutside{#1}{#2}{#3}{#4}{#5}{#6}
\end{scope}}
\newcommand\tileJunctionOOIO[6]{
\begin{scope}
\tileJunctionBackground{#1}{#2}{\ourColorBlue}{\ourColorBlue}{\ourColorYellow}{\ourColorBlue}
\labeloutside{#1}{#2}{#3}{#4}{#5}{#6}
\end{scope}}
\newcommand\tileJunctionOOII[6]{
\begin{scope}
\tileJunctionBackground{#1}{#2}{\ourColorBlue}{\ourColorBlue}{\ourColorYellow}{\ourColorYellow}
\labeloutside{#1}{#2}{#3}{#4}{#5}{#6}
\end{scope}}

\newcommand\tileJunctionOIIO[6]{
\begin{scope}
\tileJunctionBackground{#1}{#2}{\ourColorBlue}{\ourColorYellow}{\ourColorYellow}{\ourColorBlue}
\labeloutside{#1}{#2}{#3}{#4}{#5}{#6}
\crossX{#1}{#2}
\end{scope}}
\newcommand\tileJunctionOIII[6]{
\begin{scope}
\tileJunctionBackground{#1}{#2}{\ourColorBlue}{\ourColorYellow}{\ourColorYellow}{\ourColorYellow}
\labeloutside{#1}{#2}{#3}{#4}{#5}{#6}
\end{scope}}

\newcommand\tileJunctionIOOI[6]{
\begin{scope}
\tileJunctionBackground{#1}{#2}{\ourColorYellow}{\ourColorBlue}{\ourColorBlue}{\ourColorYellow}
\labeloutside{#1}{#2}{#3}{#4}{#5}{#6}
\crossX{#1}{#2}
\end{scope}}

\newcommand\tileJunctionIOII[6]{
\begin{scope}
\tileJunctionBackground{#1}{#2}{\ourColorYellow}{\ourColorBlue}{\ourColorYellow}{\ourColorYellow}
\labeloutside{#1}{#2}{#3}{#4}{#5}{#6}
\end{scope}}

\newcommand\tileJunctionIIII[6]{
\begin{scope}
\tileJunctionBackground{#1}{#2}{\ourColorYellow}{\ourColorYellow}{\ourColorYellow}{\ourColorYellow}
\labeloutside{#1}{#2}{#3}{#4}{#5}{#6}
\end{scope}}

\newcommand\tileJunctionGRAY[7]{
\begin{scope}
\tileJunctionBackground{#1}{#2}{#7}{#7}{#7}{#7}
\labeloutside{#1}{#2}{#3}{#4}{#5}{#6}
\end{scope}}

\newcommand\tileJunctionInsideOOOO[6]{
\begin{scope}
\tileJunctionBackground{#1}{#2}{\ourColorBlue}{\ourColorBlue}{\ourColorBlue}{\ourColorBlue}
\labelinside{#1}{#2}{#3}{#4}{#5}{#6}
\end{scope}}
\newcommand\tileJunctionInsideOOOI[6]{
\begin{scope}
\tileJunctionBackground{#1}{#2}{\ourColorBlue}{\ourColorBlue}{\ourColorBlue}{\ourColorYellow}
\labelinside{#1}{#2}{#3}{#4}{#5}{#6}
\end{scope}}
\newcommand\tileJunctionInsideOOIO[6]{
\begin{scope}
\tileJunctionBackground{#1}{#2}{\ourColorBlue}{\ourColorBlue}{\ourColorYellow}{\ourColorBlue}
\labelinside{#1}{#2}{#3}{#4}{#5}{#6}
\end{scope}}
\newcommand\tileJunctionInsideOOII[6]{
\begin{scope}
\tileJunctionBackground{#1}{#2}{\ourColorBlue}{\ourColorBlue}{\ourColorYellow}{\ourColorYellow}
\labelinside{#1}{#2}{#3}{#4}{#5}{#6}
\end{scope}}

\newcommand\tileJunctionInsideOIIO[6]{
\begin{scope}
\tileJunctionBackground{#1}{#2}{\ourColorBlue}{\ourColorYellow}{\ourColorYellow}{\ourColorBlue}
\labelinside{#1}{#2}{#3}{#4}{#5}{#6}
\crossX{#1}{#2}
\end{scope}}
\newcommand\tileJunctionInsideOIII[6]{
\begin{scope}
\tileJunctionBackground{#1}{#2}{\ourColorBlue}{\ourColorYellow}{\ourColorYellow}{\ourColorYellow}
\labelinside{#1}{#2}{#3}{#4}{#5}{#6}
\end{scope}}

\newcommand\tileJunctionInsideIIII[6]{
\begin{scope}
\tileJunctionBackground{#1}{#2}{\ourColorYellow}{\ourColorYellow}{\ourColorYellow}{\ourColorYellow}
\labelinside{#1}{#2}{#3}{#4}{#5}{#6}
\end{scope}}
\newcommand\tileJunctionInsideOXXO[8]{
\begin{scope}
\tileJunctionBackground{#1}{#2}{\ourColorBlue}{#7}{#8}{\ourColorBlue}
\labelinside{#1}{#2}{#3}{#4}{#5}{#6}
\end{scope}}
\newcommand\tileJunctionInsideOXXI[8]{
\begin{scope}
\tileJunctionBackground{#1}{#2}{\ourColorBlue}{#7}{#8}{\ourColorYellow}
\labelinside{#1}{#2}{#3}{#4}{#5}{#6}
\end{scope}}

\newcommand\tileJunctionInsideIXXI[8]{
\begin{scope}
\tileJunctionBackground{#1}{#2}{\ourColorYellow}{#7}{#8}{\ourColorYellow}
\labelinside{#1}{#2}{#3}{#4}{#5}{#6}
\end{scope}}

\newcommand\tileJunctionInsideGRAY[7]{
\begin{scope}
\tileJunctionBackground{#1}{#2}{#7}{#7}{#7}{#7}
\labelinside{#1}{#2}{#3}{#4}{#5}{#6}
\end{scope}}

\newcommand\tileHgreen[6]{
\begin{scope}
\tileHgreenBackground{#1}{#2}
\labeloutside{#1}{#2}{#3}{#4}{#5}{#6}
\end{scope}}

\newcommand\tileVgreen[6]{
\begin{scope}
\tileVgreenBackground{#1}{#2}
\labeloutside{#1}{#2}{#3}{#4}{#5}{#6}
\end{scope}}

\newcommand\tileHantigreen[6]{
\begin{scope}
\tileHpinkBackground{#1}{#2}
\labeloutside{#1}{#2}{#3}{#4}{#5}{#6}
\end{scope}}

\newcommand\tileVantigreen[6]{
\begin{scope}
\tileVpinkBackground{#1}{#2}
\labeloutside{#1}{#2}{#3}{#4}{#5}{#6}
\end{scope}}

\newcommand\tileHgreenInside[6]{
\begin{scope}
\tileHgreenBackground{#1}{#2}
\labelinside{#1}{#2}{#3}{#4}{#5}{#6}
\end{scope}}

\newcommand\tileVgreenInside[6]{
\begin{scope}
\tileVgreenBackground{#1}{#2}
\labelinside{#1}{#2}{#3}{#4}{#5}{#6}
\end{scope}}

\newcommand\tilelabelinside[7]{
\begin{scope}
\draw[draw=none,fill=#1] (#2,#3) rectangle (#2+\size,#3+\size);
\labelinside{#2}{#3}{#4}{#5}{#6}{#7}
\end{scope}}
\newcommand\tilelabelinsideV[7]{
\begin{scope}
\draw[draw=none,fill=#1] (#2+\dt,#3) rectangle (#2+\size-\dt,#3+\size);
\labelinside{#2}{#3}{#4}{#5}{#6}{#7}
\end{scope}}
\newcommand\tilelabelinsideVgreen[6]{
\begin{scope}
\tileVgreenBackground{#1}{#2}
\labelinside{#1}{#2}{#3}{#4}{#5}{#6}
\end{scope}}

\newcommand\tilelabelinsideH[7]{
\begin{scope}
\draw[draw=none,fill=#1] (#2,#3+\dt) rectangle (#2+\size,#3+\size-\dt);
\labelinside{#2}{#3}{#4}{#5}{#6}{#7}
\end{scope}}
\newcommand\tilelabelinsideHgreen[6]{
\begin{scope}
\tileHgreenBackground{#1}{#2}
\labelinside{#1}{#2}{#3}{#4}{#5}{#6}
\end{scope}}

\newcommand\tilelabelinsidescope[7]{
\begin{scope}
\tikzstyle{every node}=[font=\scriptsize]
\draw[draw=none,fill=#1] (#2,#3) rectangle (#2+\size,#3+\size);
\labelinside{#2}{#3}{#4}{#5}{#6}{#7}
\end{scope}}
\newcommand\tilelabelinsidescopeH[7]{
\begin{scope}
\tikzstyle{every node}=[font=\scriptsize]
\draw[draw=none,fill=#1] (#2,#3+\dt) rectangle (#2+\size,#3+\size-\dt);
\labelinside{#2}{#3}{#4}{#5}{#6}{#7}
\end{scope}}
\newcommand\tilelabelinsidescopeHpink[6]{
\begin{scope}
\tikzstyle{every node}=[font=\scriptsize]
\tileHpinkBackground{#1}{#2}
\labelinside{#1}{#2}{#3}{#4}{#5}{#6}
\end{scope}}

\newcommand\tilelabelinsidescopeV[7]{
\begin{scope}
\tikzstyle{every node}=[font=\scriptsize]
\draw[draw=none,fill=#1] (#2+\dt,#3) rectangle (#2+\size-\dt,#3+\size);
\labelinside{#2}{#3}{#4}{#5}{#6}{#7}
\end{scope}}

\begin{document}

\title
[Metallic mean Wang tiles I: self-similarity, aperiodicity and minimality]
{Metallic mean Wang tiles I:\\self-similarity, aperiodicity and minimality}

\author[S.~Labb\'e]{S\'ebastien Labb\'e}
\address[S.~Labb\'e]{CNRS, LaBRI, UMR 5800, F-33400 Talence, France}
\email{sebastien.labbe@labri.fr}
\urladdr{http://www.slabbe.org/}

\makeatletter
\@namedef{subjclassname@2020}{\textup{2020} Mathematics Subject Classification}
\makeatother

\keywords{Wang tiles \and aperiodic tiling \and self-similar \and metallic mean}
\subjclass[2020]{Primary 52C23; Secondary 37B51, 37B05, 11B39}

\begin{abstract}
For every positive integer $n$, we introduce a set $\mathcal{T}_n$ made of $(n+3)^2$ Wang tiles (unit squares with labeled edges). We represent a tiling by translates of these tiles as a configuration $\mathbb{Z}^2\to\mathcal{T}_n$. A configuration is valid if the common edge of adjacent tiles has the same label. For every $n\geq1$, we show that the Wang shift $\Omega_n$, defined as the set of valid configurations over the tiles $\mathcal{T}_n$, is self-similar, aperiodic and minimal for the shift action. We say that $\{\Omega_n\}_{n\geq1}$ is a family of metallic mean Wang shifts, since the inflation factor of the self-similarity of $\Omega_n$ is the positive root of the polynomial $x^2-nx-1$. This root is sometimes called the $n$-th metallic mean, and in particular, the golden mean when $n=1$, and the silver mean when $n=2$. When $n=1$, the set of Wang tiles $\mathcal{T}_1$ is equivalent to the Ammann aperiodic set of 16 Wang tiles.
\end{abstract}

\maketitle

\setcounter{tocdepth}{1}
\tableofcontents

\section{Introduction}
\label{sec:intro}

One of the most well-known aperiodic tiling was discovered by Penrose. In its original version, 
four shapes derived
from the regular pentagon can be used to tile the plane and none of the allowed 
tilings are periodic \cite{penrose_role_1974}.
Penrose tilings were soon given an equivalent description in terms of
multigrids or cut and project schemes \cite{MR609465};
see also \cite[\S 10]{MR857454} and \cite[\S 6.2]{MR3136260}.
The aperiodic structure of Penrose tilings is explained by the properties of a specific
irrational number: the positive root $\varphi$ of the polynomial $x^2-x-1$, also known as
the golden ratio or golden mean.
For example, in the kite-and-dart version of the Penrose tilings,
the ratio of kites to darts is equal to the golden ratio \cite{zbMATH03663938}.

Recently, the discovery of an aperiodic monotile \cite{smith_aperiodic_2023}
attracted a lot of attention
\cite{socolar_quasicrystalline_2023,baake_dynamics_2023,akiyama_alternative_2023}.
Smith and coauthors presented a single shape, a 13-edge
polygon called the hat, whose isometric copies tile the plane but never
periodically.
Again the golden ratio appears in tilings by the hat.
In a tiling by isometric copies of the hat, both the hat and its mirror image appear
(up to orientation preserving isometries, that is, translations and rotations).
The frequency of the hat and its mirror image in a tiling are not equal.
The ratio of the most frequent orientation of the hat to the least frequent one
is equal to the fourth power of the golden ratio.\footnote{
    The figure \cite[Fig.~2.11]{smith_aperiodic_2023}
    shows a substitution where 
    the image of a shape $H_7$ contains 5 shapes $H_8$ and 1 shape $H_7$ and
    the image of the shape $H_8$ contains 6 shapes $H_8$ and 1 shape $H_7$.
    Shape $H_7$ contains 6 hats and 1 anti-hats;
    shape $H_8$ contains 7 hats and 1 anti-hats.
    We compute that the Perron--Frobenius dominant right-eigenvector 
    $\left(\begin{smallmatrix} -3\varphi+5\\3\varphi-4 \end{smallmatrix}\right)$
    of the incidence matrix
    $\left(\begin{smallmatrix} 1&1\\ 5&6 \end{smallmatrix}\right)$
    of the substitution is mapped to
    $\left(\begin{smallmatrix} \varphi^4\\1 \end{smallmatrix}\right)$
    by the matrix
    $\left(\begin{smallmatrix} 6&7\\ 1&1 \end{smallmatrix}\right)$.}
Two months later the same authors discovered another aperiodic tile called
Spectre which does not need its mirror image to tile the plane
\cite{smith_chiral_2023}.
Tilings by the Spectre are not all combinatorially equivalent to tilings
by the hat: some are periodic (if the reflected tile is allowed). 
But every tiling by the hat tile is combinatorially equivalent
to some Spectre tiling.

Other examples of aperiodic tilings are related to the golden mean,
including Ammann A2 L-shaped tiles
\cite{MR1156132} (also studied in \cite{akiyama_note_2012,zbMATH07187340});
see Figure~\ref{fig:Ammann-A2-with-bars}.
The golden mean also appears in the description of
tilings generated by the Jeandel--Rao aperiodic set of 11 Wang tiles
\cite{zbMATH07421483}: 
the frequency of the tiles \cite{MR4213162},
the inflation factor of its self-similarity \cite{MR3978536,MR4226493},
and the slopes of its nonexpansive directions \cite{MR4730985}
are all expressed in $\Q(\phi)$.

\begin{figure}[h]
\begin{center}
\begin{tikzpicture}[auto,scale=2]
    \tikzstyle{every node}=[font=\footnotesize]
    \def\lphi{1.61}
    \begin{scope}
        \draw[very thick] (0,0) -- node[swap] {$\varphi$} (\lphi,0) 
                    -- node[swap] {$1$} (\lphi,1) 
                    -- node[swap] {$1/\varphi$} (1,1) 
                    -- node[swap] {$1/\varphi$} (1,\lphi) 
                    -- node[swap] {$1$} (0,\lphi) 
                    -- node[swap] {$\varphi$} (0,0) -- cycle;
        \draw[dashed] (0,\lphi) -- (1,0) -- (\lphi,1);
        \draw[solid] (1,1) -- (0,\lphi-1) -- (\lphi,0);
    \end{scope}
    \begin{scope}[xshift=3cm]
        \draw[very thick] (0,0) -- node {$\varphi$}   (0,\lphi) 
                                -- node {$\varphi$}   (\lphi,\lphi) 
                                -- node {$1/\varphi$} (\lphi  ,1) 
                                -- node {$1$}      (\lphi+1,1) 
                                -- node {$1$}      (\lphi+1,0) 
                                -- node {$\varphi^2$} (0,0) -- cycle;
        \draw[dashed] (0,0) -- (1,\lphi) -- (2,0) -- (\lphi+1,1);
        \draw[solid] (\lphi+1,0) -- (0,1) -- (\lphi,\lphi);
    \end{scope}
\end{tikzpicture}
\end{center}
    \caption{Two shapes belonging to the Ammann A2 family.
    The matching conditions are given by what are called Ammann bars
    appearing as dashed and solid lines in the interior of the tiles
    and which must continue straight across the edges of the tiling.
    This is a reproduction of Figure 10.4.1 from \cite{MR857454}.
    See also Figure 12 from \cite{akiyama_note_2012}.}
    \label{fig:Ammann-A2-with-bars}
\end{figure}
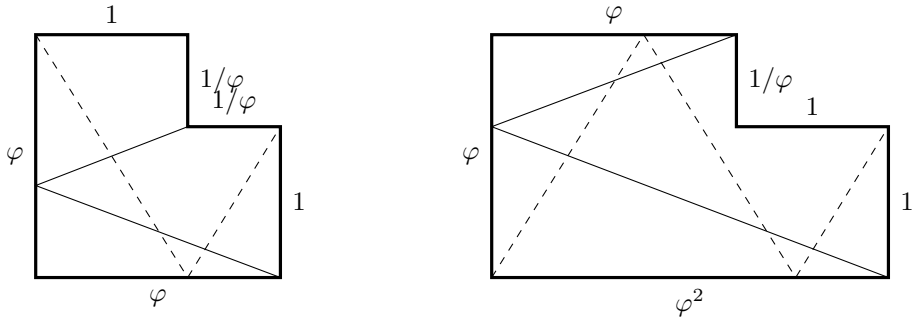

It is then natural to ask whether there are aperiodic tilings out there 
such that the ratios of tile frequencies is not in $\Q(\varphi)$.
It turns out that there are many.
Recall that the first examples of aperiodic tilings provided by Berger
\cite{MR0216954}, simplified by Knuth \cite{MR0286317} and Robinson
\cite{MR0297572}, are
described by substitutions whose inflation factor is an integer (2 in this
case). Many other substitutive and aperiodic planar tilings have integer
inflation factor and are listed in \cite[\S 6.4]{MR3136260}. 
It includes the chair tiling \cite{zbMATH01803721},
the sphinx tiling \cite{MR1452190}, 
the $(1+\varepsilon+\varepsilon^2)$-tiling \cite{zbMATH01064728} and
the Taylor and Socolar-Taylor tilings \cite{MR2834173}.

A lot of substitution tilings with non-integer inflation factor are known.
Various types of planar aperiodic substitution tilings 
with $n$-fold rotational symmetry 
involving cyclotomic numbers
were described in recent years
\cite{zbMATH06426253,zbMATH06624560,zbMATH06676130,sym9020019,zbMATH07663342};
see the sections \cite[\S 1.7]{MR3791847} and \cite[\S 7.3]{MR3136260}.
Examples of algebraic non-Pisot aperiodic tilings were portrayed in \cite[\S 6.5]{MR3136260}.
Moreover, substitution tilings with transcendental inflation factor were
recently proposed in \cite{frettloh_substitution_2022} using compact alphabets.

Closer to golden mean are other algebraic integers, starting with those of
degree two, for which aperiodic tilings exist. 
In Ammann A4 and A5 aperiodic tilings \cite{MR857454}, the ratio of frequency
of the two involved tiles is $\sqrt{2}$ \cite[p.~22]{MR1156132}.
Nowadays these tilings are known as Ammann--Beenker tilings
\cite[\S 6.1]{MR3136260}, 
since their algebraic properties were independently described
in \cite{zbMATH03867362}.
In \cite{MR1156132}, the question whether there exist sets of aperiodic prototiles
associated with irrational numbers other than $\sqrt{2}$ and the golden ratio was mentioned.
But they had ``\textit{no conjecture concerning the characterization
of all numbers that are possible for such ratios}'' of frequencies of tiles.

The inflation factor of Ammann--Beenker substitution tilings is $1+\sqrt{2}$ 
\cite[Prop.~6.2]{MR3136260}. This number is sometimes called the silver mean
because its continued fraction expansion is $[2;2,2,\dots]$
where that of the golden mean is $[1;1,1,\dots]$.
The golden mean and the silver mean belong to a larger family
made of the positive root of the polynomial $x^2-nx-1$, where $n$ is a
positive integer:
\[
    \beta_n=
    \frac{n+\sqrt{n^2+4}}{2}
    = n + \frac{\displaystyle 1}{\displaystyle n
        + \frac{\displaystyle 1}{\displaystyle n
        + \frac{\displaystyle 1}{\displaystyle \dots}}}.
\]
We refer to this root as the \defn{$n^{th}$ metallic mean} \cite{OEIS_metallic_means}.
These numbers were called silver means \cite{zbMATH00051561} and noble means in
\cite[\S~4.4]{MR3136260} (note that noble mean was already used in
\cite[Appendix B, p.~392--394]{zbMATH00051561} for a different meaning).
Observe also that the definition of metallic means from \cite{MR1729917}
is larger, as it contains all positive roots of polynomial
$x^2-px-q$, where $p$ and $q$ are positive integers.
In this contribution, we consider only the metallic means, in the sense 
of de Spinadel, which are algebraic units, that is, $p\geq1$ and $q=1$.

When a tiling space is preserved by a substitution, it is also preserved by
powers of this substitution. Since odd-powers of metallic means are metallic means,
we know substitution tilings for infinitely many other metallic means.
In particular, 
the inflation factor of the third power of the substitution for Penrose tilings
is the $4^{th}$ metallic mean $\beta_1^3 = \beta_4$.
Also, the inflation factor of the third power of the substitution for
Ammann--Beenker tilings is the cube of the silver ratio which 
is the $14^{th}$ metallic mean $\beta_2^3 = \beta_{14}$, etc.
For more information, 
we refer the reader to the OEIS \cite{OEISA352403} where
indices of metallic means that are powers of other metallic means
are listed as sequence \href{https://oeis.org/A352403}{A352403}.

In recent years, new discoveries were made in the theory of quasicrystals related
to metallic mean numbers. A self-similar hexagonal quasicrystal whose inflation
factor is the $3^{rd}$ metallic mean
(also called bronze-mean) was described in \cite{Dotera2017}.
It is given by a substitution rule involving a small and a large equilateral
triangles and a rectangle; see \cite{tiling_encyclopedia_bronze_mean}.
Their construction was further extended to every $(3n)^{th}$ metallic mean in
\cite{Nakakura2019} where $n\geq1$ is a positive integer.

\subsection*{Our contribution}
In this contribution, we introduce a new family of aperiodic tiles using the
oldest known shape for aperiodic tiles: the unit square. 
Unit squares with labeled edges and tilings of the plane by infinitely many
translated copies of them were considered by Wang \cite{wang_proving_1961} with
the condition that
adjacent tiles must share the same label on the common edge. Such tiles are
nowadays called \defn{Wang tiles}.
A set of Wang tiles is \defn{aperiodic} if it admits at least one valid tiling,
and none of them is periodic.
The first known aperiodic set of tiles was 
discovered by Berger \cite{MR0216954}: 
a set of 20426 Wang tiles. Many smaller examples were discovered thereafter, and
we refer the reader to \cite{zbMATH07421483} for an overview of these developments
leading to the discovery of the smallest possible size ($=11$) for an aperiodic
set of Wang tiles.

For every positive integer $n$, we construct a
set $\Tcal_n$ made of $(n+3)^2$ Wang tiles
and we consider the subshift $\Omega_n$ defined as the set of valid
configurations $\Z^2\to\Tcal_n$ over these tiles.
We also say that $\Omega_n$ is a \defn{Wang shift}, because it is a subshift defined
from a set of Wang tiles.
The set $\Tcal_n$ is the disjoint union of 5 sets of tiles:
\begin{itemize}
    \item $n^2$ white tiles, 
    \item $n$ yellow horizontal stripe tiles and
          $n$ yellow vertical stripe tiles, 
    \item $n$ blue horizontal stripe tiles and
          $n$ blue vertical stripe tiles, 
    \item $n+1$ green horizontal overlap tiles and
          $n+1$ green vertical overlap tiles, 
    \item $7$ junction tiles.
\end{itemize}
We observe that
the sum of cardinalities of the five subsets is $n^2+2n+2n+2(n+1)+7=(n+3)^2$.
The sets $\Tcal_n$ of Wang tiles for $n=1,2,3,4,5$ are shown in
Figure~\ref{fig:Tn-for-1-a-5}
and rectangular valid tilings
over the sets $\Tcal_n$ for $n=1,2,3,4$
are shown in
Figure~\ref{fig:T1-17x23},
Figure~\ref{fig:T2-17x23},
Figure~\ref{fig:T3-17x23} and
Figure~\ref{fig:T4-17x23}.

\begin{figure}[h]
\begin{center}
    \begin{tabular}{ccc}
         \includegraphics{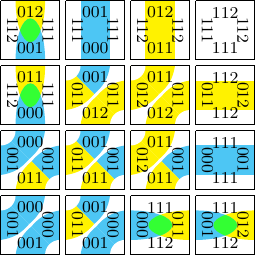}
         &\includegraphics{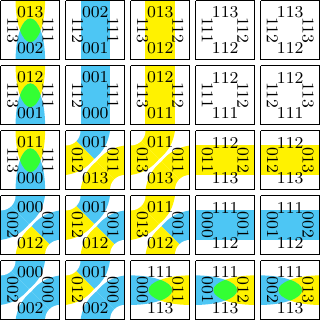}
         &\includegraphics{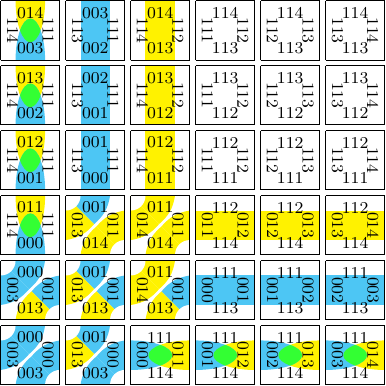}\\
        $\Tcal_1$ & $\Tcal_2$ & $\Tcal_3$\\
    \end{tabular}
    \begin{tabular}{ccc}
         \includegraphics{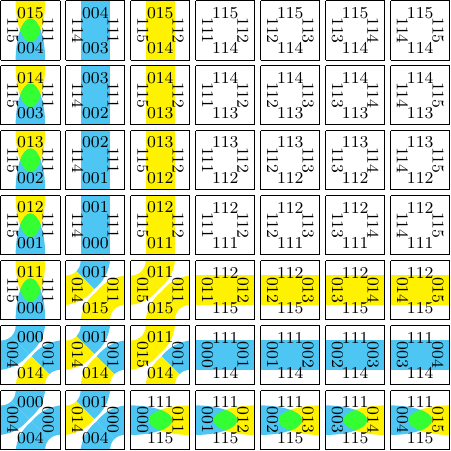}
         &\includegraphics{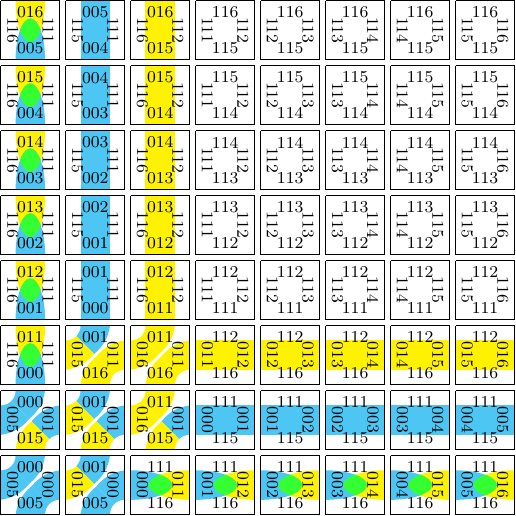}\\
        $\Tcal_4$ & $\Tcal_5$
    \end{tabular}
\end{center}
    \caption{Metallic mean Wang tile sets $\Tcal_n$ for $n=1,2,3,4,5$.}
    \label{fig:Tn-for-1-a-5}
\end{figure}

The family of Wang shift $(\Omega_n)_{n\geq1}$ has
too many nice properties to hold in one article. 
In this first article dedicated to its study, we focus on its substitutive properties.
Its dynamical properties and the consideration of $\Tcal_n$ as the set of
instances of a computer chip will be considered separately in a follow-up
contribution.

The main result of the current contribution is to prove that the Wang shift $\Omega_n$ 
is self-similar for every integer $n\geq1$. The self-similarity is given by a
2-dimensional substitution over an alphabet of size $(n+3)^2$.
The self-similarity is not a bijection, but informally
it is essentially one. This is formalized with the terminology of
recognizability (one-to-one up to a shift) and surjectivity up to a shift.
See Section~\ref{sec:prelim-wang-shifts} for the definition of Wang shifts
and Section~\ref{sec:prelim-2dsubstitutions} for the definition of
2-dimensional substitutions, self-similarity and recognizability.

\newcommand\MainTheoremI{
    For every integer $n\geq1$, the set $\Tcal_n$ containing $(n+3)^2$ Wang
    tiles defines a Wang shift $\Omega_n$ which is self-similar.
    More precisely,
    there exists an expansive and recognizable 2-dimensional substitution
    $\omega_n:\Omega_n\to\Omega_n$ 
    which is onto up to a shift, that is, such that
    $\Omega_n=\overline{\omega_n(\Omega_n)}^\sigma$.
}

\newcommand\MainTheoremII{
    For every integer $n\geq1$, the 2-dimensional substitution
    $\omega_n:\Omega_n\to\Omega_n$ is primitive.
    The Perron--Frobenius dominant eigenvalue of the incidence matrix of
    $\omega_n$ is $\beta_n^2$, the square of the
    $n^{th}$ metallic mean number,
    and the inflation factor of $\omega_n$ is $\beta_n$.
}

\newcommand\MainTheoremIII{
    For every integer $n\geq1$, the Wang shift $\Omega_n$ is minimal
    and is equal to the substitutive subshift $\Omega_n=\Xcal_{\omega_n}$.
}

\newcommand\MainTheoremIV{
    When $n=1$, the set $\Tcal_n$ is equal,
    up to symbol relabeling,
    to the Ammann set of 16 Wang tiles.
}

\begin{maintheorem}\label{thm:similar-to-itself}
    \MainTheoremI
\end{maintheorem}

The proof of Theorem~\ref{thm:similar-to-itself} is the same for every integer $n\geq1$.
Indeed, we show that every configuration in $\Omega_n$ can be decomposed uniquely
into rectangular blocks that we call return blocks.
These return blocks and their right, top, left and bottom labels 
are in bijection with an
extended set $\Tcal_n'\supset\Tcal_n$ of Wang tiles.
Then we show that in the extended Wang shift $\Omega_n'\supseteq\Omega_n$ defined
from the extended set $\Tcal_n'$ of Wang tiles, only the tiles in $\Tcal_n$
appear. Thus, $\Omega_n'\subseteq\Omega_n$. This shows that
$\Omega_n=\Omega_n'$ and that $\Omega_n$ is self-similar.

As a corollary, we deduce that the Wang shift $\Omega_n$ is aperiodic.

\begin{maincorollary}\label{cor:aperiodic}
    For every integer $n\geq1$, the Wang shift $\Omega_n$ is aperiodic.
\end{maincorollary}

Our second result is that the self-similarity is primitive.
As in the one-dimensional case, we say that a two-dimensional substitution
$\omega$ is primitive if there exists $m\in\N$ such that, for every
$a,b\in\Acal$, the letter $b$ occurs in $\omega^m(a)$.

\begin{maintheorem}\label{thm:primitivity}
    \MainTheoremII
\end{maintheorem}

Our third result is that the Wang shift $\Omega_n$ is minimal,
that is, if $X\subseteq\Omega_n$ is a nonempty closed shift-invariant subset,
then $X=\Omega_n$.
Equivalently, every shift orbit is
dense, which implies that every configuration in $\Omega_n$ is uniformly recurrent.
Every small set of aperiodic Wang tiles do not satisfy this property.
For instance, the 
Robinson Wang shift is not minimal \cite{MR2956156},
and neither is the Jeandel--Rao Wang shift \cite{MR4226493}.
The proof of minimality is based on a criterion involving the
patterns of shapes $1\times 2$, $2\times 1$ and $2\times 2$ 
and their images under the substitution;
see Lemma~\ref{lem:criterion-for-minimality}.

\begin{maintheorem}\label{thm:minimal}
    \MainTheoremIII
\end{maintheorem}

In a tiling of the plane by the two shapes shown in Figure~\ref{fig:Ammann-A2-with-bars} 
respecting the matching condition, there appear what are called Ammann bars.
In this case, the slopes of the Ammann bars takes four different values:
two slope values for the dashed Ammann bars and two slope values for the solid Ammann bars.
As explained in \cite[p.594--598]{MR857454},
the solid bars can be regarded as the edges of a new tiling by rhombs and
parallelograms, for which the dashed bars can be regarded as markings on the tiles
specifying the matching conditions. Sixteen parallelogram tiles arise from this
construction which can be recoded as 16 Wang tiles. As we show in 
Theorem~\ref{thm:n=1-equiv-Ammann}, the Ammann 16 Wang tiles is equivalent to $\Tcal_1$,
the first member of the family $\Tcal_n$ when $n=1$.

\begin{maintheorem}\label{thm:n=1-equiv-Ammann}
    \MainTheoremIV
\end{maintheorem}

Thus, the family $(\Tcal_n)_{n\geq1}$ can be considered as an extension 
of the Ammann set of Wang tiles to the metallic mean numbers.

\subsection*{Structure of the article}

In Section~\ref{sec:prelim-wang-shifts}, we present
preliminaries on dynamical systems, subshifts and Wang shifts.
In Section~\ref{sec:prelim-2dsubstitutions},
we recall definitions of 2-dimensional substitutions.
In Section~\ref{sec:family-Tcaln},
we introduce two Wang shifts $\Omega_n\subseteq\Omega_n'$ defined
by the sets $\Tcal_n\subseteq\Tcal_n'$ of Wang tiles.
In Section~\ref{sec:substitution},
we define two substitutions $\omega_n':\Omega_n'\to\Omega_n'$
and $\omega_n:\Omega_n\to\Omega_n$.
In Section~\ref{sec:desubstitution},
we describe the return blocks
in the Wang shifts $\Omega_n$ and $\Omega_n'$
and we prove that every configuration in
the Wang shift $\Omega_n$ can be desubstituted into a configuration from
$\Omega_n'$.
In Section~\ref{sec:Omegan'subseteqOmegan},
we prove that tiles in $\Tcal_n'\setminus\Tcal_n$ do not appear
in configurations of $\Omega_n'$. Thus,
$\Omega_n'\subseteq\Omega_n$.
Observe that
Section~\ref{sec:Omegan'subseteqOmegan}
depends on the results from
Section~\ref{sec:substitution} and
Section~\ref{sec:desubstitution}.
In Section~\ref{sec:self-similar-aperiodic},
we prove that $\Omega_n$ is self-similar and aperiodic.
In Section~\ref{sec:primitive}, we prove that
the self-similarity is primitive.
In Section~\ref{sec:minimal}, we prove that
$\Omega_n$ is minimal.
In Section~\ref{sec:question}, we state some questions raised by the current work.
The article finishes with two appendices.
Section~\ref{sec:Appendix-substitutions-1ot5} (Appendix A) gathers pictures of
the substitutions $\omega_n$ for $1\leq n\leq5$.
In Section~\ref{sec:Appendix-self-similarity-Omega2} (Appendix B),
we prove the self-similarity of $\Omega_n$ when $n=2$ using computer explorations.

\subsection*{Batteries included}
All results proved in this article are proved by hand
except the proof of Lemma~\ref{lem:minimality-when-n=1} 
and the computations performed in Appendix B
which are based on the open-source
mathematical software SageMath \cite{sagemathv10.5} and the
optional package \texttt{slabbe} \cite{labbe_slabbe_0_7_7_2024}.
All SageMath input/output blocks in this article were created using
the \texttt{sageexample} environment with 
SageTeX version \texttt{2021/10/16 v3.6}
and with the following software versions:
\begin{sagecommandline}
sage: version()
'SageMath version ..., Release Date: ...'
sage: import importlib.metadata
sage: importlib.metadata.version("slabbe")
'...'
\end{sagecommandline}
The fact that these software are open-source means that anyone
is free to use, reproduce, verify, adapt for their own needs all of the
computations performed therein according to the
GNU General Public License (version 2, 1991, \url{http://www.gnu.org/licenses/gpl.html}).

The contents of all of the \texttt{sageexample} environments from the tex
source are gathered in the file
\texttt{demos/arXiv\_2312\_03652\_doctest.sage}
autogenerated by SageTeX when running \texttt{pdflatex}.
This file is included in the \texttt{slabbe} package and available at
\url{https://gitlab.com/seblabbe/slabbe/}.
It allows to make sure that future releases of the package do not break the
code included in this article.
It is possible to reproduce all computations present in this article and 
check that all outputs are correct, by \emph{doctesting} this file, that is,
by running the command
\texttt{sage -t demos/arXiv\_2312\_03652\_doctest.sage}.
It should output \texttt{All tests passed!} and 
\texttt{[67 tests, 31.02s wall]} (most probably with a different timing).

\subsection*{Acknowledgments}
The author is thankful to the reviewers for their careful reading and remarks
which led, in particular, to an improved and more formal proof of the
self-similarity of $\Omega_n$.
The author also wants to thank Dirk Frettl\"{o}h for making him aware of other
existing aperiodic substitutive tilings involving metallic mean numbers
including the bronze mean \cite{Dotera2017}.
This work was partly funded from France's Agence Nationale de la Recherche
(ANR) project CODYS (ANR-18-CE40-0007) and project IZES (ANR-22-CE40-0011).

\begin{figure}%
\begin{center}
    \includegraphics{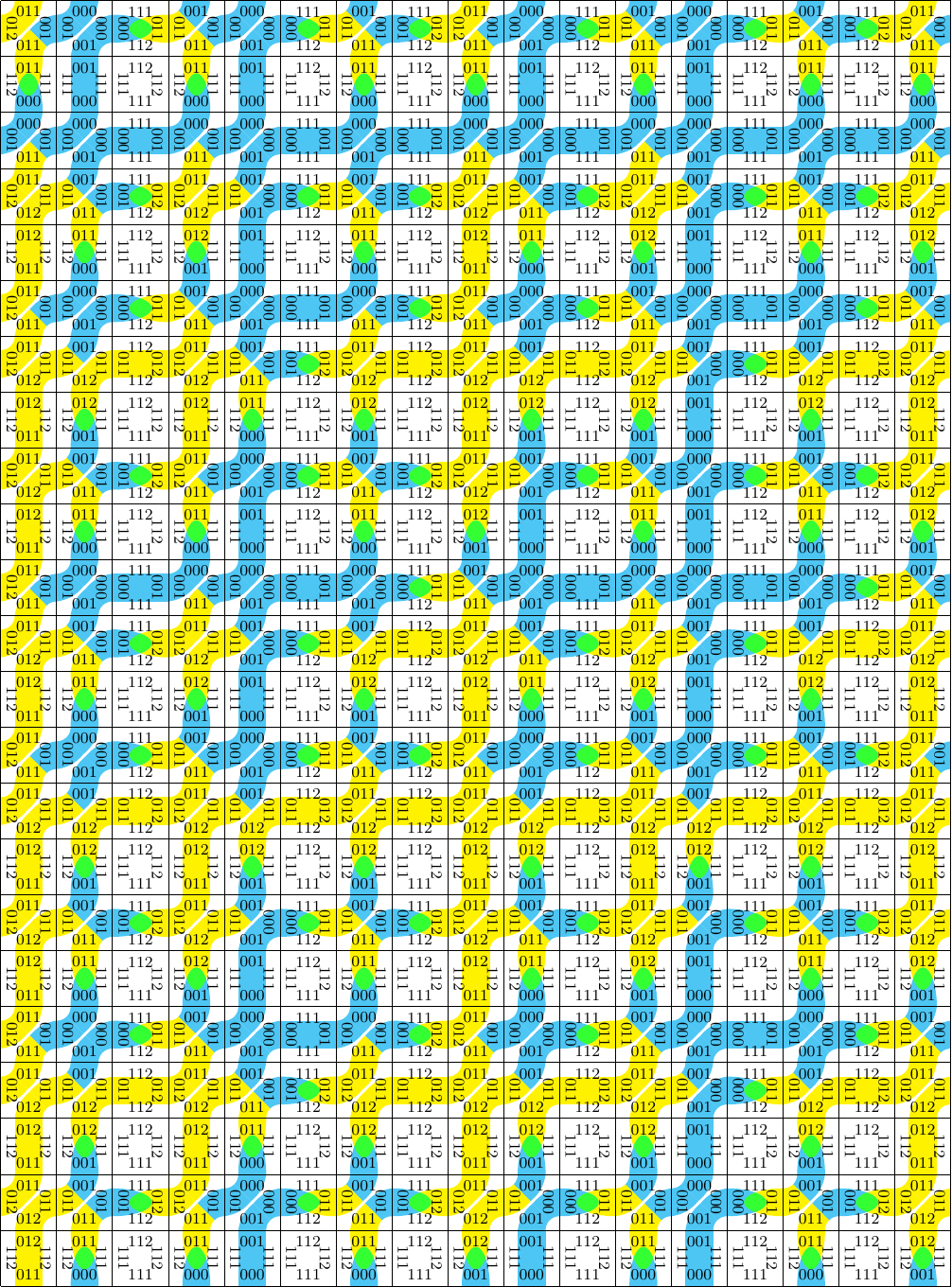}
\end{center}
    \caption{A valid $17\times 23$ pattern with Wang tile set $\Tcal_1$.}
    \label{fig:T1-17x23}
\end{figure}

\begin{figure}%
\begin{center}
    \includegraphics{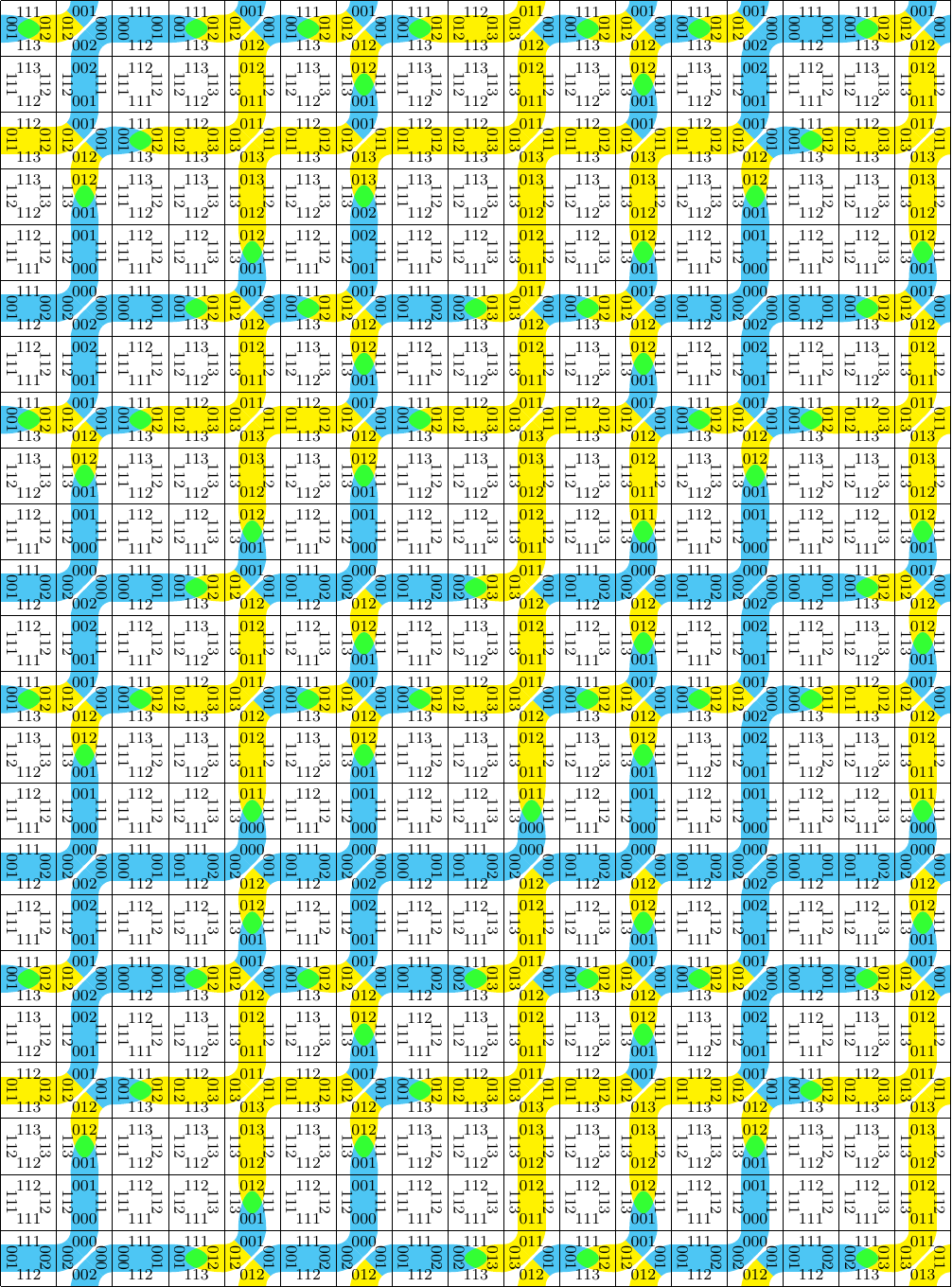}
\end{center}
    \caption{A valid $17\times 23$ pattern with Wang tile set $\Tcal_2$.}
    \label{fig:T2-17x23}
\end{figure}

\begin{figure}%
\begin{center}
    \includegraphics{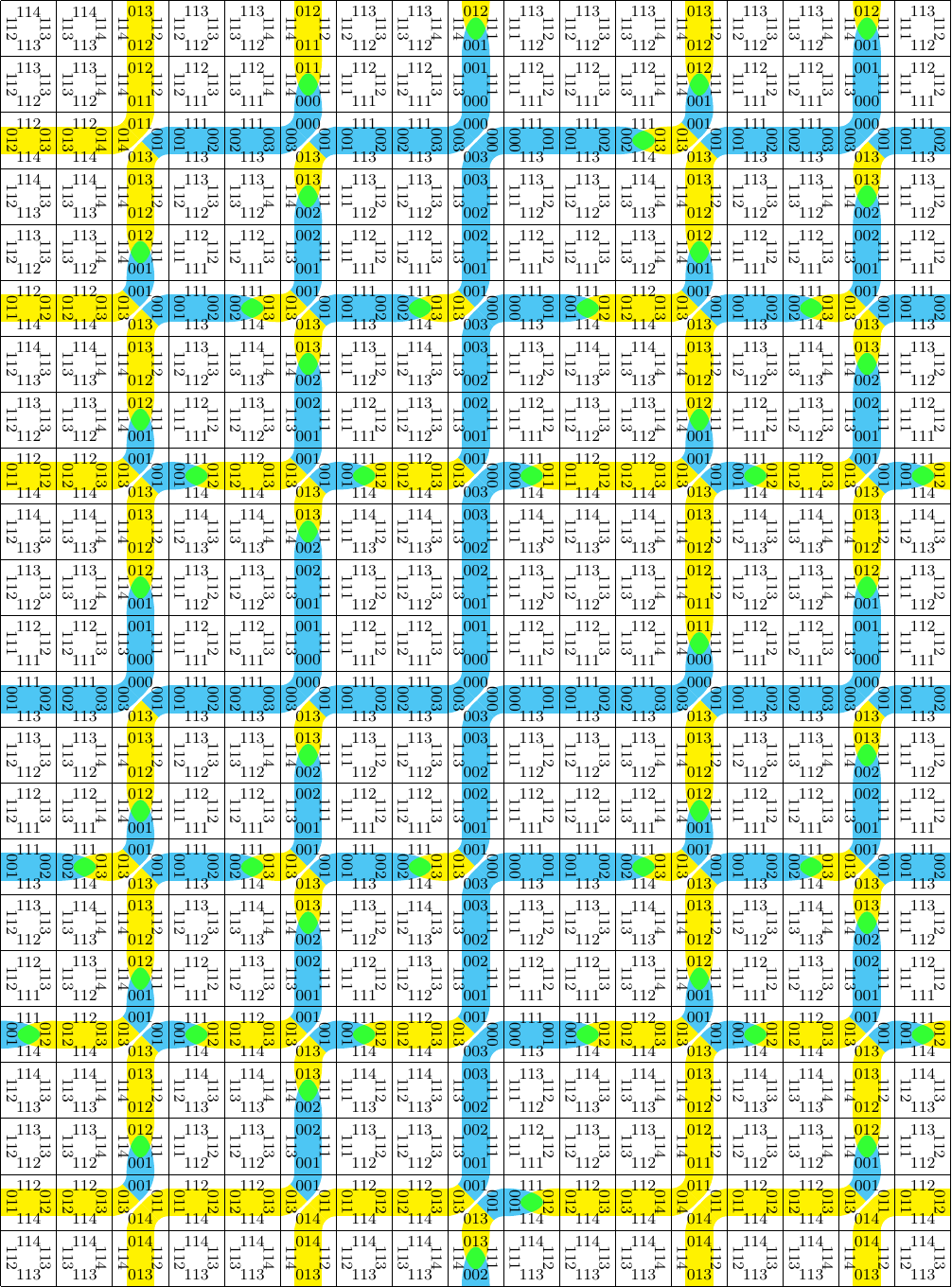}
\end{center}
    \caption{A valid $17\times 23$ pattern with Wang tile set $\Tcal_3$.}
    \label{fig:T3-17x23}
\end{figure}

\begin{figure}%
\begin{center}
    \includegraphics{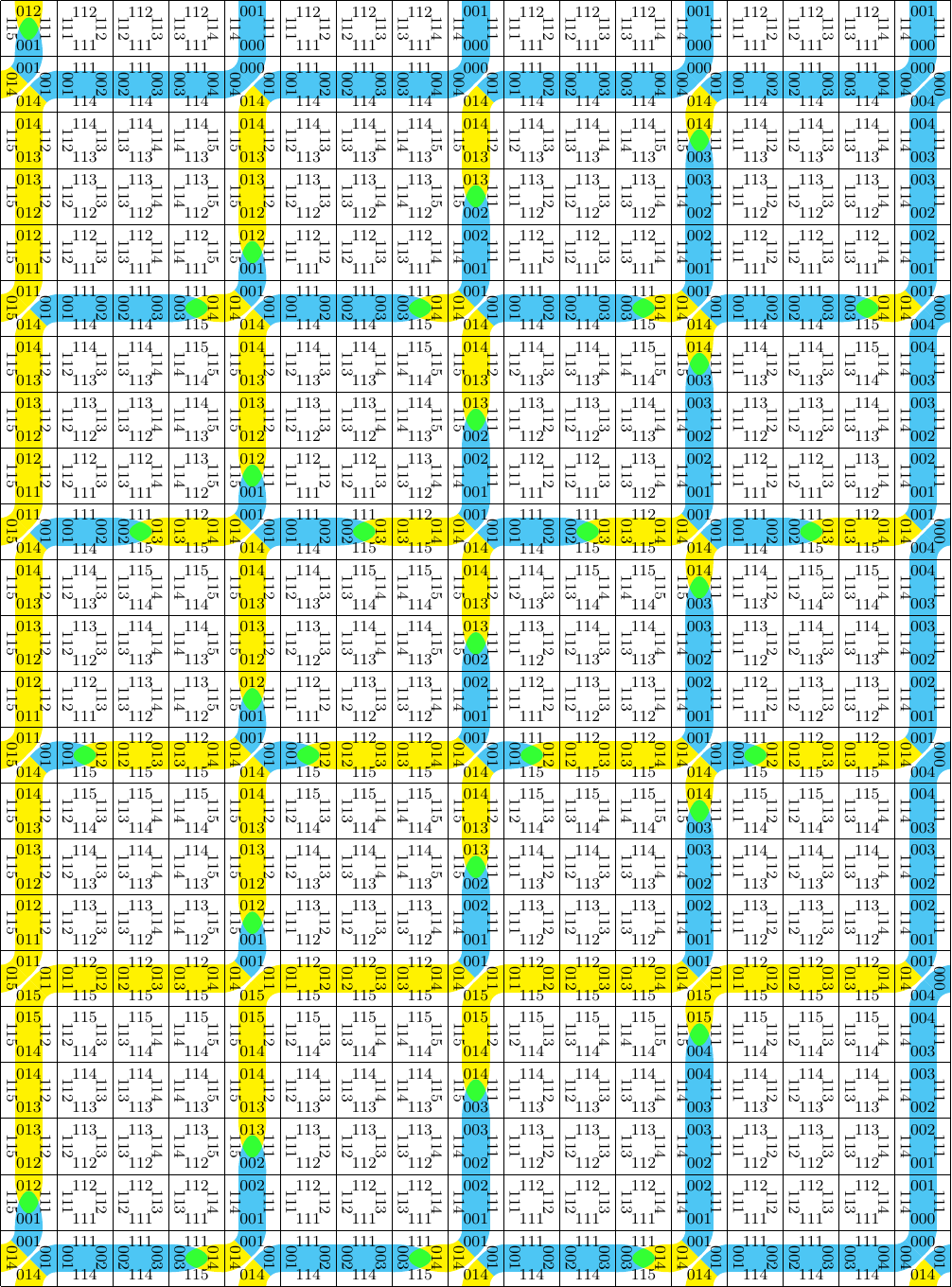}
\end{center}
    \caption{A valid $17\times 23$ pattern with Wang tile set $\Tcal_4$.}
    \label{fig:T4-17x23}
\end{figure}

\section{Preliminaries on Wang shifts}
\label{sec:prelim-wang-shifts}

This section follows the preliminary section of the chapter
\cite{labbe_three_2020}.

\subsection{Topological dynamical systems}

Most of the notions introduced here can be found in \cite{MR648108}.
A \defn{dynamical system} is
a triple $(X,G,T)$, where $X$ is a topological space, $G$ is a topological
group and $T$ is a continuous function $G\times X\to X$ defining a left action
of $G$ on $X$:
if $x\in X$, $e$ is the identity element of $G$ and $g,h\in G$, then using
additive notation for the operation in $G$ we have $T(e,x)=x$
and $T(g+h,x)=T(g,T(h,x))$.
In other words, if one denotes the transformation $x\mapsto T(g,x)$
by $T^g$, then $T^{g+h}=T^g T^h$.
In this work, we consider the Abelian group $G=\Z\times\Z$.

If $Y\subset X$, let $\overline{Y}$ denote the topological closure of $Y$ and
let $\overline{Y}^T:=\cup_{g\in G}T^g(Y)$ denote the $T$-closure of $Y$.
A subset $Y\subset X$ is \defn{$T$-invariant} if $\overline{Y}^T=Y$.
A dynamical system $(X,G,T)$ is called \defn{minimal} if $X$ does
not contain any nonempty, proper, closed $T$-invariant subset.
The left action of $G$ on $X$ is \defn{free}
if $g=e$ whenever there exists $x\in X$ such that $T^g(x)=x$.

Let $(X,G,T)$ and $(Y,G,S)$ be two dynamical systems with
the same topological group $G$.
A \defn{homomorphism} $\theta:(X,G,T)\to(Y,G,S)$ is a continuous
function $\theta:X\to Y$ satisfying the commuting property
that $S^g\circ\theta=\theta\circ T^g$ for every $g\in G$.
A homomorphism $\theta:(X,G,T)\to(Y,G,S)$ is called an \defn{embedding}
if it is one-to-one, a \defn{factor map} if it is onto, and a \defn{topological
conjugacy} if it is both one-to-one and onto and its inverse map is continuous.
If $\theta:(X,G,T)\to(Y,G,S)$ is a factor map,
then $(Y,G,S)$ is called a \defn{factor} of $(X,G,T)$
and $(X,G,T)$ is called an \defn{extension} of $(Y,G,S)$.
Two dynamical systems are \defn{topologically conjugate} if there is a
topological conjugacy between them.

\subsection{Subshifts and shifts of finite type}\label{sec:subshift-SFT}

In this section, we introduce multidimensional subshifts,
a particular type of dynamical systems 
\cite[\S 13.10]{MR1369092},
\cite{MR1861953,MR2078846,MR3525488}.
Let $\Acal$ be a finite set, $d\geq 1$, and let $\Acal^{\Z^d}$ be the set of all maps
$x:\Z^d\to\Acal$, equipped with the compact product topology. 
An element $x\in\Acal^{\Z^d}$ is called \defn{configuration}
and we write it as $x=(x_\bm)=(x_\bm:\bm\in\Z^d)$,
where $x_\bm\in\Acal$ denotes the value of $x$ at $\bm$. 
The topology on $\Acal^{\Z^d}$ is compatible with the metric defined for all
configurations $x,x'\in\Acal^{\Z^d}$ by $\dist(x,x')=2^{-\min\left\{\Vert\bn\Vert\,:\,
x_\bn\neq x'_\bn\right\}}$
where $\Vert\bn\Vert = |n_1| + \dots + |n_d|$.
The \defn{shift action} $\sigma:\bn\mapsto
\sigma^\bn$ of the additive group $\Z^d$ on $\Acal^{\Z^d}$ is defined by
\begin{equation}\label{eq:shift-action}
    (\sigma^\bn(x))_\bm = x_{\bm+\bn}
\end{equation}
for every $x=(x_\bm)\in\Acal^{\Z^d}$ and $\bn\in\Z^d$. 
If $X\subset \Acal^{\Z^d}$,
let $\overline{X}$ denote the topological closure of $X$
and let $\shiftclosure{X}:=\{\sigma^\bn(x)\mid x\in X, \bn\in\Z^d\}$
denote the shift-closure of $X$.
A subset $X\subset
\Acal^{\Z^d}$ is \defn{shift-invariant} if 
$\shiftclosure{X}=X$. A closed, shift-invariant subset
$X\subset\Acal^{\Z^d}$ is a \defn{subshift}. 
If $X\subset\Acal^{\Z^d}$ is a subshift we write
$\sigma=\sigma^X$ for the restriction of the shift action
\eqref{eq:shift-action} to $X$. 
When $X$ is a subshift,
the triple $(X,\Z^d,\sigma)$ is a dynamical system
and the notions presented in the previous section hold.

A configuration $x\in X$ is \defn{periodic} if there is a nonzero vector
$\bn\in\Z^d\setminus\{\zero\}$ such that $x=\sigma^\bn(x)$,
and otherwise it is \defn{nonperiodic}.
We say that a nonempty subshift $X$ is \defn{aperiodic}
if the shift action $\sigma$ on $X$ is free.

For any subset $S\subset\Z^d$ let $\pi_S:\Acal^{\Z^d}\to\Acal^S$ denote the
projection map which restricts every $x\in\Acal^{\Z^d}$ to $S$. 
A \defn{pattern} is a function $p\in\Acal^S$ for some finite subset
$S\subset\Z^d$.
To every pattern $p\in\Acal^S$ corresponds
a subset $\pi_S^{-1}(p)\subset\Acal^{\Z^d}$ called \defn{cylinder}.
A nonempty set $X\subset\Acal^{\Z^d}$ is a
\defn{subshift} if and only if there exists a set $\Fcal$
of \defn{forbidden} patterns such that
\begin{equation}\label{eq:SFT}
    X = \{x\in\Acal^{\Z^d} \mid \pi_S\circ\sigma^\bn(x)\notin\Fcal
    \text{ for every } \bn\in\Z^d \text{ and } S\subset\Z^d\},
\end{equation}
see \cite[Prop.~9.2.4]{MR3525488}.
A subshift $X\subset\Acal^{\Z^d}$ is a 
\defn{subshift of finite type} (SFT) if there exists a finite set $\Fcal$ such that \eqref{eq:SFT} holds.
In this article, we consider shifts of finite type on $\Z\times\Z$, that is, the case
$d=2$.

\subsection{Wang shifts}

A \defn{Wang tile} 
is a tuple of four colors $(a,b,c,d)\in I\times J\times
I\times J$
where $I$
is a finite set of vertical colors
and $J$
is a finite set of horizontal colors; see
\cite{wang_proving_1961,MR0297572}.
A Wang tile is represented as a unit square with colored edges:
\begin{center}
    \raisebox{-9.5mm}{
    \begin{tikzpicture}[auto]
    \tile{white}{0}{0}{a}{b}{c}{d}
    \end{tikzpicture}}
\end{center}
For each Wang tile $\tau=(a,b,c,d)$, let
$\scright(\tau)=a$,
$\sctop(\tau)=b$,
$\scleft(\tau)=c$,
$\scbottom(\tau)=d$
denote respectively the \defn{colors} or \defn{labels} of the right, top, left
and bottom edges of $\tau$.

\begin{figure}[h]
\begin{center}
    \includegraphics{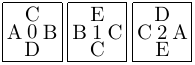}
\end{center}
    \caption{The set of 3 Wang tiles introduced
    in \cite{wang_proving_1961} using letters $\{A,B,C,D,E\}$ instead of
    numbers from the set $\{1,2,3,4,5\}$ for labeling the edges.
    Each tile is identified uniquely by an index from the
    set $\{0,1,2\}$ written at the center each tile.}
    \label{fig:wang-three-tiles}
\end{figure}

Let $\Tcal=\{t_0,\dots,t_{m-1}\}$ be a set of Wang tiles as the one shown in Figure~\ref{fig:wang-three-tiles}.
A configuration $x:\Z^2\to\{0,\dots,m-1\}$ is \defn{valid} with respect to $\Tcal$ if
it assigns a tile in $\Tcal$ to each position of $\Z^2$, so that contiguous edges
of adjacent tiles have the same color, that is,
\begin{align}
    \scright(t_{x(\bn)})&=\scleft(t_{x(\bn+\be_1)})\label{eq:validwangtiling1}\\
    \sctop(t_{x(\bn)})&=\scbottom(t_{x(\bn+\be_2)})\label{eq:validwangtiling2}
\end{align}
for every $\bn\in\Z^2$ where $\be_1=(1,0)$ and $\be_2=(0,1)$.
A finite pattern which is valid with respect to $\Ucal$ is shown in 
Figure~\ref{fig:wang-three-tiles-3x3-tiling}.

\begin{figure}[h]
\begin{center}
\begin{tikzpicture}
    \node (A) at (0,0) {$\arraycolsep=1.8pt
                        \input{SAGEOUTPUT/Wang_1961_three_tiles_tiling_config.tex}
                        $};
    \node (B) at (3,0) {\includegraphics{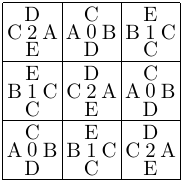}};
    \draw[|->] (A) to (B);
\end{tikzpicture}
\end{center}
    \caption{A finite $3\times 3$ pattern on the left is valid with respect to the Wang tiles
    since it respects Equations~\eqref{eq:validwangtiling1}
    and~\eqref{eq:validwangtiling2}. Validity can be verified on the tiling shown on
    the right.}
    \label{fig:wang-three-tiles-3x3-tiling}
\end{figure}

Let $\Omega_\Tcal\subset\{0,\dots,m-1\}^{\Z^2}$ denote the set of all valid 
configurations with respect to $\Tcal$,
called the \defn{Wang shift} of $\Tcal$. 
To a configuration $x\in\Omega_\Tcal$ corresponds a tiling of the plane $\R^2$ by
the tiles $\Tcal$ where the unit square Wang tile $t_{x(\bn)}$ is placed at position $\bn$ for every
$\bn\in\Z^2$, as in Figure~\ref{fig:wang-three-tiles-3x3-tiling}.
Together with the shift action $\sigma$ of $\Z^2$,
$\Omega_\Tcal$ is a SFT of the form \eqref{eq:SFT}
since there exists a finite set of
forbidden patterns made of all horizontal and vertical dominoes of two tiles
that do not share an edge of the same color.

A configuration $x\in\Omega_\Tcal$ is \defn{periodic} if there exists
$\bn\in\Z^2\setminus\{0\}$ such that $x=\sigma^\bn(x)$.
A set $\Tcal$ of Wang tiles is \defn{periodic} if there exists a periodic configuration
$x\in\Omega_\Tcal$. 
Originally, Wang thought that every set $\Tcal$ of Wang tiles is periodic 
as soon as $\Omega_\Tcal$ is nonempty \cite{wang_proving_1961}.
Wang noticed that if this statement were true, 
it would imply the existence of an algorithm 
solving the \emph{domino problem}, that is, taking as input a set of Wang tiles
and returning \textit{yes} or \textit{no} whether there exists a valid
configuration with these tiles. 
Berger, a student of Wang, later proved that the domino problem is undecidable
and he also provided a first example of an aperiodic set of Wang tiles
\cite{MR0216954}.
A set $\Tcal$ of Wang tiles is \defn{aperiodic} if
the Wang shift $\Omega_\Tcal$ is a nonempty aperiodic subshift.

\subsection{Directional determinism}
A set $\Tcal$ of Wang tiles is called \defn{SW-deterministic} if there do not exist two
different tiles in $\Tcal$ that would have the same colors on their bottom and left edges,
respectively \cite{MR1692474}. In other words, for all colors $C_1$ and $C_2$
there exists at most one tile in $\Tcal$ whose bottom and left edges have
colors $C_1$ and $C_2$, respectively.

Let $S=\{a_1,a_1+1,\dots,b_1\}\times\{a_2,a_2+1,\dots,b_2\}$
be a rectangular support where $a_1,b_1,a_2,b_2$ are integers such that 
$a_1\leq b_1$ and $a_2\leq b_2$.
Let $p:S\to\Tcal$ be a valid rectangular pattern over the tiles $\Tcal$.
We say that the 
\defn{bottom labels of $p$} 
and \defn{top labels of $p$} 
are, respectively, the sequences
\begin{align*}
&\scbottom(p_{a_1,a_2}),\scbottom(p_{a_1+1,a_2}),\dots,\scbottom(p_{b_1,a_2})\text{ and }\\
&\sctop(p_{a_1,b_2}),\sctop(p_{a_1+1,b_2}),\dots,\sctop(p_{b_1,b_2})
\end{align*}
read on the pattern from left to right.
Also, we say that the 
\defn{left labels of $p$} 
and \defn{right labels of $p$} 
are, respectively, the sequences
\begin{align*}
&\scleft(p_{a_1,a_2}),\scleft(p_{a_1,a_2+1}),\dots,\scleft(p_{a_1,b_2})\text{ and }\\
&\scright(p_{b_1,a_2}),\scright(p_{b_1,a_2+1}),\dots,\scright(p_{b_1,b_2})
\end{align*}
read on the pattern from bottom to top.

As shown in the next lemma, the local definition of SW-deterministic sets of Wang
tiles extends into a wider property on rectangular patterns.

\begin{lemma}\label{lem:unicity-rect-pattern}
    Let $\Tcal$ be a SW-deterministic set of Wang tiles.
    If $p$ and $q$ are two rectangular valid patterns 
    with the same shape, the same sequence of
    bottom labels and the same sequence of left labels, then $p=q$.
\end{lemma}

\begin{proof}
    By contradiction,
    suppose that there are two distinct rectangular patterns $p$ and $q$ whose
    sequence of bottom labels is $X$ 
    and sequence of left labels is $Y$. 
    Since $p$ and $q$ are distinct,
    there exists a position $k\in\N^2$ such that $p_k\neq q_k$.
    Consider such a position in the support of $p$ and $q$ which minimizes the norm $\Vert k\Vert_1$.
    Since the position is minimal, 
    every tile at position smaller in norm is the same in both patterns.
    In particular, it implies that
    $\scleft(p_k)=\scleft(q_k)$
    and
    $\scbottom(p_k)=\scbottom(q_k)$.
    The set of Wang tile $\Tcal_n$ is SW-deterministic.
    This implies that
    $\sctop(p_k)=\sctop(q_k)$ and
    $\scright(p_k)=\scright(q_k)$.
    Since the four labels of the Wang tiles are the same, we must have
    $p_k=q_k$, a contradiction.
    We conclude the uniqueness of the rectangular pattern.
\end{proof}

\defn{NW-}, \defn{NE-} and \defn{SE-deterministic} sets of Wang tiles are
defined analogously. Recall that it was shown in \cite{MR1692474} that there
exist aperiodic tile sets that are deterministic in all four directions
simultaneously.

\section{Preliminaries on 2-dimensional substitutions}
\label{sec:prelim-2dsubstitutions}

Rectangular two-dimensional substitutions and their symbolic dynamical systems were
considered in \cite{MR1014984}.
For a certain class of two-dimensional substitution systems, 
it was shown how to construct a set of Wang tiles such that the associated Wang
shift is an almost everywhere one-to-one extension of the substitution system
\cite[Theorem~4.5]{MR1014984}.
This result was generalized later for geometrical substitutions over polygonal tiles
\cite{MR1609510}.

In this section, we introduce $2$-dimensional substitutions.
Our definition and the one presented in \cite{MR1014984} are incomparable.
On the one hand, we restrict to the deterministic case (every letter has a
unique image).
On the other hand, we extend to different alphabets $\Acal$ and $\Bcal$ for the
domain and codomain.
The section follows the preliminary section of the chapter \cite{labbe_three_2020}.

\subsection{$d$-dimensional word}

We denote by
$\{\be_k|1\leq k\leq d\}$ the canonical
basis of $\Z^d$ where $d\geq1$ is an integer.
If $i\leq j$ are integers, then $\llbracket i, j\rrbracket$ denotes the
interval of integers $\{i, i+1, \dots, j\}$.
Let $\bn=(n_1,\dots,n_d)\in\N^d$ and $\Acal$ be an alphabet.
We denote by $\Acal^{\bn}$ the set of functions
\begin{equation*}
    u:
\llbracket 0,n_1-1\rrbracket
\times
\cdots
\times
\llbracket 0,n_d-1\rrbracket
\to\Acal.
\end{equation*}
An element $u\in\Acal^\bn$ is called a
\defn{$d$-dimensional word} of \defn{size} $\bn=(n_1,\dots,n_d)\in\N^d$
on the alphabet~$\Acal$.
We use the notation $\scsize(u)=\bn$ when necessary.
The set of all finite $d$-dimensional words is 
$\Acal^{*^d}=\bigcup_{\bn\in\N^d} \Acal^\bn$.
A $d$-dimensional word of size $\be_k+\sum_{i=1}^d\be_i$ is called a
\defn{domino in the direction $\be_k$}.
When the context is clear, we write $\Acal$ instead of $\Acal^{(1,\dots,1)}$.
When $d=2$, we represent a $d$-dimensional word $u$ of size $(n_1,n_2)$ as a
matrix with Cartesian coordinates:
\begin{equation*}
    u=
    \left(\begin{array}{ccc}
        u_{0,n_2-1} &\dots   & u_{n_1-1,n_2-1} \\
        \dots   &\dots   & \dots \\
        u_{0,0} &\dots   & u_{n_1-1,0}
    \end{array}\right).
\end{equation*}
Let $\bn,\bm\in\N^d$ and $u\in\Acal^\bn$ and $v\in\Acal^\bm$.
If there exists an index $i$ such that 
$n_j=m_j$ for all $j\in\{1,\dots,d\}\setminus\{i\}$,
then the \defn{concatenation} of $u$ and $v$ in the direction $\be_i$ 
is defined: it is
the 
$d$-dimensional word $u\odot^i v$ of size $(n_1,\dots,n_{i-1},n_i+m_i,n_{i+1},\dots,n_d)\in\N^d$
given as
\begin{equation*}
    (u\odot^i v) (\ba) = 
\begin{cases}
    u(\ba)          & \text{if}\quad 0 \leq a_i < n_i,\\
    v(\ba-n_i\be_i) & \text{if}\quad n_i \leq a_i < n_i+m_i.
\end{cases}
\end{equation*}
The notation $u\odot^i v$ was used in \cite{MR2579856}.

The following equation illustrates the concatenation of $2$-dimensional words 
in the direction $\be_2$:
\[
    \arraycolsep=2.5pt
\left(\begin{array}{ccccc}
4 & 5 \\
10 & 5
\end{array}\right)
\odot^2
\left(\begin{array}{ccccc}
3 & 10 \\
9 & 9 \\
0 & 0 \\
\end{array}\right)
    =
\left(\begin{array}{ccccc}
3 & 10 \\
9 & 9 \\
0 & 0 \\
4 & 5 \\
10 & 5
\end{array}\right)
\]
and in the direction $\be_1$:
\[
    \arraycolsep=2.5pt
\left(\begin{array}{ccccc}
2 & 8 & 7  \\
7 & 3 & 9 \\
1 & 1 & 0 \\
6 & 6 & 7 \\
7 & 4 & 3 
\end{array}\right)
\odot^1
\left(\begin{array}{ccccc}
3 & 10 \\
9 & 9 \\
0 & 0 \\
4 & 5 \\
10 & 5
\end{array}\right)
    =
\left(\begin{array}{ccccc}
2 & 8 & 7 & 3 & 10 \\
7 & 3 & 9 & 9 & 9 \\
1 & 1 & 0 & 0 & 0 \\
6 & 6 & 7 & 4 & 5 \\
7 & 4 & 3 & 10 & 5
\end{array}\right).
\]

Let $\bn,\bm\in\N^d$ and $u\in\Acal^\bn$ and $v\in\Acal^\bm$.
We say that $u$ \defn{occurs in $v$ at position} $\bp\in\N^d$ if
$v$ is large enough, i.e., $\bm-\bp-\bn\in\N^d$ and
\[
    v(\ba+\bp) = u(\ba)
\]
for all $\ba=(a_1,\dots,a_d)\in\N^d$ such that 
$0\leq a_i<n_i$ with $1\leq i\leq d$.
If $u$ occurs in $v$ at some position, then we say that $u$ is a
$d$-dimensional \defn{subword} or \defn{factor} of $v$.

\subsection{$d$-dimensional language}

A subset $L\subseteq\Acal^{*^d}$ is called a $d$-dimensional \defn{language}. 
A language $L\subseteq\Acal^{*^d}$ is called \defn{factorial}
if for every $v\in L$ and every $d$-dimensional subword $u$ of $v$, we have $u\in L$.
All languages considered in this contribution are factorial.
Given a configuration $x\in\Acal^{\Z^d}$, the \defn{language} $\Lcal(x)$ defined by $x$ is
\begin{equation*}
    \Lcal(x) = \{u\in\Acal^{*^d} \mid u\text{ is a $d$-dimensional subword of } x\}.
\end{equation*}
The \defn{language} of a subshift $X\subseteq\Acal^{\Z^d}$ is
    $\Lcal_X = \cup_{x\in X} \Lcal(x)$.
Conversely, given a factorial language $L\subseteq\Acal^{*^d}$ we define the subshift
\begin{equation*}
    \Xcal_L = \{x\in\Acal^{\Z^d}\mid \Lcal(x)\subseteq L\}.
\end{equation*}
A $d$-dimensional subword $u\in\Acal^{*^d}$ is \defn{legal} (or \defn{allowed}) in a 
subshift $X\subset\Acal^{\Z^d}$ if $u\in\Lcal_X$
and it is \defn{illegal} in $X$ if $u\notin\Lcal_X$
\cite{MR3136260}.
A language $L\subseteq\Acal^{*^d}$ is \defn{illegal} in a 
subshift
$X\subset\Acal^{\Z^d}$ if $L\cap\Lcal_X=\varnothing$.

\subsection{$d$-dimensional morphisms}

Let $\Acal$ and $\Bcal$ be two alphabets.
Let $L\subseteq\Acal^{*^d}$ be a factorial language.
A function $\omega:L\to\Bcal^{*^d}$ is a \defn{$d$-dimensional
morphism} if for every
$i$ with $1\leq i\leq d$,
and every $u,v\in L$ such that 
$u\odot^i v$ is defined
and
$u\odot^i v\in L$,
we have
that the concatenation $\omega(u)\odot^i \omega(v)$
in direction $\be_i$ is defined and
\begin{equation*}
    \omega(u\odot^i v) = \omega(u)\odot^i \omega(v).
\end{equation*}
Note that the left-hand side of the equation is defined since
$u\odot^i v$ belongs to the domain of $\omega$.
A $d$-dimensional morphism $L\to\Bcal^{*^d}$ is thus completely defined from the
image of the letters in $\Acal$, so we sometimes denote
a $d$-dimensional morphism as a rule $\Acal\to\Bcal^{*^d}$ 
when the language $L$ is unspecified.

As noticed by \cite[p.144]{MR1014984}, 
the images under the morphism of any two letters appearing in the same row
of a word from $L$ have the same height. Symmetrically, 
the images under the morphism of any two letters appearing in the same column
of a word from $L$ have the same width.

Let $L\subseteq\Acal^{*^d}$ be a factorial language
and $\Xcal_L\subseteq\Acal^{\Z^d}$ be the subshift generated by $L$.
A $d$-dimensional morphism 
$\omega:L \to\Bcal^{*^d}$ 
can be extended to a continuous map
$\omega:\Xcal_L\to\Bcal^{\Z^d}$ 
(over the topology of subshifts, as defined in Section~\ref{sec:subshift-SFT})
in such a way that the origin of $\omega(x)$ is at position $\zero$
in the word $\omega(x_\zero)$
for all $x\in\Xcal_L$. 
More precisely, the image
under $\omega$ of the configuration $x\in\Xcal_L$ is
\[
    \omega(x) =
    \lim_{n\to\infty}\sigma^{f(n)}\omega\left(\sigma^{-n\UN}(x|_{\llbracket-n\UN,n\UN\llbracket})\right)
    \in\Bcal^{\Z^d}
\]
where $\UN=(1,\dots,1)\in\Z^d$,
$f(n)=\scsize\left(\omega(\sigma^{-n\UN}(x|_{\llbracket-n\UN,\zero\llbracket}))\right)$
for all $n\in\N$ and
$\llbracket\bm,\bn\llbracket
=
\ZZrange{m_1}{n_1-1}
\times \cdots \times
\ZZrange{m_d}{n_d-1}$.
We say that the map $\omega:\Xcal_L\to\Bcal^{\Z^d}$ 
is a \defn{$d$-dimensional substitution}.

In general, the image of a subshift under a $d$-dimensional substitution
might not be closed under the shift.
But the closure under the shift of the image of a subshift $X\subseteq\Acal^{\Z^d}$ under $\omega$
is the subshift
\begin{equation*}
\overline{\omega(X)}^{\sigma}
    = \{\sigma^\bk\omega(x)\in\Bcal^{\Z^d} \mid \bk\in\Z^d, x\in X\}
    \subseteq \Bcal^{\Z^d}.
\end{equation*}

This motivates the following definition.

\begin{definition}
Let $X$, $Y$ be two subshifts
and $\omega:X\to Y$ be a $d$-dimensional substitution
If $Y=\overline{\omega(X)}^\sigma$, then
we say that $\omega$
is \defn{onto up to a shift}.
\end{definition}

\subsection{Self-similar subshifts}\label{chap:Labbe:subsec:self-similar-subshifts}

In this section, we consider languages and subshifts defined from morphisms
leading to self-similar structures. 
In this situation, the domain and codomain
of morphisms are defined over the same alphabet. 
Formally, we consider the case of $d$-dimensional morphisms
$\Acal\to\Bcal^{*^d}$ where $\Acal=\Bcal$.

The definition of self-similarity depends on the notion of expansiveness.
It avoids the presence of lower-dimensional self-similar structure by having
expansion in all directions.

\begin{definition}\label{def:expansive}
We say that a $d$-dimensional morphism $\omega:\Acal\to\Acal^{*^d}$ is
\defn{expansive}
if for every $a\in\Acal$ and $K\in\N$,
there exists $m\in\N$ such that $\min(\scsize(\omega^m(a)))>K$.
\end{definition}

\begin{definition}\label{def:self-similar}
A subshift $X\subseteq\Acal^{\Z^d}$ 
is \defn{self-similar}
if there exists an expansive
$d$-dimensional morphism $\omega:\Acal\to\Acal^{*^d}$ such that
$X=\overline{\omega(X)}^\sigma$.
\end{definition}

Self-similar subshifts can be constructed by iterative
application of a morphism $\omega$ starting with the letters.
The \defn{language} $\Lcal_\omega$ defined by an expansive $d$-dimensional
morphism $\omega:\Acal\to\Acal^{*^d}$ is
\begin{equation*}
    \Lcal_\omega = \{u\in\Acal^{*^d} \mid u\text{ is a $d$-dimensional subword of }
    \omega^n(a) \text{ for some } a\in\Acal\text{ and } n\in\N \}.
\end{equation*}
The \defn{substitutive shift} $\Xcal_\omega=\Xcal_{\Lcal_\omega}$
defined from the language of $\omega$ is a self-similar subshift
since $\Xcal_\omega=\overline{\omega(\Xcal_\omega)}^{\sigma}$ holds.

A $d$-dimensional morphism $\omega:\Acal\to\Acal^{*^d}$ is
\defn{primitive}
if there exists $m\in\N$ such that for every
$a,b\in\Acal$ the letter $b$ occurs in $\omega^m(a)$.
Note that if $\omega$ is primitive, then
the Perron--Frobenius theorem applies for its \defn{incidence matrix}
$M_\omega=(|\omega(a)|_b)_{(b,a)\in\Acal\times\Acal}$; see \cite{MR2590264}.

\subsection{$d$-dimensional recognizability}

The definition of recognizability dates back to the work of Host, Qu\'effelec and
Moss\'e \cite{MR1168468}. 
The definition introduced below is based on some work of Berth\'e et al.
\cite{MR4015135} on the recognizability in the case
of $S$-adic systems where more than one substitution is involved.

\begin{definition}[recognizable]
Let $X\subseteq\Acal^{\Z^d}$ and
$\omega:X\to\Bcal^{\Z^d}$ be a $d$-dimensional substitution.
If $y\in\overline{\omega(X)}^{\sigma}$, i.e.,
$y=\sigma^\bk\omega(x)$ for some $x\in X$ and $\bk\in\Z^d$, where $\sigma$ is
the $d$-dimensional shift map, we say that $(\bk,x)$ is an
\defn{$\omega$-representation of $y$}. We say that it is \defn{centered} if
$y_\zero$ lies inside of the image of $x_\zero$, i.e., if
$\zero\leq\bk<\scsize(\omega(x_\zero))$ coordinate-wise.
    We say that $\omega$ is \defn{recognizable in} $X\subseteq\Acal^{\Z^d}$
if each $y\in\Bcal^{\Z^d}$ has at most one centered $\omega$-representation 
$(\bk,x)$ with $x\in X$.
\end{definition}

The self-similarity of $\Omega_n$ allows to conclude aperiodicity of the Wang shift
using well-known arguments, see \cite{MR1637896,MR1168468}, who showed that
recognizability and aperiodicity are equivalent for primitive substitutive
sequences. 

The following statement corresponds to only one of the direction (the easy
one) of the equivalence which does not need the notion of primitivity.
It was proved for 2-dimensional substitutions in \cite{MR3978536}; see also
\cite[Proposition 3.6]{labbe_three_2020}.

\begin{proposition}\label{prop:expansive-recognizable-aperiodic}
    \cite[Proposition 6]{MR3978536}
    Let $\omega:\Acal\to\Acal^{*^d}$ be an expansive $d$-dimensional morphism.
    Let $X\subseteq\Acal^{\Z^d}$ be a self-similar subshift such that
    $\overline{\omega(X)}^\sigma=X$.
    If $\omega$ is recognizable in $X$, then $X$ is aperiodic.
\end{proposition}

\section{The family of metallic mean Wang tiles}
\label{sec:family-Tcaln}

For every integer $n\in\Z$, 
we write $\overline{n}$ to denote $n+1$
and $\underline{n}$ to denote $n-1$:
\begin{align*}
    \overline{n} &:= n+1,\\
    \underline{n} &:= n-1.
\end{align*}

For every Wang tile $\tau=(a,b,c,d)$, we define its
symmetric image under the positive diagonal as
 $\widehat{\tau}=(b,a,d,c)$:
\[
    \text{ if }
    \tau=
    \raisebox{-9.8mm}{
    \begin{tikzpicture}[auto]
    \tile{white}{0}{0}{a}{b}{c}{d}
    \end{tikzpicture}},
    \qquad
    \text{ then }
    \qquad
    \widehat{\tau}=
    \raisebox{-8.0mm}{
    \begin{tikzpicture}[auto]
    \tile{white}{0}{0}{b}{a}{d}{c}
    \end{tikzpicture}}.
\]

\subsection{The tiles}

For every integer $n\geq 1$, let
\[
    V_n = \{(v_0,v_1,v_2)\in\Z^3\colon 0\leq v_0\leq v_1\leq 1\text{ and } v_1\leq v_2\leq n+1\}.
\]
be a set of vectors having non-decreasing entries with an upper bound of 1 on the middle entry
and an upper bound of $n+1$ on the last entry.
The label of the edges of the Wang tiles considered in this article are vectors
in $V_n$. To lighten the figures and the presentation of the Wang tiles, it is convenient to denote
the vector $(v_0,v_1,v_2)\in V_n$ more compactly as a word $v_0v_1v_2$.
For instance the vector $(1,1,1)$ is represented as $111$.

For every integer $n\geq 1$,
we define the following set of Wang tiles whose labels belong to the set $V_n$.
We have $n^2$ white tiles whose labels all start with $11$:
\begin{align*}
W_n &= \left\{ 
        \;w_n^{i,j} := 
        \raisebox{-10mm}{
            \begin{tikzpicture}[auto]
                \tikzstyle{every node}=[font=\footnotesize]
                \tile{white}{0}{0}{11\ibar}{11\jbar}{11i}{11j}
            \end{tikzpicture}}
       \;\middle|
       \begin{array}{l}
       1\leq i\leq n,\\
       1\leq j\leq n 
       \end{array}
       \right\}
    &&(n^2\text{ white tiles}).
\end{align*}

We have horizontal stripe tiles whose top
and bottom labels all start with $11$
and whose left and right labels start with $0$.
These are divided into four sets according to the first two digits
of the left and right labels which can be $00$ (associated with color blue) or 
$01$ (associated with color yellow).
\begingroup
\allowdisplaybreaks
\begin{align*}
B_n' &= \left\{ 
        \;b_n^{i} := 
    \raisebox{-9mm}{
    \begin{tikzpicture}[auto]
        \tikzstyle{every node}=[font=\footnotesize]
        \tileH{\ourColorBlue}{0}{0}{00\ibar}{111}{00i}{11n}
    \end{tikzpicture}}
    \;\middle|\; 0\leq i \leq n \right\}
    &&(n+1\text{ horizontal blue stripe tiles}),\\
G_n &= \left\{
        \;g_n^{i} := 
    \raisebox{-9mm}{
    \begin{tikzpicture}[auto]
        \tikzstyle{every node}=[font=\footnotesize]
        \tileHgreen{0}{0}{01\ibar}{111}{00i}{11\nbar}
    \end{tikzpicture}}
    \;\middle|\; 0\leq i\leq n \right\}
    &&(n+1\text{ horizontal green stripe tiles}),\\
Y_n &= \left\{
        \;y_n^{i} := 
    \raisebox{-9mm}{
    \begin{tikzpicture}[auto]
        \tikzstyle{every node}=[font=\footnotesize]
    \tileH{\ourColorYellow}{0}{0}{01\ibar}{112}{01i}{11\nbar}
    \end{tikzpicture}}
    \;\middle|\; 1\leq i\leq n \right\}
    &&(n\text{ horizontal yellow stripe tiles}),\\
A_n &= \left\{ 
        \;a_n^{i} := 
    \raisebox{-9mm}{
    \begin{tikzpicture}[auto]
        \tikzstyle{every node}=[font=\footnotesize]
    \tileHantigreen{0}{0}{00\ibar}{112}{01i}{11n}
    \end{tikzpicture}}
    \;\middle|\; 1\leq i \leq n \right\}
    &&(n \text{ horizontal antigreen tiles}).
\end{align*}
\endgroup
The set $B_n'$ of horizontal blue tiles are those such that both left and
right labels start with $00$ and are identified with a horizontal blue
stripe.
The set $Y_n$ of horizontal yellow tiles are those such that both left and
right labels start with $01$ and are identified with a horizontal yellow
stripe.
The set $G_n$ of horizontal green tiles are those such that the left label
starts with $00$ and right label starts with $01$ and are identified with a
green region intersecting blue and yellow horizontal stripes.
The set $A_n$ of horizontal antigreen tiles contains the tiles whose left label
starts with $01$ and whose right label starts with $00$. They are identified with
non-intersecting blue and yellow horizontal stripes and no green intersecting
region.

The tiles in $A_n$ are called ``antigreen'' because they are ``against the system''
as shown later in Lemma~\ref{lem:no-antigreen-tile}.
Antigreen tiles do not appear in any valid configuration, but they are needed
as they play an important role in the description of the substitutive structure of
the valid configurations allowed by these tiles; 
see Proposition~\ref{prop:there-exists-rect-pattern-iff}
and Proposition~\ref{prop:exists-unique-preimage}.

We also have vertical stripe tiles which are the symmetric images of the horizontal
stripe tiles under a reflection over the positive diagonal:
\begingroup
\allowdisplaybreaks
\begin{align*}
    \widehat{B_n'} &= \left\{ 
        \;\widehat{b_n^{i}} := 
    \raisebox{-9mm}{
    \begin{tikzpicture}[auto]
        \tikzstyle{every node}=[font=\footnotesize]
        \tileV{\ourColorBlue}{5}{0}{111}{00\ibar}{11n}{00i}
    \end{tikzpicture}}
    \;\middle|\; 0\leq i \leq n \right\}
    &&(n+1\text{ vertical blue stripe tiles}),\\
    \widehat{G_n} &= \left\{
        \;\widehat{g_n^{i}} := 
    \raisebox{-9mm}{
    \begin{tikzpicture}[auto]
        \tikzstyle{every node}=[font=\footnotesize]
        \tileVgreen{5}{0}{111}{01\ibar}{11\nbar}{00i}
    \end{tikzpicture}}
    \;\middle|\; 0\leq i\leq n \right\}
    &&(n+1\text{ vertical green stripe tiles}),\\
    \widehat{Y_n} &= \left\{
        \;\widehat{y_n^{i}} := 
    \raisebox{-9mm}{
    \begin{tikzpicture}[auto]
        \tikzstyle{every node}=[font=\footnotesize]
        \tileV{\ourColorYellow}{5}{0}{112}{01\ibar}{11\nbar}{01i}
    \end{tikzpicture}}
    \;\middle|\; 1\leq i\leq n \right\}
    &&(n\text{ vertical yellow stripe tiles}),\\
    \widehat{A_n} &= \left\{ 
        \;\widehat{a_n^{i}} := 
    \raisebox{-9mm}{
    \begin{tikzpicture}[auto]
        \tikzstyle{every node}=[font=\footnotesize]
        \tileVantigreen{5}{0}{112}{00\ibar}{11n}{01i}
    \end{tikzpicture}}
    \;\middle|\; 1\leq i \leq n \right\}
    &&(n \text{ vertical antigreen tiles}).
\end{align*}
\endgroup

Finally, we have 9 junction tiles (the gray region is drawn in blue or
yellow color depending on the specific values of $k,\ell,r,s$):
\begingroup
\allowdisplaybreaks
\begin{align*}
J_n' &= 
    \left\{ 
        \;j_n^{k,\ell,r,s}:=
        \raisebox{-10mm}{
        \begin{tikzpicture}[auto]
        \tikzstyle{every node}=[font=\footnotesize]
            \tileJunctionGRAY{0}{2}{(0,k,\ell)}{(0,r,s)}{(0,s,r+n)}{(0,\ell,k+n)}{gray!50}
        \end{tikzpicture}}
    \;\middle|\; 
       \begin{array}{l}
        0\leq k \leq \ell \leq 1,\\
        0\leq r \leq s \leq 1
       \end{array}
    \right\}\\
    &=
    \left\{
    \raisebox{-8mm}{
    \begin{tikzpicture}[auto]
        \def\dt{2.2}
        \tikzstyle{every node}=[font=\footnotesize]
        \tileJunctionIIII{2*\dt}{0}  {}{011}{01\nbar}{}
        \draw[fill=white,draw=white] (2*\dt,0) -- ++ (1,1) -- ++ (0,-1) -- cycle;
        \draw[dashed,thick]  (2*\dt,0) -- ++ (1,1);
        \tileJunctionOOII{\dt}{0}  {}{001}{01n}{}
        \draw[fill=white,draw=white] (\dt,0) -- ++ (1,1) -- ++ (0,-1) -- cycle;
        \draw[dashed,thick]  (\dt,0) -- ++ (1,1);
        \tileJunctionOOOO{0}{0}  {}{000}{00n}{}        
        \draw[fill=white,draw=white] (0,0) -- ++ (1,1) -- ++ (0,-1) -- cycle;
        \draw[dashed,thick]  (0,0) -- ++ (1,1);
    \end{tikzpicture}}
    \right\}
    \times
    \left\{
    \raisebox{-8mm}{
    \begin{tikzpicture}[auto]
        \def\dt{2.2}
        \tikzstyle{every node}=[font=\footnotesize]
        \tileJunctionIIII{2*\dt}{0}  {011}{}{}{01\nbar}
        \draw[fill=white,draw=white] (2*\dt,0) -- ++ (1,1) -- ++ (-1,0) -- cycle;
        \draw[dashed,thick]  (2*\dt,0) -- ++ (1,1);
        \tileJunctionOOII{\dt}{0}{001}{}{}{01n}
        \draw[fill=white,draw=white] (\dt,0) -- ++ (1,1) -- ++ (-1,0) -- cycle;
        \draw[dashed,thick]  (\dt,0) -- ++ (1,1);
        \tileJunctionOOOO{0}{0}  {000}{}{}{00n}        
        \draw[fill=white,draw=white] (0,0) -- ++ (1,1) -- ++ (-1,0) -- cycle;
        \draw[dashed,thick]  (0,0) -- ++ (1,1);
    \end{tikzpicture}}
    \right\}\\
    &= 
    \left\{
    \raisebox{-33mm}{
    \begin{tikzpicture}[auto]
        \tikzstyle{every node}=[font=\footnotesize]
        \tileJunctionIOOI{0}{5}  {011}{000}{00n}{01\nbar}
        \tileJunctionIOII{3}{5}  {011}{001}{01n}{01\nbar}
        \tileJunctionIIII{6}{5}  {011}{011}{01\nbar}{01\nbar}
        \tileJunctionOOOI{0}{2.5}{001}{000}{00n}{01n}        
        \tileJunctionOOII{3}{2.5}{001}{001}{01n}{01n}
        \tileJunctionOIII{6}{2.5}{001}{011}{01\nbar}{01n}
        \tileJunctionOOOO{0}{0}  {000}{000}{00n}{00n}        
        \tileJunctionOOIO{3}{0}  {000}{001}{01n}{00n}
        \tileJunctionOIIO{6}{0}  {000}{011}{01\nbar}{00n}
    \end{tikzpicture}}
    \right\}
    \qquad\qquad(9 \text{ junction tiles}).
\end{align*}
\endgroup

We may observe that white tiles and junction tiles are closed under the reflection
over the positive diagonal:
\[
\widehat{W_n}=W_n
\text{ and }
\widehat{J_n'}=J_n'.
\]
Junction tiles play a similar role to junction tiles in \cite{MR1014984},
thus we reuse the same vocabulary.

\subsection{The extended set $\Tcal_n'$ of metallic mean Wang tiles}
\label{sec:larger-family-Tcaln'}

For every integer $n\geq1$, the
\defn{extended set of metallic mean Wang tiles} is
the union of all of the tiles defined above:
\[
\Tcal'_n = W_n
             \cup B'_n 
             \cup G_n  
             \cup Y_n  
             \cup A_n  
             \cup \widehat{B'_n} 
             \cup \widehat{G_n} 
             \cup \widehat{Y_n} 
             \cup \widehat{A_n} 
             \cup J'_n.
\]
The set $\Tcal_n'$ of tiles defines the \defn{extended metallic mean Wang shift}
\[
\Omega'_n=\Omega_{\Tcal'_n}.
\]
The set $\Tcal'_n$
contains $n^2+2(n+1+n+1+n+n)+9=n^2+8n+13$ Wang tiles.
The set $\Tcal'_n$ of Wang tiles for $n=4$ is shown in
Figure~\ref{fig:T'n-for-1-a-5}.

\begin{figure}[h]
\begin{center}
\includegraphics{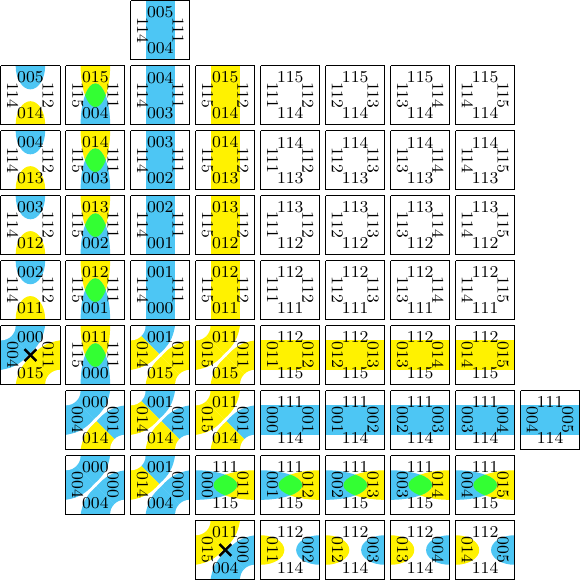}
\end{center}
    \caption{Extended metallic mean Wang tile sets $\Tcal'_n$ for $n=4$.
    The junction tiles in $\Dcal$ are shown with a $\times$-mark in their center.
    }
    \label{fig:T'n-for-1-a-5}
\end{figure}

\subsection{The subset $\Tcal_n$ of metallic mean Wang tiles}

We need to define an important subset of the extended
metallic mean Wang tiles $\Tcal_n'$, because 
some of the tiles are not necessary as they do not appear in any
valid configurations of $\Omega_n'$. 
For example, we can observe that no tile of $\Tcal_n'$ have label $00\nbar$ on
the left or bottom. Therefore, the last horizontal blue tile
and last vertical blue tile which use label $00\nbar$ on its top or right edge
admit no immediate surroundings with tiles in $\Tcal_n'$.
As shown in Section~\ref{sec:Omegan'subseteqOmegan}
using results proved in Section~\ref{sec:substitution} and
Section~\ref{sec:desubstitution},
other tiles from $\Tcal_n'$ do not admit arbitrarily large surroundings.
Therefore, it is convenient to remove them.

Let
\begin{align*}
    \Dcal&=
        \{b_n^n,\;
          \widehat{b_n^n},\;
          j_n^{0,0,1,1},\;
          j_n^{1,1,0,0}\}\\
        &=
    \left\{ 
    \raisebox{-9mm}{
    \begin{tikzpicture}[auto]
        \tikzstyle{every node}=[font=\footnotesize]
        \tileH{\ourColorBlue}{5}{0}{00\nbar}{111}{00n}{11n}
    \end{tikzpicture}},
    \raisebox{-9mm}{
    \begin{tikzpicture}[auto]
        \tikzstyle{every node}=[font=\footnotesize]
        \tileV{\ourColorBlue}{5}{0}{111}{00\nbar}{11n}{00n}
    \end{tikzpicture}},
    \raisebox{-9mm}{
    \begin{tikzpicture}[auto]
        \tikzstyle{every node}=[font=\footnotesize]
        \tileJunctionOIIO{6}{0}  {000}{011}{01\nbar}{00n}
    \end{tikzpicture}},
    \raisebox{-9mm}{
    \begin{tikzpicture}[auto]
        \tikzstyle{every node}=[font=\footnotesize]
        \tileJunctionIOOI{0}{5}  {011}{000}{00n}{01\nbar}
    \end{tikzpicture}}
    \right\}
\end{align*}
be the set containing the last blue horizontal and vertical tiles
as well as two of the junction tiles.
For every positive integer $n$,
we delete the four tiles of $\Dcal$ from $\Tcal_n'$
as well as all of the antigreen tiles.
We obtain the following \defn{subset of metallic mean Wang tiles}
\begin{align*}
    \Tcal_n&= \Tcal_n' \setminus 
    \left(A_n\cup\widehat{A_n} \cup \Dcal \right)\\
           &= W_n
             \cup B_n 
             \cup G_n  
             \cup Y_n  
             \cup \widehat{B_n} 
             \cup \widehat{G_n} 
             \cup \widehat{Y_n} 
             \cup J_n,
\end{align*}
where 
$B_n=B_n'\setminus\Dcal$ is the remaining set of $n$ horizontal blue stripe tiles
and
$J_n=J_n'\setminus\Dcal
    =\left\{ 
    j_n^{0,0,0,0},
    j_n^{0,1,0,0},
    j_n^{0,0,0,1},
    j_n^{0,1,0,1},
    j_n^{1,1,0,1},
    j_n^{0,1,1,1},
    j_n^{1,1,1,1}
    \right\}$ is the remaining set of $7$ junction tiles.
The set $\Tcal_n$ contains $n^2+2(n+n+1+n)+7=(n+3)^2$ Wang tiles.
It is shown in Figure~\ref{fig:Tn-for-1-a-5}
for $n=1,2,3,4,5$.

The set $\Tcal_n$ of tiles defines the \defn{Metallic mean Wang shift}
\[
    \Omega_n=\Omega_{\Tcal_n}
\]
which is a subshift of $\Omega_n'$ because $\Tcal_n\subset\Tcal_n'$.

\begin{remark}
    The reader may wonder why we need to introduce the extended set $\Tcal_n'$ if
    only the tiles in the subset $\Tcal_n$ appear in configurations of $\Omega_n'$.
    This is because the extended set is needed to describe and prove the
    self-similarity of $\Tcal_n$ in Theorem~\ref{thm:similar-to-itself}.
    In the proof
    (using the vocabulary of supertiles from the only article published by Ammann \cite{MR1156132}),
    we show that if the markings of the supertiles at level $k$ are in bijection 
    with the tiles in $\Tcal_n$,
    then the markings of the supertiles at level $k+1$ are in bijection
    with tiles in $\Tcal_n'$ (not $\Tcal_n$!).
    In other words, we cannot get rid of the ghost tiles in
    $\Tcal_n'\setminus\Tcal_n$, because they keep reappearing at the next level
    of the hierarchy in bigger sizes.
\end{remark}

\subsection{The Ammann aperiodic set of 16 Wang tiles}

A reproduction of the 
Ammann aperiodic set of 16 Wang tiles
\cite[p.595, Figure 11.1.13]{MR857454}
is shown in Figure~\ref{fig:Ammann}.
The Ammann set of 16 Wang tiles corresponds to $\Tcal_1$.

\begin{figure}[h]
\begin{center}
    \includegraphics{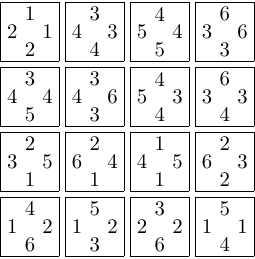}
    \qquad
    \includegraphics{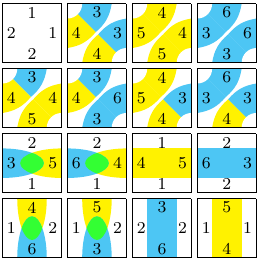}
    \qquad
    \includegraphics{SAGEOUTPUT/W1_tiles.pdf}
\end{center}
    \caption{Left: a reproduction of the 
    Ammann aperiodic set of 16 Wang tiles
    \cite[p.595, Figure 11.1.13]{MR857454}.
    Middle: the Ammann aperiodic set of 16 Wang tiles
    in the same order but
    with coloring corresponding to the white, yellow, green, blue and junction tiles
    of the set $\Tcal_1$.
    Right: The set $\Tcal_1$ of Wang tiles whose edge labels are vectors in $\N^3$.
    The sets are equivalent up to a bijection of the edge labels.}
    \label{fig:Ammann}
\end{figure}

\begin{THEOREMIV}
    \MainTheoremIV
\end{THEOREMIV}

\begin{proof}
    The following is a bijection from 
    the labels of the Ammann set of 16 Wang tiles
    and 
    the labels of the tiles in $\Tcal_1$:
    \begin{align*}       
1 \mapsto 112,\quad
2 \mapsto 111,\quad
3 \mapsto 001, \quad
4 \mapsto 011, \quad
5 \mapsto 012, \quad
6 \mapsto 000. 
    \end{align*}
    See Figure~\ref{fig:Ammann} (note that the order of the tiles is not the
    same).
\end{proof}

Thus, the family $(\Tcal_n)_{n\geq1}$ can be considered as a generalization of
the Ammann aperiodic set of 16 Wang tiles.

\subsection{Symmetric properties}

The set $\Tcal_n$ has nice symmetric properties. The first being that it is
closed under the mirror image through the positive diagonal, that is,
$\widehat{\Tcal_n}=\Tcal_n$. Another less evident observation
is that the set $\Tcal_n$ is equivalent to its image under a half-turn rotation 
up to the application of an involution of $V_n\setminus\{(0,0,\nbar)\}$
applied on the edge labels of the Wang tiles.

\begin{lemma}
    Let $\sigma:(i,j,k) \mapsto (i,1+i-j,n+1+i-k)$
    (an involution on $V_n\setminus\{(0,0,\nbar)\}$).
    Then,
    \[
        \Tcal_n = 
        \left\{
    \raisebox{-8mm}{
    \begin{tikzpicture}[auto,scale=.5]
        \tikzstyle{every node}=[font=\footnotesize]
        \tile{white}{0}{0}{\sigma(u)}{\sigma(v)}{\sigma(\alpha)}{\sigma(\beta)}
    \end{tikzpicture}}
    \;\middle|\;
    \raisebox{-6mm}{
    \begin{tikzpicture}[auto,scale=.5]
        \tikzstyle{every node}=[font=\footnotesize]
        \tile{white}{0}{0}{\alpha}{\beta}{u}{v}
    \end{tikzpicture}}
    \in\Tcal_n
        \right\}
    \]
\end{lemma}

\begin{proof}
When rotating the tiles of $\Tcal_n$ by half a turn
and applying the map $\sigma$ on the resulting labels,
we may observe that
yellow tiles become blue tiles and vice versa,
white tiles are mapped to white tiles, 
junction tiles are mapped to junction tiles and
green tiles are mapped to green tiles.
\end{proof}

This translates into the existence of non trivial reflection symmetry and
rotational symmetry for the Wang shift $\Omega_n$.
As we show in this article, it has no translational symmetries.

\subsection{Directional determinism}
We show in this section that
the sets $\Tcal_n$ and $\Tcal_n'$ are SW- and NE-deterministic.

\begin{lemma}\label{lem:SW-and-NE-deterministic}
    The sets $\Tcal_n$ and $\Tcal_n'$ are SW- and NE-deterministic.
    However, the sets $\Tcal_n$ and $\Tcal_n'$ are neither NW- nor
    SE-deterministic.
\end{lemma}

\begin{proof}
    Let us show that $\Tcal_n'$ is SW-deterministic.
    Let $s,t\in\Tcal_n'$ be such that
    $\scleft(s)=u=\scleft(t)$
    and $\scbottom(s)=v=\scbottom(t)$
    for some vectors $u=(u_0,u_1,u_2),v=(v_0,v_1,v_2)\in V_n$.
    \begin{itemize}
        \item If $u_0=0$, $v_0=0$, then $s,t\in J'_n$.
        \item If $u_0=1$, $v_0=1$, then $s,t\in W_n$.
        \item If $u_0=0$, $v_0=1$, $u_1=0$ and $v_2=n$,     then $s,t\in B'_n$.
        \item If $u_0=0$, $v_0=1$, $u_1=0$ and $v_2=\nbar$, then $s,t\in G_n$.
        \item If $u_0=0$, $v_0=1$, $u_1=1$ and $v_2=n$,     then $s,t\in A_n$.
        \item If $u_0=0$, $v_0=1$, $u_1=1$ and $v_2=\nbar$, then $s,t\in Y_n$.
        \item If $u_0=1$, $v_0=0$, $v_1=0$ and $u_2=n$,     then $s,t\in \widehat{B'_n}$.
        \item If $u_0=1$, $v_0=0$, $v_1=0$ and $u_2=\nbar$, then $s,t\in \widehat{G_n}$.
        \item If $u_0=1$, $v_0=0$, $v_1=1$ and $u_2=n$,     then $s,t\in \widehat{A_n}$.
        \item If $u_0=1$, $v_0=0$, $v_1=1$ and $u_2=\nbar$, then $s,t\in \widehat{Y_n}$.
    \end{itemize}
    One can observe that each of the subsets
    $W_n$, $B'_n$, $G_n$, $Y_n$, $A_n$, $J'_n$
    is SW-deterministic.
    By symmetry, $\widehat{B'_n}$,
             $\widehat{G_n}$,
             $\widehat{Y_n}$ and
             $\widehat{A_n}$
    are SW-deterministic.
    We conclude that $s=t$. Thus,
    $\Tcal_n'$ is SW-deterministic.
    Using a similar argument, one can observe that $\Tcal_n'$ is NE-deterministic.
    By restriction, the subset $\Tcal_n\subset\Tcal_n'$ is SW- and NE-deterministic.

    However, $\Tcal_n$ is neither NW- nor SE-deterministic, because the subset
    $J_n$ is neither NW- nor SE-deterministic.
    By extension, the extended set $\Tcal_n'$
    is neither NW- nor SE-deterministic.
\end{proof}

\section{A substitution $\Omega_n\to\Omega_n$}
\label{sec:substitution}

The goal of this section is twofold.
First, we introduce a 2-dimensional substitution $\Omega_n\to\Omega_n$
deduced from a substitution $\tau_n:V_n\to(V_n)^*$ defined on the boundary labels of
the Wang tiles.
Then, we characterize the possible valid rectangular
tilings with external labels in the image of $\tau_n$; see
Proposition~\ref{prop:there-exists-rect-pattern-iff}.

\subsection{A one-dimensional substitution for the boundary}
It is convenient to define, for every integer $n\geq1$, the following map
\[
\begin{array}{rccl}
    \tau_n:&V_n & \to & (V_n)^*\\
    &xyz & \mapsto &
\begin{cases}
    0(x{-}y{+}1)n    \cdot (11n)^{z-x-1} \cdot (11\nbar)^{n+1-z}   & \text{ if } x\neq z,\\
    0(x{-}y{+}1)\nbar\cdot (11\nbar)^{n-z}                         & \text{ if } x=z.
\end{cases}
\end{array}
\]
The above formula declines into the following five cases:
\begin{equation}\label{eq:tau_n_five_cases}
\begin{aligned}
    \tau_n(000)    &=01\nbar\cdot(11\nbar)^{n},\\
    \tau_n(111)    &=01\nbar\cdot(11\nbar)^{n-1},\\
    \tau_n(00\ibar)&=01n\cdot(11n)^{i}\cdot(11\nbar)^{n-i}   &&\qquad\text{ if }0\leq i\leq n,\\
    \tau_n(01\ibar)&=00n\cdot(11n)^{i}\cdot(11\nbar)^{n-i}   &&\qquad\text{ if }0\leq i\leq n,\\
    \tau_n(11\ibar)&=01n\cdot(11n)^{i-1}\cdot(11\nbar)^{n-i} &&\qquad\text{ if }1\leq i\leq n.
\end{aligned}
\end{equation}
For example, when $n=1$, $n=2$ or $n=4$, we have
\[
    \tau_1
    \left\{\begin{array}{l}
000\mapsto 012,112\\
001\mapsto 011,112\\
002\mapsto 011,111\\
011\mapsto 001,112\\
012\mapsto 001,111\\
111\mapsto 012\\
112\mapsto 011\\
    \end{array}\right.,
    \qquad
    \tau_2
    \left\{\begin{array}{l}
000\mapsto 013,113,113\\
001\mapsto 012,113,113\\
002\mapsto 012,112,113\\
003\mapsto 012,112,112\\
011\mapsto 002,113,113\\
012\mapsto 002,112,113\\
013\mapsto 002,112,112\\
111\mapsto 013,113\\
112\mapsto 012,113\\
113\mapsto 012,112\\
    \end{array}\right.,
    \qquad
     \tau_4
     \footnotesize
     \left\{\begin{array}{l}
 000\mapsto 015,115,115,115,115\\
 001\mapsto 014,115,115,115,115\\
 002\mapsto 014,114,115,115,115\\
 003\mapsto 014,114,114,115,115\\
 004\mapsto 014,114,114,114,115\\
 005\mapsto 014,114,114,114,114\\
 011\mapsto 004,115,115,115,115\\
 012\mapsto 004,114,115,115,115\\
 013\mapsto 004,114,114,115,115\\
 014\mapsto 004,114,114,114,115\\
 015\mapsto 004,114,114,114,114\\
 111\mapsto 015,115,115,115\\
 112\mapsto 014,115,115,115\\
 113\mapsto 014,114,115,115\\
 114\mapsto 014,114,114,115\\
 115\mapsto 014,114,114,114\\
     \end{array}\right..
\]

The map $\tau_n$ was discovered during computer explorations.
It appears naturally when searching for a self-similarity for the tilings
in $\Omega_n$; see Appendix B in Section~\ref{sec:Appendix-self-similarity-Omega2}
and in particular the output of the computation performed at line~\ref{vert_bijection}.

\begin{lemma}
    For every $(x,y,z)\in V_n$, the map $\tau_n$ satisfies:
    \begin{itemize}
        \item the length of $\tau_n(xyz)\in(V_n)^*$ is $n+1-x$;
        \item the first item of $\tau_n(xyz)$ is $0(x{-}y{+}1)n$ or
            $0(x{-}y{+}1)\nbar$;
        \item there are $z-x$ occurrences of ${*}{*}n$ in the image of $\tau_n(xyz)$.
    \end{itemize}
    In particular, $\tau_n$ is injective.
\end{lemma}

\begin{proof}
    The three items follow from the definition.
    We prove that $\tau_n$ is injective.
    Assume that $xyz\neq x'y'z'$.
    We want to show that $\tau_n(xyz)\neq \tau_n(x'y'z')$.
    If $x\neq x'$, then the images are distinct because their lengths are different.
    If $x=x'$ and $y\neq y'$, then the images are distinct because
    of the second digit of their first item satisfy $x{-}y{+}1\neq x{-}y'{+}1$.
    If $x=x'$, $y=y'$ and $z\neq z'$, then the images are distinct because
    there are $z-x$ occurrences of ${*}{*}n$ in the image of $\tau_n(xyz)$.
\end{proof}

\subsection{A substitution $\omega_n'$ for the tiles in $\Tcal_n'$}

Let
\[
    Q_n=\left\{
    \raisebox{-9mm}{
    \begin{tikzpicture}[auto]
        \def\dt{2.2}
        \tikzstyle{every node}=[font=\footnotesize]
        \tileJunctionIIII{2*\dt}{0}  {011}{}{}{01\nbar}
        \draw[fill=white,draw=white] (2*\dt,0) -- ++ (1,1) -- ++ (-1,0) -- cycle;
        \draw[dashed,thick]  (2*\dt,0) -- ++ (1,1);
        \tileJunctionOOII{\dt}{0}{001}{}{}{01n}
        \draw[fill=white,draw=white] (\dt,0) -- ++ (1,1) -- ++ (-1,0) -- cycle;
        \draw[dashed,thick]  (\dt,0) -- ++ (1,1);
        \tileJunctionOOOO{0}{0}  {000}{}{}{00n}        
        \draw[fill=white,draw=white] (0,0) -- ++ (1,1) -- ++ (-1,0) -- cycle;
        \draw[dashed,thick]  (0,0) -- ++ (1,1);
    \end{tikzpicture}}
    \right\}
\]
be the set of possible values for the bottom right part of a junction
tile in $J_n'$.

\begin{lemma}\label{lem:unique-horizontal-strip}
    Let $n\geq1$ be an integer.
    For every $v\in V_n$,
    there exist
    a unique bottom right part $q\in Q_n$ and
    a unique sequence $t_1t_2\dots t_{k-1}$ of tiles in $\Tcal_n'$
    such that
    $q\,t_1t_2\dots t_{k-1}$ is a valid horizontal strip of tiles from left to right
    whose sequence of bottom labels is $\tau_n(v)$
    where $k=|\tau_n(v)|$.

    Moreover,
    if $\gamma$ is the sequence of top labels of $t_1t_2\dots t_{k-1}$ 
    and $\delta$ is its right-most right label, that is, the right label of $t_{k-1}$
    (see Figure~\ref{fig:horizontal-stripe-of-tiles}), then
    the following statements hold.
    \begin{itemize}
    \item If $v=00i$ with $0\leq i\leq\nbar$, then
        \begin{itemize}
            \item if $0\leq i\leq n$, then
                    $\gamma=(111)^i\cdot(112)^{n-i}$ and $\delta=01\nbar$,
            \item if $i=n+1$, then $\gamma=(111)^{n}$ and $\delta=00\nbar$.
        \end{itemize}
    \item If $v=01i$ with $1\leq i\leq\nbar$, then
        \begin{itemize}
            \item if $1\leq i\leq n$, then
                    $\gamma=(111)^i\cdot(112)^{n-i}$ and $\delta=01n$,
            \item if $i=n+1$, then $\gamma=(111)^{n}$ and $\delta=00n$.
        \end{itemize}
    \item If $v=11i$ with $1\leq i\leq\nbar$, then
        \begin{itemize}
            \item if $1\leq i\leq n$, then $\gamma=(111)^{i-1}\cdot(112)^{n-i}$ 
                    and $\delta=01n$,
            \item if  $i=n+1$, then $\gamma=(111)^{n-1}$ and $\delta=00n$.
        \end{itemize}
    \end{itemize}
In particular, no antigreen tiles appear in the horizontal strip. Also, if
$v\in V_n\setminus\{00\nbar\}$, then the last blue tile
does not appear in the strip.
\end{lemma}

\begin{figure}[h]
\begin{center}
    \begin{tikzpicture}[scale=1.2]
    \tileJunctionBackground{0}{0}{gray!50}{gray!50}{gray!50}{gray!50}
    \draw[fill=white,draw=white] (0,0) -- ++ (1,1) -- ++ (-1,0) -- cycle;
    \draw[dashed,thick]  (0,0) -- ++ (1,1);
    \draw (0,0) -- (1,0);
    \draw[dotted,thick] (0,0) -- (0,1) -- (1,1);
    \draw[draw=none,fill=gray!50] (1,\dt) rectangle (5,1-\dt);
    \draw (1,0) rectangle (5,1);
        \node at (.66,.33) {$q$};
    \node at (3,.5) {$t_1t_2\dots t_{k-1}$};
    \node[right] at (5,0.5) {$\delta$};
    \draw[yshift=-3mm,|-|] (0,0) -- node[fill=white] {$\tau_n(v)$}        (5,0);
    \draw[yshift=3mm,|-|] (1,1) -- node[fill=white] {$\gamma$}        (5,1);
\end{tikzpicture}
\end{center}
    \caption{A horizontal strip of tiles from $\Tcal_n'$
    made of a bottom right part $q$ of a junction tile and a sequence
    $t_1t_2\dots t_{k-1}$ of horizontal stripe tiles.
    The bottom labels of the strip is $\tau_n(v)$ for some $v\in V_n$.
    The top labels of the horizontal stripe tiles is $\gamma\in (V_n)^*$
    and its right-most right label is $\delta\in V_n$.}
    \label{fig:horizontal-stripe-of-tiles}
\end{figure}

\begin{proof}
    Assume $v=00i$ with $0\leq i\leq\nbar$. The following three cases occur.
            \begin{itemize}
                \item If $i=0$, then the sequence of bottom labels is
                $\tau_n(000) = 01\nbar\cdot (11\nbar)^n$, 
                        the sequence of top labels is $(112)^n$ 
                        and the right-most right label is $01\nbar$.
                \item If $1\leq i\leq n$, then the sequence of bottom labels is
                $\tau_n(00i) = 01n\cdot (11n)^{i-1}\cdot (11\nbar)^{n+1-i}$,
                        the sequence of top labels is $(111)^i\cdot(112)^{n-i}$ 
                        and the right-most right label is $01\nbar$.
                \item If $i=n+1$, then the sequence of bottom labels is
                    $\tau_n(00i) = 01n\cdot (11n)^{n}$, the sequence of top
                    labels is $(111)^{n}$ and the right-most right label is
                    $00\nbar$.
            \end{itemize}
        The $n+1$ tiles of the strip for the three cases are illustrated in 
        Figure~\ref{fig:bottom-row-if-bottom-labels-is-tau00i}.
        We observe that the last blue tile (the blue horizontal stripe tile
        with left label $00n$) is used in the strip only when $i=n+1$.

    Assume $v=01i$ with $1\leq i\leq\nbar$. The following two cases occur.
        \begin{itemize}
            \item If $1\leq i\leq n$, then the sequence of bottom labels is
            $\tau_n(01i) = 00n\cdot (11n)^{i-1}\cdot (11\nbar)^{n+1-i}$,
                    the sequence of top labels is $(111)^i\cdot(112)^{n-i}$ 
                    and the right-most right label is $01n$.
            \item If $i=n+1$, then the sequence of bottom labels is
            $\tau_n(00i) = 00n\cdot (11n)^{n}$,
                        the sequence of top labels is $(111)^{n}$ 
                    and the right-most right label is $00n$.
        \end{itemize}
        The $n+1$ tiles of the strip for the two cases are illustrated in 
        Figure~\ref{fig:bottom-row-if-bottom-labels-is-tau01i}.

    Assume $v=11i$ with $1\leq i\leq\nbar$. The following three cases occur.
        \begin{itemize}
            \item If $i=1$, then the sequence of bottom labels is
            $\tau_n(111) = 01\nbar\cdot (11\nbar)^{n-1}$,
                    the sequence of top labels is $(112)^{n-1}$ 
                    and the right-most right label is $01n$.
            \item If  $2\leq i\leq n$, then the sequence of bottom labels is
            $\tau_n(11i) = 01n\cdot (11n)^{i-2}\cdot (11\nbar)^{n+1-i}$,
                    the sequence of top labels is $(111)^{i-1}\cdot(112)^{n-i}$ 
                    and the right-most right label is $01n$.
            \item If  $i=n+1$, then the sequence of bottom labels is
                $\tau_n(11i) = 01n\cdot (11n)^{n-1}$, the sequence of top
                labels is $(111)^{n-1}$ and the right-most right label is
                $00n$.
        \end{itemize}
        The $n$ tiles of the strip for the three cases are illustrated in 
        Figure~\ref{fig:bottom-row-if-bottom-labels-is-tau11i}.
\end{proof}

\begin{figure}
    \begin{center}
    \begin{tikzpicture}[scale=1.25]
        \tikzstyle{every node}=[font=\footnotesize]
    \begin{scope}
        \tileJunctionInsideIXXI{0}{0}{011}{*}{*}{01\nbar}{gray!50}{gray!50}
        \draw[draw=white,fill=white]   (0,0) -- (0,1) -- (1,1) -- cycle;
        \draw[draw=black,dashed,thick] (0,0) -- (0,1) -- (1,1) -- cycle;
        \tilelabelinsidescopeH{\ourColorYellow}{1}{0}{012}{112}{011}{11\nbar}
        \tilelabelinsidescopeH{\ourColorYellow}{2}{0}{013}{112}{012}{11\nbar}
        \node at (3.5,0.5) {$\dots$};
        \tilelabelinsidescopeH{\ourColorYellow}{4}{0}{01i}{112}{01\iunder}{11\nbar}
        \tilelabelinsidescopeH{\ourColorYellow}{5}{0}{01\ibar}{112}{01i}{11\nbar}
        \tilelabelinsidescopeH{\ourColorYellow}{6}{0}{\phantom{01(i+2)}}{112}{01\ibar}{11\nbar}
        \node at (7.5,0.5) {$\dots$};
        \tilelabelinsidescopeH{\ourColorYellow}{8}{0}{01n}{112}{01\nunder}{11\nbar}
        \tilelabelinsidescopeH{\ourColorYellow}{9}{0}{01\nbar}{112}{01n}{11\nbar}
        \node[left,text width=4cm] at (-.15,.5)   {horizontal strip with\\bottom word\\
                                           $\tau_n(000) = 01\nbar\cdot (11\nbar)^n$};
        \draw[yshift=-3mm,|-|] (1,0) -- node[fill=white] {$(11\nbar)^n$}       (10,0);
        \draw[yshift=3mm,|-|] (1,1) -- node[fill=white] {$(112)^{n}$}   (10,1);
    \end{scope}
    \begin{scope}[yshift=-23mm]
        \tileJunctionInsideOXXI{0}{0}{001}{*}{*}{01n}{gray!50}{gray!50}
        \draw[draw=white,fill=white]   (0,0) -- (0,1) -- (1,1) -- cycle;
        \draw[draw=black,dashed,thick] (0,0) -- (0,1) -- (1,1) -- cycle;
        \tilelabelinsidescopeH{\ourColorBlue}{1}{0}{002}{111}{001}{11n}
        \tilelabelinsidescopeH{\ourColorBlue}{2}{0}{003}{111}{002}{11n}
        \node at (3.5,0.5) {$\dots$};
        \tilelabelinsidescopeH{\ourColorBlue}  {4}{0}{00i}{111}{01\iunder}{11n}
        \tilelabelinsideHgreen                 {5}{0}{01\ibar}{111}{00i}{11\nbar}
        \tilelabelinsidescopeH{\ourColorYellow}{6}{0}{\phantom{01(i+2)}}{112}{01\ibar}{11\nbar}
        \node at (7.5,0.5) {$\dots$};
        \tilelabelinsidescopeH{\ourColorYellow}{8}{0}{01n}{112}{01\nunder}{11\nbar}
        \tilelabelinsidescopeH{\ourColorYellow}{9}{0}{01\nbar}{112}{01n}{11\nbar}
        \node[left,text width=4cm] at (-.15,.5)   {horizontal strip with\\bottom word\\
                                                  $\tau_n(00i)$ with $1\leq i\leq n$};
        \draw[yshift=-3mm,|-|] (1,0) -- node[fill=white] {$(11n)^{i-1}$}        (5,0);
        \draw[yshift=-3mm,|-|] (5,0) -- node[fill=white] {$(11\nbar)^{n+1-i}$} (10,0);
        \draw[yshift=3mm,|-|] (1,1) -- node[fill=white] {$(111)^{i}$}        (6,1);
        \draw[yshift=3mm,|-|] (6,1) -- node[fill=white] {$(112)^{n-i}$} (10,1);
    \end{scope}
    \begin{scope}[yshift=-46mm]
        \tileJunctionInsideOXXI{0}{0}{001}{*}{*}{01n}{gray!50}{gray!50}
        \draw[draw=white,fill=white]   (0,0) -- (0,1) -- (1,1) -- cycle;
        \draw[draw=black,dashed,thick] (0,0) -- (0,1) -- (1,1) -- cycle;
        \tilelabelinsidescopeH{\ourColorBlue}{1}{0}{002}{111}{001}{11n}
        \tilelabelinsidescopeH{\ourColorBlue}{2}{0}{003}{111}{002}{11n}
        \node at (3.5,0.5) {$\dots$};
        \tilelabelinsidescopeH{\ourColorBlue}{4}{0}{00i}{111}{00\iunder}{11n}
        \tilelabelinsidescopeH{\ourColorBlue}{5}{0}{00\ibar}{111}{00i}{11n}
        \tilelabelinsidescopeH{\ourColorBlue}{6}{0}{\phantom{00(i+2)}}{111}{00\ibar}{11n}
        \node at (7.5,0.5) {$\dots$};
        \tilelabelinsidescopeH{\ourColorBlue}{8}{0}{00n}{111}{00\nunder}{11n}
        \tilelabelinsidescopeH{\ourColorBlue}{9}{0}{00\nbar}{111}{00n}{11n}
        \node[left,text width=4cm] at (-.15,.5)   {horizontal strip with\\bottom word\\
                                          $\tau_n(00\nbar) = 01n\cdot (11n)^n$};
        \draw[yshift=-3mm,|-|] (1,0) -- node[fill=white] {$(11n)^n$}       (10,0);
        \draw[yshift=3mm,|-|] (1,1) -- node[fill=white] {$(111)^{n}$}   (10,1);
    \end{scope}
    \end{tikzpicture}
    \end{center}
\caption{
Horizontal strip with bottom word $\tau_n(00i)$ with $0\leq i\leq n$.
}
\label{fig:bottom-row-if-bottom-labels-is-tau00i}
\end{figure}

\begin{figure}
    \begin{center}
    \begin{tikzpicture}[scale=1.25]
        \tikzstyle{every node}=[font=\footnotesize]
    \begin{scope}
        \tileJunctionInsideOXXO{0}{0}{000}{*}{*}{00n}{gray!50}{gray!50}
        \draw[draw=white,fill=white]   (0,0) -- (0,1) -- (1,1) -- cycle;
        \draw[draw=black,dashed,thick] (0,0) -- (0,1) -- (1,1) -- cycle;
        \tilelabelinsidescopeH{\ourColorBlue}{1}{0}{001}{111}{000}{11n}
        \tilelabelinsidescopeH{\ourColorBlue}{2}{0}{002}{111}{001}{11n}
        \node at (3.5,0.5) {$\dots$};
        \tilelabelinsidescopeH{\ourColorBlue}  {4}{0}{00\iunder}{111}{\phantom{01(i-2)}}{11n}
        \tilelabelinsideHgreen                 {5}{0}{01i}{111}{00\iunder}{11\nbar}
        \tilelabelinsidescopeH{\ourColorYellow}{6}{0}{01\ibar}{112}{01i}{11\nbar}
        \node at (7.5,0.5) {$\dots$};
        \tilelabelinsidescopeH{\ourColorYellow}{8}{0}{01\nunder}{112}{\phantom{01(n-2)}}{11\nbar}
        \tilelabelinsidescopeH{\ourColorYellow}{9}{0}{01n}{112}{01\nunder}{11\nbar}
        \node[left,text width=4cm] at (-.15,.5)   {horizontal strip with\\bottom word\\
                                                  $\tau_n(01i)$ with $1\leq i\leq\nbar$}; 
        \draw[yshift=-3mm,|-|] (1,0) -- node[fill=white] {$(11n)^{i-1}$}        (5,0);
        \draw[yshift=-3mm,|-|] (5,0) -- node[fill=white] {$(11\nbar)^{n+1-i}$} (10,0);
        \draw[yshift=3mm,|-|] (1,1) -- node[fill=white] {$(111)^{i}$}        (6,1);
        \draw[yshift=3mm,|-|] (6,1) -- node[fill=white] {$(112)^{n-i}$} (10,1);
    \end{scope}
    \begin{scope}[yshift=-23mm]
        \tileJunctionInsideOXXO{0}{0}{000}{*}{*}{00n}{gray!50}{gray!50}
        \draw[draw=white,fill=white]   (0,0) -- (0,1) -- (1,1) -- cycle;
        \draw[draw=black,dashed,thick] (0,0) -- (0,1) -- (1,1) -- cycle;
        \tilelabelinsidescopeH{\ourColorBlue}{1}{0}{001}{111}{000}{11n}
        \tilelabelinsidescopeH{\ourColorBlue}{2}{0}{002}{111}{001}{11n}
        \node at (3.5,0.5) {$\dots$};
        \tilelabelinsidescopeH{\ourColorBlue}  {4}{0}{00\iunder}{111}{\phantom{01(i-2)}}{11n}
        \tilelabelinsidescopeH{\ourColorBlue}  {5}{0}{00i}      {111}{01\iunder}{11n}
        \tilelabelinsidescopeH{\ourColorBlue}  {6}{0}{00\ibar}  {111}{01i}{11n}
        \node at (7.5,0.5) {$\dots$};
        \tilelabelinsidescopeH{\ourColorBlue}{8}{0}{00\nunder}{111}{\phantom{01(n-2)}}{11n}
        \tilelabelinsidescopeH{\ourColorBlue}{9}{0}{00n}      {111}{01\nunder}{11n}
        \node[left,text width=4cm] at (-.15,.5)   {horizontal strip with\\bottom word\\
                                          $\tau_n(01\nbar) = 00n\cdot (11n)^n$}; 
        \draw[yshift=-3mm,|-|] (1,0) -- node[fill=white] {$(11n)^n$}       (10,0);
        \draw[yshift=3mm,|-|] (1,1) -- node[fill=white] {$(111)^{n}$}   (10,1);
    \end{scope}
    \end{tikzpicture}
    \end{center}
\caption{
Horizontal strip with bottom word $\tau_n(01i)$ with $1\leq i\leq n+1$.
}
\label{fig:bottom-row-if-bottom-labels-is-tau01i}
\end{figure}

\begin{figure}
    \begin{center}
    \begin{tikzpicture}[scale=1.25]
        \tikzstyle{every node}=[font=\footnotesize]
    \begin{scope}
        \tileJunctionInsideIXXI{0}{0}{011}{*}{*}{01\nbar}{gray!50}{gray!50}
        \draw[draw=white,fill=white]   (0,0) -- (0,1) -- (1,1) -- cycle;
        \draw[draw=black,dashed,thick] (0,0) -- (0,1) -- (1,1) -- cycle;
        \tilelabelinsidescopeH{\ourColorYellow}{1}{0}{012}{112}{011}{11\nbar}
        \tilelabelinsidescopeH{\ourColorYellow}{2}{0}{013}{112}{012}{11\nbar}
        \node at (3.5,0.5) {$\dots$};
        \tilelabelinsidescopeH{\ourColorYellow}{4}{0}{01i}{112}{01\iunder}{11\nbar}
        \tilelabelinsidescopeH{\ourColorYellow}{5}{0}{01\ibar}{112}{01i}{11\nbar}
        \tilelabelinsidescopeH{\ourColorYellow}{6}{0}{\phantom{01(i+2)}}{112}{01\ibar}{11\nbar}
        \node at (7.5,0.5) {$\dots$};
        \tilelabelinsidescopeH{\ourColorYellow}{8}{0}{01n}{112}{01\nunder}{11\nbar}
        \phantom{\tilelabelinsidescopeH{\ourColorYellow}{9}{0}{01\nbar}{112}{01n}{11\nbar}}
        \node[left,text width=4cm] at (-.15,.5)   {horizontal strip with\\bottom word\\
                                            $\tau_n(111)=01\nbar\cdot(11\nbar)^{n-1}$};
        \draw[yshift=-3mm,|-|] (1,0) -- node[fill=white] {$(11\nbar)^{n-1}$} (9,0);
        \draw[yshift=3mm,|-|] (1,1) -- node[fill=white] {$(112)^{n-1}$}   (9,1);
    \end{scope}
    \begin{scope}[yshift=-23mm]
        \tileJunctionInsideOXXI{0}{0}{001}{*}{*}{01n}{gray!50}{gray!50}
        \draw[draw=white,fill=white]   (0,0) -- (0,1) -- (1,1) -- cycle;
        \draw[draw=black,dashed,thick] (0,0) -- (0,1) -- (1,1) -- cycle;
        \tilelabelinsidescopeH{\ourColorBlue}{1}{0}{002}{111}{001}{11n}
        \node at (2.5,0.5) {$\dots$};
        \tilelabelinsidescopeH{\ourColorBlue}  {3}{0}{00\iunder}{111}{\phantom{01(i-2)}}{11n}
        \tilelabelinsideHgreen                 {4}{0}{01i}{111}{00\iunder}{11\nbar}
        \tilelabelinsidescopeH{\ourColorYellow}{5}{0}{01\ibar}{112}{01i}{11\nbar}
        \node at (7.0,0.5) {$\dots$};
        \tilelabelinsidescopeH{\ourColorYellow}{8}{0}{01n}{112}{01\nunder}{11\nbar}
        \phantom{\tilelabelinsidescopeH{\ourColorYellow}{9}{0}{01\nbar}{112}{01n}{11\nbar}}
        \node[left,text width=4cm] at (-.15,.5)   {horizontal strip with\\bottom word\\
                                $\tau_n(11i)$ with $2\leq i\leq n$};
        \draw[yshift=-3mm,|-|] (1,0) -- node[fill=white] {$(11n)^{i-2}$}        (4,0);
        \draw[yshift=-3mm,|-|] (4,0) -- node[fill=white] {$(11\nbar)^{n+1-i}$}  (9,0);
        \draw[yshift=3mm,|-|] (1,1) -- node[fill=white] {$(111)^{i-1}$}           (5,1);
        \draw[yshift=3mm,|-|] (5,1) -- node[fill=white] {$(112)^{n-i}$}         (9,1);
    \end{scope}
    \begin{scope}[yshift=-46mm]
        \tileJunctionInsideOXXI{0}{0}{001}{*}{*}{01n}{gray!50}{gray!50}
        \draw[draw=white,fill=white]   (0,0) -- (0,1) -- (1,1) -- cycle;
        \draw[draw=black,dashed,thick] (0,0) -- (0,1) -- (1,1) -- cycle;
        \tilelabelinsidescopeH{\ourColorBlue}{1}{0}{002}{111}{001}{11n}
        \tilelabelinsidescopeH{\ourColorBlue}{2}{0}{003}{111}{002}{11n}
        \node at (3.5,0.5) {$\dots$};
        \tilelabelinsidescopeH{\ourColorBlue}{4}{0}{00i}{111}{00\iunder}{11n}
        \tilelabelinsidescopeH{\ourColorBlue}{5}{0}{00\ibar}{111}{00i}{11n}
        \tilelabelinsidescopeH{\ourColorBlue}{6}{0}{\phantom{00(i+2)}}{111}{00\ibar}{11n}
        \node at (7.5,0.5) {$\dots$};
        \tilelabelinsidescopeH{\ourColorBlue}{8}{0}{00n}{111}{00\nunder}{11n}
        \phantom{\tilelabelinsidescopeH{\ourColorBlue}{9}{0}{00\nbar}{111}{00n}{11n}}
        \node[left,text width=4cm] at (-.15,.5)   {horizontal strip with\\bottom word\\
                                          $\tau_n(11\nbar) = 01n\cdot (11n)^{n-1}$};
        \draw[yshift=-3mm,|-|] (1,0) -- node[fill=white] {$(11n)^{n-1}$} (9,0);
        \draw[yshift=3mm,|-|]  (1,1) -- node[fill=white] {$(111)^{n-1}$}   (9,1);
    \end{scope}
    \end{tikzpicture}
    \end{center}
\caption{
Horizontal strip with bottom word $\tau_n(11i)$ with $1\leq i\leq n+1$.
}
\label{fig:bottom-row-if-bottom-labels-is-tau11i}
\end{figure}

Since $\Tcal_n=\widehat{\Tcal_n}$, 
Lemma~\ref{lem:unique-horizontal-strip} has a symmetric version
describing the vertical strip of tiles from $\Tcal_n'$ with left labels equal
to $\tau_n(u)$ for some $u\in V_n$.
Lemma~\ref{lem:unique-horizontal-strip} and its symmetric version
can be used together to construct valid rectangular patterns with
external boundaries given by the images under the map $\tau_n$;
see Figure~\ref{fig:desusbsitut-to-antigreen}.

\begin{lemma}\label{lem:Tcaln'-there-exists-rect-pattern}
    Let $\alpha,\beta,u,v\in V_n$.
    If
    $\raisebox{-6mm}{
    \begin{tikzpicture}[auto,scale=.5]
        \tikzstyle{every node}=[font=\footnotesize]
        \tile{white}{0}{0}{\alpha}{\beta}{u}{v}
    \end{tikzpicture}}
    \in\Tcal_n'$,
    then there exists a unique valid rectangular pattern with tiles
    in $\Tcal_n'$
    whose right, top, left and bottom labels
    are respectively
    $\tau_n(\alpha)$, $\tau_n(\beta)$, $\tau_n(u)$ and $\tau_n(v)$.
\end{lemma}

\begin{proof}
    Let $u,v\in V_n$.
    For every tile in $\Tcal_n'$, the left label starts with $0$ if and
    only if the right label starts with 0 and, symmetrically,
    the bottom label starts with 0 if and only if the top label starts with 0.
    Since we have $\tau_n(V_n)\subset\{00n,01n,01\nbar\}\cdot\{11n,11\nbar\}^*$,
    any valid rectangular pattern with tiles in $\Tcal_n'$
    whose sequence of bottom labels is $\tau_n(v)$
    and sequence of left labels is $\tau_n(u)$
    can be split into four disjoint regions:
    a junction tile at the bottom left corner,
    a row of horizontal stripe tiles at the bottom,
    a column of vertical stripe tiles on the left and a rectangular pattern
    of white tiles for the remaining rectangle;
    see Figure~\ref{fig:return-block-to-junction-tile-copy}.
\begin{figure}[h]
\begin{center}
    \begin{tikzpicture}[scale=1.2]
    \tileJunctionBackground{0}{0}{gray!50}{gray!50}{gray!50}{gray!50}
    \draw[draw=none,fill=gray!50] (\dt,1) rectangle (1-\dt,5);
    \draw[draw=none,fill=gray!50] (1,\dt) rectangle (5,1-\dt);
    \draw (0,0) rectangle (5,5);
    \draw (1,0) -- (1,5);
    \draw (0,1) -- (5,1);
    \node[align=center,text width=3cm] at (3,3) {\textbf{white}\\\textbf{tiles}};
    \node[align=center,text width=3cm] at (3,.5) {\textbf{bottom row}};
    \node[rotate=90,align=center,text width=3cm] at (.5,3) {\textbf{left column}};
    \draw (0.33,.67) -- ++ (-.7,.3) node[left] {\textbf{junction tile}};
    \node[left] at (0,2.5) {$\tau_n(u)$};
    \node[below] at (2.5,0) {$\tau_n(v)$};
    \node[right] at (5,2.5) {$\tau_n(\alpha)$};
    \node[above] at (2.5,5) {$\tau_n(\beta)$};
    \node[rotate=90,below] at (1,3) {$(111)^*(112)^*$};
    \node[above] at (3,1) {$(111)^*(112)^*$};
    \draw[dashed,thick] (0,0) -- (1,1);
\end{tikzpicture}
\end{center}
    \caption{The global shape of a rectangular pattern whose
    sequence of bottom labels is $\tau_n(v)$
    and sequence of left labels is $\tau_n(u)$.
    The pattern is split into four disjoint parts: the junction tile,
    the left column, the bottom row and the white tiles.
    }
    \label{fig:return-block-to-junction-tile-copy}
\end{figure}

    (Existence)
    Let
    $u,v,\alpha,\beta\in V_n$ be such that
    $t=\raisebox{-6mm}{
    \begin{tikzpicture}[auto,scale=.5]
        \tikzstyle{every node}=[font=\footnotesize]
        \tile{white}{0}{0}{\alpha}{\beta}{u}{v}
    \end{tikzpicture}}
    \in\Tcal_n'$.
    First, we show that the junction tile at the bottom left corner
    of the rectangular pattern with bottom labels $\tau_n(v)$
    and left labels $\tau_n(u)$
    is one of the 9 junction tile in $\Tcal_n'$. 
    For every $u,v\in V_n$, we have
    $\tau_n(u),\tau_n(v)\in\{00n,01n,01\nbar\}\cdot(V_n)^*$.
    For every $x,y\in\{00n,01n,01\nbar\}$, there exists
    a unique junction tile in $\Tcal_n'$ with bottom label $x$ and left label $y$.

    It remains to show the existence of tiles from $\Tcal_n'$ to cover
    the bottom row, the left column and the region of white tiles
    while respecting the label constraints;
    see Figure~\ref{fig:return-block-to-junction-tile-copy}.
    Again, we proceed case by case.

    Suppose that $t$ is a junction tile in $\Tcal_n'$,
    that is $u,v\in\{00n,01n,01\nbar\}$.
    We have $|\tau_n(u)|=|\tau_n(v)|=n+1$.
    In order to formalize the argument that follows, it is practical to define
    the following two maps on the subset $\{00n,01n,01\nbar\}\subset V_n$:
    \[
        \begin{array}{rccl}
            \sigma:&\{00n,01n,01\nbar\}&\to& V_n\\
            &00n    &\mapsto& 01\nbar,\\ 
            &01n    &\mapsto& 01n,\\ 
            &01\nbar&\mapsto& 00n,
        \end{array}
        \quad
        \text{ and }
        \quad
        \begin{array}{rccl}
            \mu:&\{00n,01n,01\nbar\}&\to& V_n\\
            &00n    &\mapsto& 000,\\ 
            &01n    &\mapsto& 001,\\ 
            &01\nbar&\mapsto& 011.
        \end{array}
    \]
    Notice that $\alpha=\mu(v)$ and $\beta=\mu(u)$
    and $\sigma$ is an involution.
    Also, if $v\in \{00n,01n,01\nbar\}$, then
    $\tau_n\circ\mu(v)=\sigma(v)\cdot(11\nbar)^n$.
    From Lemma~\ref{lem:unique-horizontal-strip},
    there exists a unique choice of tiles for the bottom row whose sequence of
    top labels is $(111)^n$ and right-most right label is 
    $01\nbar$ if $v=00n$,
    $01n$ if $v=01n$,
    $00n$ if $v=01\nbar$.
    In other words, the right-most right label of the bottom row is $\sigma(v)$.
    Symmetrically,
    there exists a unique choice of tiles for the left column whose sequence of
    right labels is $(111)^n$ and top-most top label is $\sigma(u)$.
    Since the bottom row is of length $n$,
    and white tiles increase the last digit by one,
    the remaining region can be uniquely filled with white tiles
    such that the sequence of right labels of the rectangular
    pattern is $\sigma(v)\cdot (11\nbar)^n=\tau_n\circ\mu(v)=\tau_n(\alpha)$.
    Symmetrically, the sequence of top labels of the rectangular pattern is
    $\sigma(u)\cdot (11\nbar)^n=\tau_n\circ\mu(u)=\tau_n(\beta)$.

    Suppose that
    $t=\raisebox{-6mm}{
    \begin{tikzpicture}[auto,scale=.5]
        \tikzstyle{every node}=[font=\footnotesize]
        \tile{white}{0}{0}{\alpha}{\beta}{u}{v}
    \end{tikzpicture}}
    =
    \raisebox{-6mm}{
    \begin{tikzpicture}[auto,scale=.5]
        \tikzstyle{every node}=[font=\footnotesize]
        \tile{white}{0}{0}{11\jbar}{11\ibar}{11j}{11i}
    \end{tikzpicture}}$ 
    is a white tile in $\Tcal_n'$,
    that is, $u=11j$ with $1\leq j\leq n$
    and $v=11i$ with $1\leq i\leq n$.
    Also $\alpha=11\jbar$ and $\beta=11\ibar$.
    We have $|\tau_n(u)|=|\tau_n(v)|=n$.
        From
            Lemma~\ref{lem:unique-horizontal-strip}, there exists a unique
            choice of tiles for the bottom row whose sequence of top labels is
            $(111)^{i-1}\cdot(112)^{n-i}$ and right-most right label is $01n$.
        From
            a symmetric version of Lemma~\ref{lem:unique-horizontal-strip},
            there exists a unique choice of tiles for the left column whose sequence of
            right labels is $(111)^{j-1}\cdot(112)^{n-j}$
            and top-most top label is $01n$.
    The remaining region can be uniquely filled with white tiles.
    In this case, the sequence of right labels of the rectangular
    pattern is $01n\cdot(11n)^{j-1}\cdot(11\nbar)^{n-j}=\tau_n(11\jbar)=\tau_n(\alpha)$.
    Symmetrically, the sequence of top labels of the rectangular
    pattern is $01n\cdot(11n)^{i-1}\cdot(11\nbar)^{n-i}=\tau_n(11\ibar)=\tau_n(\beta)$.

    Suppose that $t$ is a horizontal stripe tile in $\Tcal_n'$.
    We have
    $u=00j\text{ with }0\leq j\leq n$
    or $u=01j\text{ with }1\leq j\leq n$.
    Also $v\in \{11n,11\nbar\}$.
    Let
    \[
        \beta=
        \begin{cases}
            111 & \text{ if } u=00j\text{ with }0\leq j\leq n,\\
            112 & \text{ if } u=01j\text{ with }1\leq j\leq n,
        \end{cases}
        \quad \text{ and } \quad
        \alpha=
        \begin{cases}
        00\jbar & \text{ if } v=11n,\\
        01\jbar & \text{ if } v=11\nbar.
        \end{cases}
    \]
    Also, $|\tau_n(u)|=n+1$ and $|\tau_n(v)|=n$.
    There are two cases for $v$:
    \begin{itemize}
        \item If $v=11n$, then from Lemma~\ref{lem:unique-horizontal-strip},
            there exists a unique choice of tiles for the bottom row whose
            sequence of top labels is $(111)^{n-1}$ and right-most right label
            is $01n$.
        \item If $v=11\nbar$, then from
            Lemma~\ref{lem:unique-horizontal-strip}, there exists a unique
            choice of tiles for the bottom row whose sequence of top labels is
            $(111)^{n-1}$ and right-most right label is $00n$.
    \end{itemize}
    There are two cases for $u$:
    \begin{itemize}
        \item If $u=00j$ with $0\leq j\leq n$, then from the symmetric version of
            Lemma~\ref{lem:unique-horizontal-strip}, there exists a unique
            choice of tiles for the left column whose sequence of right labels
            is $(111)^{j}\cdot(112)^{n-j}$ and top-most top label is $01\nbar$.
        \item If $u=01j$ with $1\leq j\leq n$, then from the symmetric version
            of Lemma~\ref{lem:unique-horizontal-strip}, there exists a unique
            choice of tiles for the left column whose sequence of right labels
            is $(111)^{j}\cdot(112)^{n-j}$ and top-most top label is $01n$.
    \end{itemize}
    Thus, the remaining region can be uniquely filled with white tiles
    and the sequence of right labels of the rectangular
    pattern is 
    \[
        \tau_n(\alpha)=
        \begin{cases}
        01n\cdot(11n)^{j}\cdot(11\nbar)^{n-j}=\tau_n(00\jbar)
            & \text{ if } v=11n,\\
        00n\cdot(11n)^{j}\cdot(11\nbar)^{n-j}=\tau_n(01\jbar)
            & \text{ if } v=11\nbar.
        \end{cases}
    \]
    Symmetrically, the sequence of top labels of the rectangular
    pattern is 
    \[
        \tau_n(\beta)=
        \begin{cases}
    01\nbar\cdot(11\nbar)^{n-1}=\tau_n(111)
            & \text{ if } u=00j \text{ with } 0\leq j\leq n,\\
    01n\cdot(11\nbar)^{n-1}=\tau_n(112) & \text{ if } u=01j\text{ with }1\leq j\leq n.
        \end{cases}
    \]

    Suppose that $t$ is a vertical stripe tile in $\Tcal_n'$.
    A rectangular pattern respecting the constraints can be obtained by taking
    the image under reflection of the rectangular pattern constructed above for
    when $t$ is a horizontal stripe tile in $\Tcal_n'$.

    (Uniqueness) 
    Uniqueness follows from Lemma~\ref{lem:unicity-rect-pattern}
    and Lemma~\ref{lem:SW-and-NE-deterministic}.
\end{proof}

\begin{figure}[h]
\begin{center}
\includegraphics{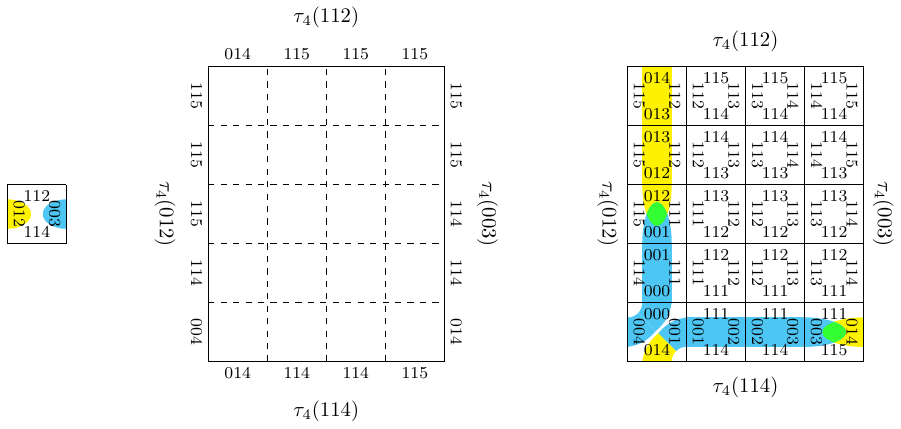}
\end{center}
    \caption{
        Left: some antigren tile in $\Tcal_4'$.
        Middle: the images under $\tau_4$ of the labels of the tile
        form the boundary labels of a rectangle.
        Right: there is a unique rectangular pattern with such boundary words
        and tiles in $\Tcal_4'$.
        As shown in Lemma~\ref{lem:Tcaln'-there-exists-rect-pattern}, this
        holds for every $n\geq1$ and for every tile in $\Tcal_n'$
        allowing to define the map $\omega_n'$.
    }
    \label{fig:desusbsitut-to-antigreen}
\end{figure}

Following Lemma~\ref{lem:Tcaln'-there-exists-rect-pattern},
we define the following map:
\begin{equation}\label{eq:defintion-omegan'}
    \begin{array}{rccc}
        \omega_n':&\Tcal_n'&\to&(\Tcal_n')^{*2}\\
        &\raisebox{-6mm}{
        \begin{tikzpicture}[auto,scale=.5]
            \tikzstyle{every node}=[font=\footnotesize]
            \tile{white}{0}{0}{\alpha}{\beta}{u}{v}
        \end{tikzpicture}}
        &\mapsto&
        \begin{array}{l}
            \text{the unique rectangular }\\
            \text{pattern with external labels}
        \end{array}
        \raisebox{-10mm}{
        \begin{tikzpicture}[auto]
            \tikzstyle{every node}=[font=\footnotesize]
            \tile{white}{0}{0}{\tau_n(\alpha)}{\tau_n(\beta)}{\tau_n(u)}{\tau_n(v)}
        \end{tikzpicture}}.
    \end{array}
\end{equation}
For example, the map $\omega_1'$ is illustrated in Figure~\ref{fig:omega1'}.

\begin{figure}[h]
\begin{center}
    \includegraphics[width=.9\linewidth]{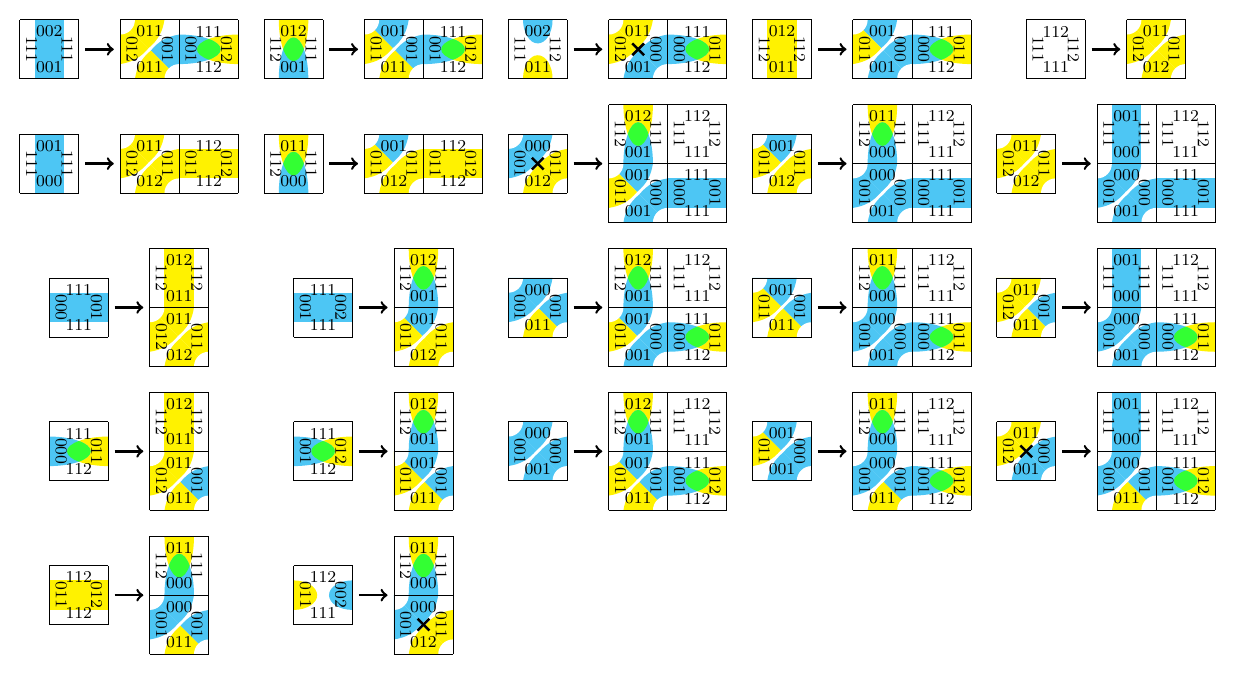}
\end{center}
    \caption{
        The substitution $\omega_1'$.
        An $\times$-mark indicates the tiles in $J_1'\setminus J_1$.
    }
    \label{fig:omega1'}
\end{figure}

\begin{lemma}\label{lem:omegan'-is-2-dim-substitution}
    The map $\omega_n'$ defines a $2$-dimensional substitution
    $\omega_n':\Omega_n'\to\Omega_n'$.
\end{lemma}

\begin{proof}
    From Lemma~\ref{lem:Tcaln'-there-exists-rect-pattern},
    for every tile $t\in\Tcal_n'$,
    the image $\omega_n'(t)$ is a valid rectangular pattern
    over the Wang tiles $\Tcal_n'$.
    Moreover, if $s\odot^1t\in(\Tcal_n')^{*^2}$ is a valid horizontal domino,
    then $\omega_n'(s\odot^1t)$ is a valid rectangular pattern
    over the Wang tiles $\Tcal_n'$.
    Similarly, if $s\odot^2t\in(\Tcal_n')^{*^2}$ is a valid vertical domino,
    then $\omega_n'(s\odot^2t)$ is a valid rectangular pattern
    over the Wang tiles $\Tcal_n'$.
    Thus, if $y\in\Omega_n'$ is a valid configuration, then
    $\omega_n'(y)$ is also a valid configuration.
    Therefore, $\omega_n'(y)\in\Omega_n'$.
\end{proof}

\subsection{A substitution $\omega_n$ for the tiles in $\Tcal_n$}
Not all tiles of $\Tcal_n'$ appear in the image of a tile under the
substitution $\omega_n'$. For example, it follows from
Lemma~\ref{lem:unique-horizontal-strip} that antigreen stripe tiles do not
appear in the images of tiles under $\omega_n'$.
Therefore, the substitution $\omega_n'$ is not primitive.

As it can be observed in Figure~\ref{fig:omega1'}, 
some tiles in $\Tcal_1'\setminus\Tcal_1$ appear in the images of $\omega_1'$.
Namely, the images of the antigreen tiles under $\omega_1'$ contain
junction tiles in $J_1'\setminus J_1=\{j_1^{0,0,1,1},j_1^{1,1,0,0}\}$.
As shown in the next lemma, this is the only exception.

\begin{lemma}\label{lem:Tcaln-there-exists-rect-pattern}
    Let $n\geq1$ be an integer and $t\in\Tcal_n'$.
    The pattern $\omega_n'(t)$ contains a tile in $\Tcal_n'\setminus\Tcal_n$
    if and only if
    $n=1$ and $t$ is an antigreen tile.
\end{lemma}

\begin{proof}
    Let $n\geq1$ be an integer.
    ($\impliedby$) 
    If $n=1$, the set of antigreen tiles in $\Tcal_1'$ 
    is $A_1\cup\widehat{A_1}=\{a_1^1,\widehat{a_1^1}\}$.
    In Figure~\ref{fig:omega1'}, we observe that 
    $\omega_1'(a_1^1)$ contains the junction tile
    $j_1^{1,1,0,0}\in\Tcal_1'\setminus\Tcal_1$ and 
    $\omega_1'(\widehat{a_1^1})$ contains the junction tile
    $j_1^{0,0,1,1}\in\Tcal_1'\setminus\Tcal_1$.

    ($\implies$)
    Let $t=\raisebox{-6mm}{
        \begin{tikzpicture}[auto,scale=.5]
            \tikzstyle{every node}=[font=\footnotesize]
            \tile{white}{0}{0}{\alpha}{\beta}{u}{v}
        \end{tikzpicture}}\in\Tcal_n'$.
    The labels of the boundary of $\omega_n'(t)$ are
    $\raisebox{-10mm}{
    \begin{tikzpicture}[auto]
        \tikzstyle{every node}=[font=\footnotesize]
        \tile{white}{0}{0}{\tau_n(\alpha)}{\tau_n(\beta)}{\tau_n(u)}{\tau_n(v)}
    \end{tikzpicture}}$.
    Suppose that the pattern $\omega_n'(t)$ contains a tile in $\Tcal_n'\setminus\Tcal_n$.
    We have $v\in V_n\setminus\{00\nbar\}$.
    From Lemma~\ref{lem:unique-horizontal-strip},
    the bottom row of the pattern $\omega_n'(t)$ does not contain the last blue tile.
    Also, $u\in V_n\setminus\{00\nbar\}$.
    From the symmetric version of Lemma~\ref{lem:unique-horizontal-strip},
    the left column of the pattern $\omega_n'(t)$ does not contain the last blue tile.
    From Lemma~\ref{lem:unique-horizontal-strip}, the 
    pattern $\omega_n'(t)$ does not contain any antigreen (vertical or
    horizontal) stripe tile.
    Therefore, the pattern $\omega_n'(t)$ must contain a junction tile in 
    $J_n'\setminus J_n=\{j_n^{1,1,0,0},j_n^{0,0,1,1}\}$.

    Suppose that $\omega_n'(t)$ contains the junction tile $j_n^{1,1,0,0}$.
    Therefore, we must have $\tau_n(v)\in 01\nbar\cdot(V_n)^*$ and
    $\tau_n(u)\in 00n\cdot(V_n)^*$.
    Thus, $v\in\{000,111\}$ and
    $u=01j$ with $1\leq j\leq n$.
    We proceed case by case.
    \begin{itemize}
        \item Assume $v=000$. 
            The only tiles $t\in\Tcal_n'$ with bottom label $v=000$ is a 
            blue or green vertical stripe tile whose left label is $u=11n$ or
            $u=11\nbar$, a contradiction.
        \item Assume $v=111$ and $n>1$.
            The only tiles $t\in\Tcal_n'$ with bottom label $v=111$ is a 
            white tile whose left label is $u=11i$ 
            with $1\leq i\leq n$, a contradiction.
        \item Assume $v=111$ and $n=1$.
            The only tiles $t\in\Tcal_1'$ with bottom label $v=111$ is a 
            white tile whose left label is $u=111$,
            a blue horizontal stripe tile
            whose left label is $000$ or $001$,
            or an antigreen tile $a_1^1$ whose left label is $011$.
            Only the antigreen tile does not yield a contradiction with the
            value of $u$ given above. Thus, $t=a_1^1$.
    \end{itemize}
    Symmetrically, if $\omega_n'(t)$ contains the junction tile $j_n^{0,0,1,1}$,
    we conclude that $n=1$ and $t=\widehat{a_1^1}$.
\end{proof}

A consequence of Lemma~\ref{lem:Tcaln-there-exists-rect-pattern}
is that if $n\geq2$ and $t\in\Tcal_n'$,
then the pattern $\omega_n'(t)$ contains only tiles from $\Tcal_n$.
Also for every $n\geq1$ and $t\in\Tcal_n$, 
the pattern $\omega_n'(t)$ contains only tiles from $\Tcal_n$.
Thus, it becomes natural to restrict the substitution $\omega_n'$
to the set $\Tcal_n$.
We obtain the following map $\omega_n=\omega_n'|_{\Tcal_n}$:
\begin{equation}\label{eq:defintion-omegan}
    \begin{array}{rccc}
        \omega_n:&\Tcal_n&\to&(\Tcal_n)^{*2}\\
        &\raisebox{-6mm}{
        \begin{tikzpicture}[auto,scale=.5]
            \tikzstyle{every node}=[font=\footnotesize]
            \tile{white}{0}{0}{\alpha}{\beta}{u}{v}
        \end{tikzpicture}}
        &\mapsto&
        \begin{array}{l}
            \text{the unique rectangular }\\
            \text{pattern with external labels}
        \end{array}
        \raisebox{-10mm}{
        \begin{tikzpicture}[auto]
            \tikzstyle{every node}=[font=\footnotesize]
            \tile{white}{0}{0}{\tau_n(\alpha)}{\tau_n(\beta)}{\tau_n(u)}{\tau_n(v)}
        \end{tikzpicture}}
    \end{array}
\end{equation}

The substitutions $\omega_n$ for $n=1,\dots,5$ are illustrated
in
Figure~\ref{fig:W1_sub},
Figure~\ref{fig:W2_sub},
Figure~\ref{fig:W3_sub},
Figure~\ref{fig:W4_sub} and
Figure~\ref{fig:W5_sub}.

\begin{lemma}\label{lem:omegan-is-2-dim-substitution}
    The map $\omega_n$ defines a $2$-dimensional substitution
    $\omega_n:\Omega_n\to\Omega_n$ such that
    $\overline{\omega_n(\Omega_n)}^\sigma\subset\Omega_n$.
\end{lemma}

\begin{proof}
    From Lemma~\ref{lem:omegan'-is-2-dim-substitution},
    the map $\omega_n'$ defines a $2$-dimensional substitution $\Omega_n'\to\Omega_n'$.
    From Lemma~\ref{lem:Tcaln-there-exists-rect-pattern},
    $\omega_n'(x)\in\Omega_n$ for every configuration $x\in\Omega_n$.
    Thus, $\omega_n'(\Omega_n)\subseteq\Omega_n$.
    The restriction of $\omega_n'$ to $\Omega_n$ is $\omega_n$,
    so that $\omega_n(\Omega_n)\subseteq\Omega_n$.
    Since $\Omega_n$ is a subshift, it is closed under the shift.
    Therefore,
    $\overline{\omega_n(\Omega_n)}^\sigma\subset\Omega_n$.
\end{proof}

The goal of the next sections is to show that
$\Omega_n=\overline{\omega_n(\Omega_n)}^\sigma$,
namely that every configuration in $\Omega_n$
can be desubstituted using $\omega_n$.
The proof of this is completed in Section~\ref{sec:self-similar-aperiodic}.
Following the above discussion,
the $2$-dimensional substitution $\omega_n'$ is not primitive,
but we show in Section~\ref{sec:primitive} that
the substitution $\omega_n$ is primitive.

\subsection{A sufficient and necessary condition}

The goal of this section is to show that
the sufficiency in the statement of Lemma~\ref{lem:Tcaln'-there-exists-rect-pattern}
is also a necessity, namely that
every rectangular pattern, with external boundary
labeled by images under $\tau_n$, is obtained from a tile in $\Tcal_n'$.
The precise statement is given in Proposition~\ref{prop:there-exists-rect-pattern-iff}.

For every integer $n\geq1$, let 
\[
    Z_n=\{v_0v_1v_2\in V_n \mid v_0 = 0\}
\]
be the set of vectors of $V_n$ such that the first entry is zero
and let
\[
    M_n=\{v_0v_1v_2\in V_n \mid v_2\geq n\}
\]
be the set of vectors of $V_n$ such that the last entry is $n$ or $\nbar$.

\begin{lemma}\label{lem:other-cases-exists-pattern}
    If
    \begin{equation}\label{eq:complement-to-Q-in-P}
    \begin{aligned}
        (u,v) &\in 
        (\{11\nbar\}\times V_n\setminus Z_n)
            \cup(V_n\setminus Z_n\times \{11\nbar\})\\
            &\qquad\cup(\{00\nbar\}\times M_n\cap Z_n)
                   \cup(M_n\cap Z_n\times \{00\nbar\})\\
            &\qquad\cup(\{00\nbar,01\nbar\}\times M_n\setminus Z_n)
            \cup(M_n\setminus Z_n\times \{00\nbar,01\nbar\}),
    \end{aligned}
    \end{equation}
    then there exists
    a unique valid rectangular pattern with tiles
    in $\Tcal_n'$
    whose right, top, left and bottom labels
    are respectively
    $R$, $T$, $\tau_n(u)$ and $\tau_n(v)$
    for some $R,T\in(V_n)^*$,
    and there is no $(\alpha,\beta)\in V_n\times V_n$ such that
    $R=\tau_n(\alpha)$
    and $T=\tau_n(\beta)$.
\end{lemma}

\begin{proof}
    Suppose that $u\in V_n\setminus Z_n$ and $v=11\nbar$.
    We have $|\tau_n(u)|=|\tau_n(v)|=n$.
    Since $v=11\nbar$, then from
    Lemma~\ref{lem:unique-horizontal-strip}, there exists a unique
    choice of tiles for the bottom row whose sequence of top labels is
    $(111)^{n-1}$ and right-most right label is $00n$.
    There are two cases to consider for $u$:
    \begin{itemize}
        \item If $u=11j$ with $1\leq j\leq n$, then from
            a symmetric version of Lemma~\ref{lem:unique-horizontal-strip},
            there exists a unique choice of tiles for the left column whose sequence of
            right labels is $(111)^{j-1}\cdot(112)^{n-j}$.
        \item If $u=11\nbar$, then from a symmetric version of
            Lemma~\ref{lem:unique-horizontal-strip}, there exists a unique
            choice of tiles for the left column whose sequence of right labels is
            $(111)^{n-1}$. %
    \end{itemize}
    In both cases, the remaining region of the rectangular pattern can be
    uniquely filled with white tiles.  
    The sequence of right labels
    of the rectangular pattern 
    starts with $00n$. Such a sequence cannot be written as 
    an image under the map $\tau_n$,
    because there is no $\alpha\in V_n$ such that $\tau_n(\alpha)$ starts with
    $00n$ and is of length $n$.

    Suppose that $u\in M_n\cap Z_n=\{00n,00\nbar,01n,01\nbar\}$ and $v=00\nbar$.
    We have $|\tau_n(u)|=|\tau_n(v)|=n+1$.
    From Lemma~\ref{lem:unique-horizontal-strip},
    there exists a unique choice of tiles for the bottom row whose sequence of
    top labels is $(111)^n$.
    Since $v=00\nbar$, 
    from a symmetric version of Lemma~\ref{lem:unique-horizontal-strip},
    there exists a unique choice of tiles for the left column whose sequence of
    right labels is $(111)^n$ and top-most top label is $00\nbar$.
    Since the bottom row is of length $n$,
    and white tiles increase the last digit by one,
    the remaining region can be uniquely filled with white tiles.
    Since the sequence of top
    labels of the rectangular pattern starts with $00\nbar$,
    it cannot be written as an image under the map $\tau_n$.

    Suppose that $u\in\{00\nbar,01\nbar\}$ and $v\in M_n\setminus Z_n=\{11n,11\nbar\}$.
    We have $|\tau_n(u)|=n+1$ and $|\tau_n(v)|=n$.
    From Lemma~\ref{lem:unique-horizontal-strip},
    there exists a unique choice of tiles for the bottom row whose sequence of
    top labels is $(111)^n$. %
    Symmetrically,
    there exists a unique choice of tiles for the left column whose sequence of
    right labels is $(111)^n$ and top-most top label is in $\{00n,00\nbar\}$.
    The remaining region can be uniquely filled with white tiles.
    The sequence of top labels is in $\{00n,00\nbar\}\cdot(11\nbar)^{n-1}$.
    Such a sequence cannot be written as 
    an image under the map $\tau_n$,
    because there is no $\alpha\in V_n$ such that $\tau_n(\alpha)$ starts with
    $00n$ or $00\nbar$ and is of length $n$.

    Suppose that $u=11\nbar$ and $v\in V_n\setminus Z_n$,
    or $u=00\nbar$ and $v\in M_n\cap Z_n$,
    or $u\in M_n\setminus Z_n$ and $v\in\{00\nbar,01\nbar\}$.
    A rectangular pattern respecting the constraints can be obtained by taking
    the image under reflection of the rectangular pattern constructed above.
\end{proof}

\begin{proposition}\label{prop:image-of-a-tile}
    Let $u,v\in V_n$.
    There exists a valid rectangular pattern of tiles in $\Tcal_n'$
    whose 
    sequence of bottom labels is $\tau_n(v)$
    and
    sequence of left labels is $\tau_n(u)$
    if and only if
    \begin{equation}\label{eq:condition_return_block}
        (u,v) \in (V_n\setminus Z_n\times V_n\setminus Z_n)
            \cup(M_n\cap Z_n\times M_n\cap Z_n)
            \cup(M_n\setminus Z_n\times Z_n)
            \cup(Z_n\times M_n\setminus Z_n).
    \end{equation}
\end{proposition}

\begin{proof}
    Let $u,v\in V_n$.
    ($\implies$)
    We show the contrapositive, namely that if \eqref{eq:condition_return_block}
    does not hold, then there is no rectangular pattern
    with $\tau_n(u)$ as the sequence of labels on the left
    and $\tau_n(v)$ as the sequence of labels at the bottom.
    If \eqref{eq:condition_return_block} does not hold, then
    \begin{align*}
        (u,v)\in
        &\big((Z_n \times V_n\setminus Z_n)\cup
    (V_n\setminus Z_n \times Z_n)\cup
        (Z_n \times Z_n)\big)\\
        &\qquad\qquad\setminus \big(
                (M_n\cap Z_n\times M_n\cap Z_n)
            \cup(M_n\setminus Z_n\times Z_n)
            \cup(Z_n\times M_n\setminus Z_n)
            \big) \\
    &=(Z_n \times V_n\setminus (M_n\cup Z_n))\cup
    (V_n\setminus (M_n\cup Z_n) \times Z_n)
    \cup(Z_n \times Z_n\setminus M_n)
    \cup(Z_n\setminus M_n \times Z_n).
    \end{align*}
    There are four cases to consider:
    \begin{itemize}
        \item Assume $u\in Z_n$ and $v\in V_n\setminus (M_n\cup Z_n)$.
            We have $v=11j$ with $1\leq j<n$.
            From Lemma~\ref{lem:unique-horizontal-strip}, the bottom row of
            the rectangular pattern has at least one label $112$ on its top.
            Since the difference between 
            the last digit of the top label and
            the last digit of the bottom label 
            of a white tile is 1 and the maximal last digit of a white tile in
            $\Tcal_n$ is $\nbar$, the height of the white tile region is at
            most $n-1$; see Figure~\ref{fig:vertical-sequence-white-tiles}. 
            Thus, $|\tau_n(u)|\leq n$. 
            This is incompatible with $u\in Z_n$, 
            because $u\in Z_n$ implies that $|\tau_n(u)|=n+1$.
        \item Assume $u\in V_n\setminus (M_n\cup Z_n)$ and $v\in Z_n$.
            This case also leads to a contradiction following an argument
            symmetric to the previous one.
        \item Assume $u\in Z_n$ and $v\in Z_n\setminus M_n$.
            We have $v=00j$ with $0\leq j<n$
            or $v=01j$ with $1\leq j<n$.
            In both cases, we have
            from Lemma~\ref{lem:unique-horizontal-strip} that the bottom row of
            the rectangular pattern has at least one label $112$ on its top.
            For the same reason as in the first item,
            the height of the rectangular pattern is $|\tau_n(u)|\leq n$. 
            This is incompatible with $u\in Z_n$, 
            because $u\in Z_n$ implies that $|\tau_n(u)|=n+1$.
        \item Assume $u\in Z_n\setminus M_n$ and $v\in Z_n$.
            This case also leads to a contradiction following an argument
            symmetric to the previous one.
    \end{itemize}
\begin{figure}[h]
    \begin{center}
    \begin{tikzpicture}[scale=1.3]
        \begin{scope}
            \tikzstyle{every node}=[font=\footnotesize]
            \tilelabelinsidescope{white}{0}{0}{*}{113}{*}{112}
            \tilelabelinsidescope{white}{0}{1}{*}{114}{*}{113}
            \node at (0.5,2.5) {$\vdots$};
            \tilelabelinsidescope{white}{0}{3}{*}{11\nbar}{*}{11n}
            \draw [decorate,decoration={brace,amplitude=5pt,raise=4ex}]
                (1,4) -- (1,0) node[midway,xshift=5em]{height $n-1$};
        \end{scope}
        \begin{scope}[xshift=5cm]
            \tikzstyle{every node}=[font=\footnotesize]
            \tilelabelinsidescope{white}{0}{0}{*}{112}{*}{111}
            \tilelabelinsidescope{white}{0}{1}{*}{113}{*}{112}
            \node at (0.5,2.5) {$\vdots$};
            \tilelabelinsidescope{white}{0}{3}{*}{11n}{*}{11\nunder}
            \draw [decorate,decoration={brace,amplitude=5pt,raise=4ex}]
                (1,4) -- (1,0) node[midway,xshift=5em]{height $n-1$};
        \end{scope}
    \end{tikzpicture}
    \end{center}
\caption{
    The height of a valid vertical column made entirely of white tiles 
    from $\Tcal_n$ is at most $n-1$ 
    if the bottom label of the bottom-most tile is $112$
    or if the top label of the top-most tile is $11n$.}
    \label{fig:vertical-sequence-white-tiles}
\end{figure}

    ($\impliedby$)
    Let
    \begin{align*}
        P &= (V_n\setminus Z_n\times V_n\setminus Z_n)
            \cup(M_n\cap Z_n\times M_n\cap Z_n)
            \cup(M_n\setminus Z_n\times Z_n)
            \cup(Z_n\times M_n\setminus Z_n),\\
        Q &= \left\{(u,v)\in V_n\times V_n\;\middle|\; \text{ there exists } \alpha,\beta\in V_n
                            \text{ such that } 
                    \raisebox{-6mm}{
                    \begin{tikzpicture}[auto,scale=.5]
                        \tikzstyle{every node}=[font=\footnotesize]
                        \tile{white}{0}{0}{\alpha}{\beta}{u}{v}
                    \end{tikzpicture}}
                \in\Tcal_n'\right\}.
    \end{align*}
    Notice that $Q\subset P$ and
    \begin{equation}\label{eq:P-minus-Q}
    \begin{aligned}
        P\setminus Q &= 
        (\{11\nbar\}\times V_n\setminus Z_n)
            \cup(V_n\setminus Z_n\times \{11\nbar\})\\
            &\qquad\cup(\{00\nbar\}\times M_n\cap Z_n)
                   \cup(M_n\cap Z_n\times \{00\nbar\})\\
            &\qquad\cup(\{00\nbar,01\nbar\}\times M_n\setminus Z_n)
            \cup(M_n\setminus Z_n\times \{00\nbar,01\nbar\}).
    \end{aligned}
    \end{equation}
    We assume that \eqref{eq:condition_return_block} holds,
    that is, $(u,v)\in P$.
    There are two cases to consider.
    \begin{itemize}
        \item If $(u,v)\in Q$, then,
            from Lemma~\ref{lem:Tcaln'-there-exists-rect-pattern},
            there exists a valid rectangular pattern with tiles
            in $\Tcal_n'$
            whose left and bottom labels
            are respectively $\tau_n(u)$ and $\tau_n(v)$.
        \item If $(u,v)\in P\setminus Q$, then
            using \eqref{eq:P-minus-Q} and 
            Lemma~\ref{lem:other-cases-exists-pattern},
            there exists a valid rectangular pattern with tiles
            in $\Tcal_n'$
            whose left and bottom labels
            are respectively
            $\tau_n(u)$ and $\tau_n(v)$.\qedhere
    \end{itemize}
\end{proof}

\begin{proposition}\label{prop:there-exists-rect-pattern-iff}
    Let $\alpha,\beta,u,v\in V_n$.
    There exists a valid rectangular pattern with tiles
    in $\Tcal_n'$
    whose right, top, left and bottom labels
    are respectively
    $\tau_n(\alpha)$, $\tau_n(\beta)$, $\tau_n(u)$ and $\tau_n(v)$
    if and only if
    $\raisebox{-6mm}{
    \begin{tikzpicture}[auto,scale=.5]
        \tikzstyle{every node}=[font=\footnotesize]
        \tile{white}{0}{0}{\alpha}{\beta}{u}{v}
    \end{tikzpicture}}
    \in\Tcal_n'$.
\end{proposition}

\begin{proof}
    Let $\alpha,\beta,u,v\in V_n$.
    ($\impliedby$)
    The existence of the rectangular pattern was proved
    in Lemma~\ref{lem:Tcaln'-there-exists-rect-pattern}.

    ($\implies$)
    Suppose that there exists a valid rectangular pattern with tiles
    in $\Tcal_n'$
    whose right, top, left and bottom labels
    are respectively
    $\tau_n(\alpha)$, $\tau_n(\beta)$, $\tau_n(u)$ and $\tau_n(v)$.
    From Proposition~\ref{prop:image-of-a-tile},
    $(u,v)$ satisfies \eqref{eq:condition_return_block},
    that is $(u,v)\in P$.
    From Lemma~\ref{lem:other-cases-exists-pattern},
    $(u,v)$ does not satisfies \eqref{eq:complement-to-Q-in-P}
    because all boundary words can be written as the image of $\tau_n$.
    Thus, $(u,v)\notin P\setminus Q$ using \eqref{eq:P-minus-Q}.
    We conclude that $(u,v)\in Q$.
    Thus, there exists $\alpha',\beta'\in V_n$ such that
    $\raisebox{-6mm}{
        \begin{tikzpicture}[auto,scale=.5]
            \tikzstyle{every node}=[font=\footnotesize]
            \tile{white}{0}{0}{\alpha'}{\beta'}{u}{v}
        \end{tikzpicture}}
    \in\Tcal_n'$.
    From Lemma~\ref{lem:Tcaln'-there-exists-rect-pattern}.
    there exists a valid rectangular pattern with tiles
    in $\Tcal_n'$
    whose right, top, left and bottom labels
    are respectively
    $\tau_n(\alpha')$, $\tau_n(\beta')$, $\tau_n(u)$ and $\tau_n(v)$.
    From Lemma~\ref{lem:unicity-rect-pattern}, we must
    have
    $\tau_n(\alpha')=\tau_n(\alpha)$
    and
    $\tau_n(\beta')=\tau_n(\beta)$ because
    $\Tcal_n'$ is SW-deterministic
    from Lemma~\ref{lem:SW-and-NE-deterministic}.
    Since $\tau_n$ is injective over the set $V_n$,
    we have $\alpha=\alpha'$ and $\beta=\beta'$.
\end{proof}

Proposition~\ref{prop:there-exists-rect-pattern-iff}
is used in Lemma~\ref{lem:return-word-description}
in order to desubstitute configurations in $\Omega_n$
over tiles in $\Tcal_n$.
Nevertheless, considering tiles in $\Tcal_n'$ is necessary for
Proposition~\ref{prop:there-exists-rect-pattern-iff}
to hold for every integer $n\geq1$.
Following Lemma~\ref{lem:Tcaln-there-exists-rect-pattern},
Proposition~\ref{prop:there-exists-rect-pattern-iff}
can be restated as follows when $n\geq2$.

\begin{corollary}\label{cor:there-exists-rect-pattern-iff}
    Suppose that $n\geq2$ is an integer and let $\alpha,\beta,u,v\in V_n$.
    There exists a valid rectangular pattern with tiles
    in $\Tcal_n$
    whose right, top, left and bottom labels
    are respectively
    $\tau_n(\alpha)$, $\tau_n(\beta)$, $\tau_n(u)$ and $\tau_n(v)$
    if and only if
    $\raisebox{-6mm}{
    \begin{tikzpicture}[auto,scale=.5]
        \tikzstyle{every node}=[font=\footnotesize]
        \tile{white}{0}{0}{\alpha}{\beta}{u}{v}
    \end{tikzpicture}}
    \in\Tcal_n'$.
\end{corollary}

\begin{proof}
    Let $\alpha,\beta,u,v\in V_n$.
    ($\implies$)
    Follows from Proposition~\ref{prop:there-exists-rect-pattern-iff}
    since $\Tcal_n\subset\Tcal_n'$.

    ($\impliedby$) 
    From Lemma~\ref{lem:Tcaln-there-exists-rect-pattern},
    for every tile $t\in\Tcal_n'$,
    the rectangular pattern $\omega_n'(t)$ satisfies
    the boundary conditions and it contains only the tiles from the set $\Tcal_n$.
\end{proof}

\section{A desubstitution $\Omega_n\leftarrow\Omega'_n$}
\label{sec:desubstitution}

In this section, we decompose configurations 
in $\Omega_n'$ and in $\Omega_n$
into rectangular blocks called return blocks. 
The external boundary labels of
the return blocks within a configuration in $\Omega_n$ 
behave like a new set $\Tcal_n'$ of Wang tiles
which contains $\Tcal_n$ as a subset.

\subsection{Return blocks in the Wang shift $\Omega_n'$}
\label{sec:return-blocks}

In this section, we study some properties of the Wang shift
$\Omega_n'$ defined by the Wang tiles $\Tcal_n'$. 
Since $\Tcal_n \subset \Tcal_n'$ for every $n\geq1$,
we have $\Omega_n\subset\Omega_n'$.
Thus, the properties shown for $\Omega_n'$ also hold for $\Omega_n$.

\begin{figure}[h]
\begin{center}
    \includegraphics{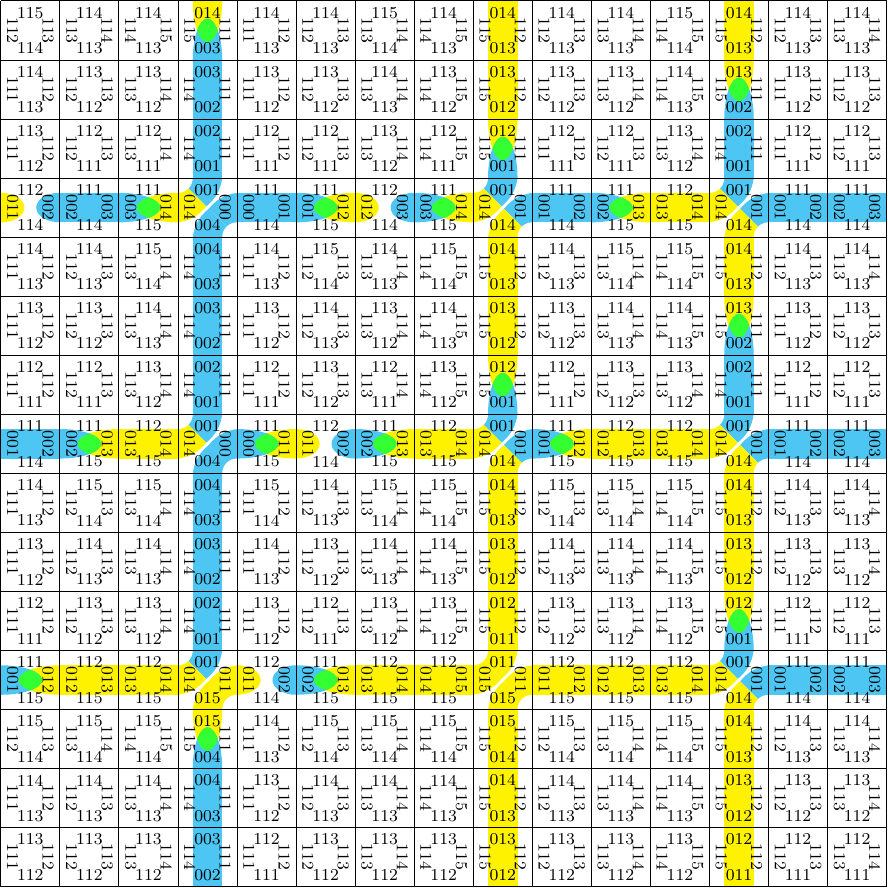}
\end{center}
    \caption{A valid $15\times 15$ pattern using the extended set $\Tcal_4'$ of Wang tiles.
    Note that it contains some antigreen tiles.}
    \label{fig:T4'-15x15}
\end{figure}

A tiling with the set $\Tcal'_4$ is shown in Figure~\ref{fig:T4'-15x15}.
We observe the presence of rows and columns of colored tiles. At the intersection
of these colored rows and columns are junction tiles.
In other words, the set of positions of junction tiles in the figure is the
Cartesian product of two sets. Also, the distance between two consecutive
junction tiles in the same row or column is 4 or 5. In the following lemmas, we
prove that these observations hold in general.

\begin{lemma} \label{lem:distance-junction-in-a-row}
    Let $n\geq1$ be an integer.
    For every valid configuration $c\in\Omega_n'$, 
    the distance between two consecutive occurrences of junction tiles in
    the same row is $n$, $n+1$ or $n+2$.

    Also, the sequence of bottom labels of the tiles between two consecutive junction tiles
    (including the left junction tile but not the right one)
    belongs to $\{00n,01n,01\nbar\} \cdot \{11n,11\nbar\}^*$.
\end{lemma}

\begin{proof}
    The horizontal Rauzy graph restricted to tiles whose vertical edge labels
    are starting with zero is shown in 
    Figure~\ref{fig:horizontal_rauzy_graph_n=5_junction_to_junction}.
    An arc in the horizontal Rauzy graph links two tiles $s\to t$ if and only
    if the right label of tile $s$ is equal to the left label of tile $t$.
    The graph allows to visualize the combinatorial structure between two
    consecutive junction tiles on the same horizontal row within a
    configuration of $\Omega_n'$.

    The right label of a junction tile is $000$, $001$ or $011$,
    which implies that the last digit of the right label of a junction tile is $0$ or $1$.
    The left label of a junction tile is $00n$, $01n$ or $01\nbar$,
    which implies that the last digit of the left label of a junction tile is $n$ or $\nbar$.
    Since the last digit increases by 1 from the left label to the right label
    of every intermediate tile (a tile appearing in between two consecutive
    junction tile in the same row), the number of tiles in between two
    consecutive junction tiles on the same row is at least $n-1$ and at most
    $\nbar-0=n+1$.
    We conclude that the distance (number of edges in the Rauzy graph) between
    two consecutive junction tiles in the same row is $n$, $n+1$ or $n+2$. In
    particular, it is at least $n$.
\begin{figure}[h]
\begin{center}
    \includegraphics{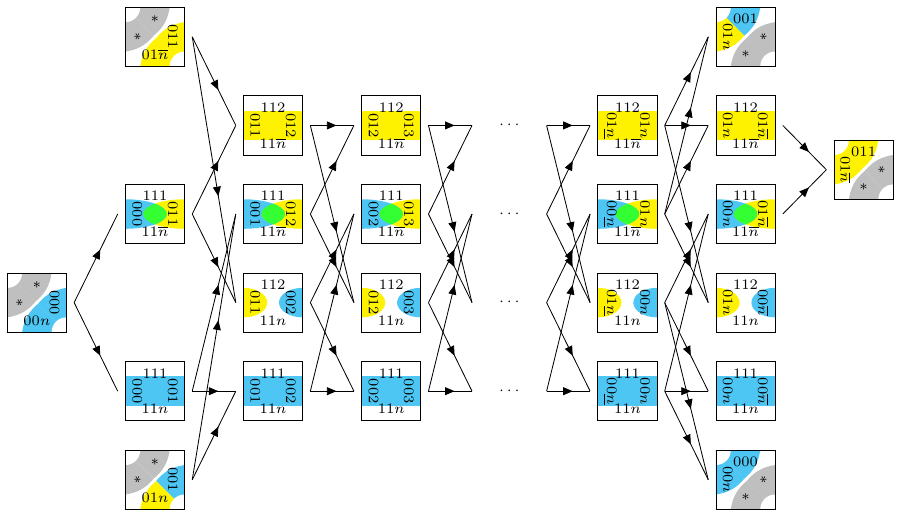}
\end{center}
\caption{Combinatorial structure between two consecutive junction tiles on the same horizontal row 
    within a configuration of $\Omega_n'$.
    The nodes of the graph are placed such that any two tiles appearing in the same
    column have the same last digit for its left or right labels.
    The length of a path from a junction tile to a junction tile is $n$, $n+1$ or $n+2$.}
\label{fig:horizontal_rauzy_graph_n=5_junction_to_junction}
\end{figure}

    The bottom label of a junction tile is in the set $\{00n,01n,01\nbar\}$.
    The bottom label of every intermediate tile is $11n$ or $11\nbar$.
    Therefore, the sequence of bottom labels of the tiles between two consecutive junction tiles
    (including the left junction tile but not the right one)
    belongs to $\{00n,01n,01\nbar\} \cdot \{11n,11\nbar\}^*$;
    see Figure~\ref{fig:horizontal_rauzy_graph_n=5_junction_to_junction}.
\end{proof}

\begin{lemma} \label{lem:distance-Vstrip-in-a-row}
    Let $n\geq1$ be an integer.
    For every valid configuration $c\in\Omega_n'$, 
    The distance between two consecutive occurrences of a vertical stripe tile
    (blue, green, yellow or antigreen) in the same row is $n-1$, $n$ or $n+1$.
\end{lemma}

\begin{proof}
    The horizontal Rauzy graph restricted to vertical edge labels starting with $1$
    is shown in Figure~\ref{fig:horizontal_rauzy_graph_n=5_BGY_to_BGY}.
    An arc in the horizontal Rauzy graph links two tiles $s\to t$ if and only
    if the right label of tile $s$ is equal to the left label of tile $t$.
    The graph allows to visualize the combinatorial structure between two
    consecutive vertical stripe tile on the same horizontal row within a
    configuration of $\Omega_n'$.

    The right label of a vertical stripe tile is $111$ or $112$,
    which implies that the last digit of the right label of a vertical stripe tile is $1$ or $2$.
    The left label of a vertical stripe tile is $11n$ or $11\nbar$,
    which implies that the last digit of the left label of a vertical stripe tile is $n$ or $\nbar$.
    Since the last digit increases by 1 from the left label to the right label
    of every intermediate tile (a tile appearing in between two consecutive
    vertical stripe tile in the same row),
    the number of tiles in between two consecutive vertical stripe tiles on
    the same row is at least $n-2$ and at most $\nbar-1=n$.
    We conclude that the distance (number of edges in the Rauzy graph) between
    two consecutive vertical stripe tiles in the same row is $n-1$, $n$ or $n+1$. In
    particular, it is at most $n+1$.
\begin{figure}[h]
\begin{center}
    \includegraphics{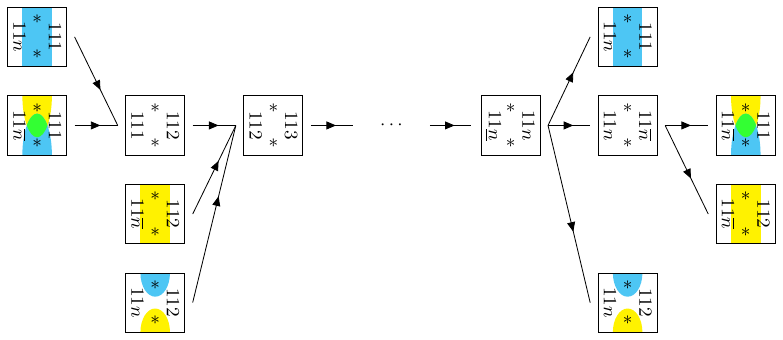}
\end{center}
\caption{Combinatorial structure between two consecutive vertical stripe tile
    on the same horizontal row within a configuration of $\Omega_n'$.
    The length of a path from a vertical stripe tile to a vertical stripe tile
    is $n-1$, $n$ or $n+1$.}
\label{fig:horizontal_rauzy_graph_n=5_BGY_to_BGY}
\end{figure}
\end{proof}

\begin{lemma} \label{lem:return-time-to-junction-tiles}
    Let $n\geq1$ be an integer.
    For every valid configuration $c\in\Omega_n'$, 
    there exist two strictly increasing sequences $A,B:\Z\to\Z$ such that
    the following hold.
    \begin{enumerate}
        \item The set of positions of junction tiles in the configuration $c$
            is the Cartesian product $c^{-1}(J_n') = A(\Z)\times B(\Z)$.
        \item The distance between two consecutive occurrences of junction tiles in
            the same row is $n$ or $n+1$, that is, 
                $A(k+1)-A(k)\in\{n,n+1\}$
            for every $k\in\Z$.
        \item The distance between two consecutive occurrences of junction tiles in
            the same column is $n$ or $n+1$, that is, 
                $B(k+1)-B(k)\in\{n,n+1\}$
            for every $k\in\Z$.
    \end{enumerate}
\end{lemma}

\begin{proof}
    (1)
    Let
    \begin{align*}
        E &= \{(\alpha_1\alpha_2\alpha_3,
                 \beta_1 \beta_2 \beta_3,
                \gamma_1\gamma_2\gamma_3,
                \delta_1\delta_2\delta_3)
                \in\Tcal_n'\mid \alpha_1=0\}\subset \Tcal_n',\\
        F &= \{(\alpha_1\alpha_2\alpha_3,
                 \beta_1 \beta_2 \beta_3,
                \gamma_1\gamma_2\gamma_3,
                \delta_1\delta_2\delta_3)
                \in\Tcal_n'\mid \beta_1=0\}\subset \Tcal_n'.
    \end{align*}
    Tiles in $E$ have zero as first coordinate of their right and left edge labels
    since $\alpha_1=\gamma_1$.
    Tiles in $F$ have zero as first coordinate of their top and bottom edge labels
    since $\beta_1=\delta_1$.
    Notice that we have $E\cup F \subset 
          Y_n\cup\widehat{Y_n}
     \cup G_n\cup\widehat{G_n} 
     \cup B_n'\cup\widehat{B_n'}\cup J_n'
     \cup A_n\cup\widehat{A_n}$ 
    and $E\cap F = J_n'$.
    Let $c\in\Omega_n'$ be a valid configuration.
    The positions of tiles from $E$ in $c$ are contiguous rows,
    that is, there exists $B\subset\Z$ such that $c^{-1}(E)=\Z\times B$.
    The positions of tiles from $F$ in $c$ are contiguous columns,
    that is, there exists $A\subset\Z$ such that $c^{-1}(F)=A\times\Z$.
    Therefore, the set of positions of junction tiles in $c$ is
    given by the Cartesian product of $A$ and $B$:
    \[
        c^{-1}(J_n')
        = c^{-1}(E \cap F)
        = c^{-1}(E) \cap c^{-1}(F)
        = (\Z\times B) \cap (A\times\Z)
        = A\times B.
    \]
    The fact that the sets $A$ and $B$ are the images of increasing maps $\Z\to\Z$
    follows from observations (2) and (3) proved below.

    (2) 
    From Lemma~\ref{lem:distance-junction-in-a-row},
    the distance between two consecutive occurrences of junction tiles in the
    same row is $n$, $n+1$ or $n+2$.
    From Lemma~\ref{lem:distance-Vstrip-in-a-row},
    the distance between two consecutive occurrences of a vertical stripe tile
    (blue, green, yellow or antigreen) in the same row is $n-1$, $n$ or $n+1$.
    Since vertical strips and junction tiles are vertically aligned,
    the difference between two consecutive elements of $A\subset\Z$ is $n$ or $n+1$.
    Also, if $a\in A$, then $a+n\in A$ or $a+n+1\in A$.
    Also $a-n\in A$ or $a-n-1\in A$.
    Thus, $A$ is the image of an increasing map $A:\Z\to\Z$
    such that $A(k+1)-A(k)\in\{n,n+1\}$ for every $k\in\Z$.

    (3)
    From the symmetry of the set $\Tcal_n'$ of tiles,
    the same observation holds for the distance between consecutive junction tiles
    in the same column.
\end{proof}

Lemma~\ref{lem:return-time-to-junction-tiles} 
means that we can subdivide valid configurations in $\Omega_n'$ by rectangular
patterns containing a unique junction tiles at their bottom left corner;
see Figure~\ref{fig:return-blocks}.

\begin{figure}[h]
    \begin{tabular}{cc}
    \includegraphics{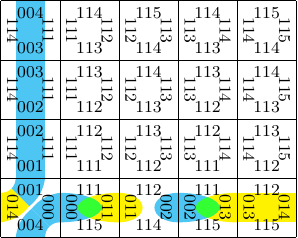} \quad&\quad
    \includegraphics{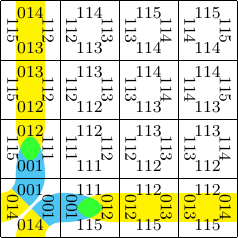} \\[5mm]
    \includegraphics{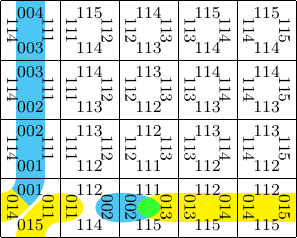} \quad&\quad
    \includegraphics{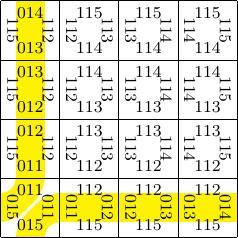}
    \end{tabular}
    \caption{Return blocks appearing in Figure~\ref{fig:T4'-15x15}.
    Each return block contains a unique junction tile at its bottom left corner.}
    \label{fig:return-blocks}
\end{figure}

\begin{proposition}\label{prop:return-blocks}
    Every configuration in $\Omega_n'$ can be divided uniquely into rectangular blocks
    of sizes
    $n\times n$, $n\times \nbar$, $\nbar\times n$ and $\nbar\times \nbar$ with
    a unique junction tile at their bottom left corner.
\end{proposition}

\begin{proof}
    Let $c\in\Omega_n'$ be a configuration.
    Let $A,B:\Z\to\Z$ be the two increasing maps
    from Lemma~\ref{lem:return-time-to-junction-tiles} 
    such that $c^{-1}(J_n')=A(\Z)\times B(\Z)$.
    For every $\bl=(\ell_1,\ell_2)\in\Z^2$, 
    the pattern appearing in $c$ at support 
    $[A(\ell_1),A(\ell_1+1)-1]\times [B(\ell_2),B(\ell_2+1)-1]$
    is a rectangular pattern containing a unique junction tile at its bottom left corner.
\end{proof}

We call such a rectangular pattern 
described in Proposition~\ref{prop:return-blocks}
a \defn{return block} (to a junction tile), see Figure~\ref{fig:return-block-to-junction-tile},
following the terminology of return words in combinatorics on words
\cite{MR1489074,MR1808196}.
While the classical notion of return word is to a single pattern,
here the notion of return block is to a subset of tiles, namely, the junction tiles.
From Proposition~\ref{prop:return-blocks},
the \defn{width} (and \defn{height}) of these blocks is $n$ or $n+1$.
On the right of the junction tile within a return block
is the \defn{bottom row} where horizontal blue, green, yellow or antigreen tiles appear.
Similarly, above the junction tile within a return block,
is the \defn{left column} where vertical blue, green, yellow or antigreen tiles appear.

\begin{figure}[h]
\begin{center}
    \begin{tikzpicture}[scale=1.2]

    \tileJunctionBackground{0}{0}{gray!50}{gray!50}{gray!50}{gray!50}

    \draw[draw=none,fill=\ourColorBlue] (\dt,1) rectangle (1-\dt,2.5);
    \tilelabelinsideVgreen {0}{2.5}{}{}{}{};
    \draw[draw=none,fill=\ourColorYellow] (\dt,3.5) rectangle (1-\dt,5);

    \draw[draw=none,fill=\ourColorBlue] (1,\dt) rectangle (3,1-\dt);
    \tilelabelinsideHgreen {3}{0}{}{}{}{};
    \draw[draw=none,fill=\ourColorYellow] (4,\dt) rectangle (5,1-\dt);

    \draw (0,0) rectangle (5,5);
    \draw (1,0) -- (1,5);
    \draw (0,1) -- (5,1);
        \node[align=center,text width=3cm] at (3,3) {\textbf{white}\\\textbf{tiles}};
        \node[align=center,text width=3cm] at (3.5,.5) {\textbf{bottom row}};
        \node[rotate=-90,align=center,text width=3cm] at (.5,3) {\textbf{left column}};
    \draw [decorate,decoration={brace,amplitude=5pt,mirror,raise=4ex}]
        (0,0) -- (5,0) node[midway,yshift=-3em]{width $W$};
    \draw [decorate,decoration={brace,amplitude=5pt,raise=4ex}]
        (0,0) -- (0,5) node[midway,xshift=-5em]{height $H$};
    \node[rotate=-90,above] at (1,4.3) {$112\cdot\cdot 112$};
    \node[rotate=-90,above] at (1,2.5) {$111\cdots 111$};
    \node[rotate=-90,below] at (0,1.8) {$11n\cdot\cdot 11n$};
    \node[rotate=-90,below] at (0,3.8) {$11\nbar\cdots\cdots 11\nbar$};
    \node[above] at (4.5,1) {$112$};
    \node[above] at (3.5,1) {$111$};
    \node[above] at (2,1)   {$111\cdots111$};
    \node[below] at (2,0)   {$11n\cdots11n$};
    \node[below] at (4,0)   {$11\nbar\cdots11\nbar$};

\end{tikzpicture}
\end{center}
    \caption{A return block is split into four disjoint parts: the junction
    tile, the left column, the bottom row and the white tiles.
    Both its width $W$ and its height $H$ take values in the set $\{n,n+1\}$.
    }
    \label{fig:return-block-to-junction-tile}
\end{figure}

We may observe that the sequences of bottom labels of a return block made of
tiles in $\Tcal_4'$ appearing completely in Figure~\ref{fig:T4'-15x15} and in
Figure~\ref{fig:return-blocks} are in the set
\begin{equation}\label{eq:not_in_taun_Vn}
    \left\{
    \begin{array}{l}
    004\cdot 114\cdot 115\cdot 114\cdot 115,\\
    004\cdot 115\cdot 114\cdot 115\cdot 115,\\
    015\cdot 114\cdot 115\cdot 115\cdot 115,\\
    014\cdot 114\cdot 115\cdot 115,\\
    014\cdot 115\cdot 115\cdot 115,\\
    015\cdot 115\cdot 115\cdot 115
    \end{array}
    \right\}
\not\subset \tau_n(V_n).
\end{equation}
In particular, $004\cdot 114\cdot 115\cdot 114\cdot 115$
does not belong to the image of $\tau_n$ when $n=4$.
But the sequence of bottom labels of a return block
has a particular structure for configurations in $\Omega_n$.
This is the subject of the next section.

\subsection{Return blocks in the Wang shift $\Omega_n$}

When considering configurations in $\Omega_n$ instead of $\Omega_n'$,
there are no antigreen tiles in the row between two consecutive junction tiles.
Thus,
Figure~\ref{fig:horizontal_rauzy_graph_n=5_junction_to_junction} simplifies
to Figure~\ref{fig:horizontal_rauzy_graph_n=5_junction_to_junction_inOmegan}.
In particular, in the bottom row of a return block within a
configuration in $\Omega_n$, the horizontal blue, green and yellow stripes
appear in this order (when they appear). The same observation holds for the
left column of a return block ordered from bottom to top.

\begin{figure}[h]
\begin{center}
    \includegraphics{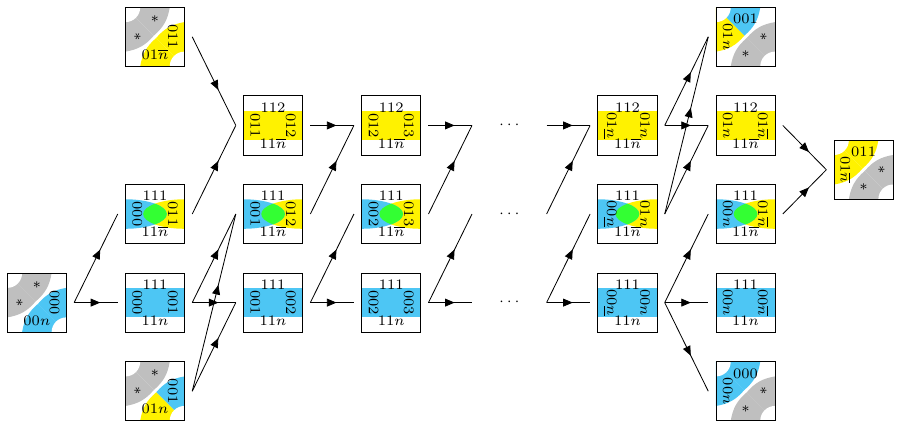}
\end{center}
\caption{Combinatorial structure between two consecutive junction tiles on the
    same horizontal row within a configuration of $\Omega_n$.
    The nodes of the graph are placed such that any two tiles appearing in the same
    column have the same last digit for its left or right labels.}
\label{fig:horizontal_rauzy_graph_n=5_junction_to_junction_inOmegan}
\end{figure}

Surprisingly, when the tiles are restricted to the set $\Tcal_n$, the boundary of the return blocks
can be decoded using the map $\tau_n$ defined in Section~\ref{sec:substitution}.

\begin{lemma}\label{lem:set-of-boundary-words}
    Let $r$ be a return block appearing in a configuration
    $c\in\Omega_n$.
    The sequences of bottom labels of tiles in the bottom row 
    of the return block $r$ (from left to right) belong to the set
    \begin{align*}
        \tau_n(V_n)=\;&    \{01\nbar\cdot (11\nbar)^{i} \mid n-1\leq i\leq n\}\\
        &\cup\{01n    \cdot (11n)^{i}(11\nbar)^{j} \mid i,j\geq0, n-1\leq i+j\leq n\}\\
        &\cup\{00n    \cdot (11n)^{i}(11\nbar)^{j} \mid i,j\geq0, i+j=n \}.
    \end{align*}
\end{lemma}

\begin{proof}
    From Lemma~\ref{lem:distance-junction-in-a-row},
    the sequence of bottom labels of the tiles between two consecutive junction tiles
    (including the left junction tile but not the right one)
    belongs to $\{00n,01n,01\nbar\} \cdot \{11n,11\nbar\}^*$.

    In the bottom row of every return block within a configuration in
    $\Omega_n$, 
    there is no antigreen stripe tile and
    the horizontal blue, green and yellow stripe tiles appear in this
    order: blue $\to$ green $\to$ yellow.
    Since the bottom label of a blue horizontal stripe tile is $11n$
    and
    the bottom label of a green or yellow horizontal stripe tile is $11\nbar$,
    the sequence of bottom labels of tiles in a horizontal row starting from a
    junction tile and ending before the next occurrence of a junction tile is in
    the set
    \[
        \{00n,01n,01\nbar\} \cdot (11n)^*(11\nbar)^*.
    \]
    Some more restrictions are imposed:
    \begin{itemize}
        \item If it starts with $00n$, the length of the sequence is $n+1$.
              Indeed, if the bottom label of a junction tile is $00n$,
              then its right label is $000$ with last digit $0$.
            From Figure~\ref{fig:horizontal_rauzy_graph_n=5_junction_to_junction_inOmegan},
            the width of the return block containing this junction tile must be $n+1$ or $n+2$.
            A return block of width $W=n+2$ is impossible from Proposition~\ref{prop:return-blocks}.
            Thus, the width of the return bock is $W=n+1$.

        \item
    Also, if it starts with $01\nbar$, the next label is not $11n$ and has to be $11\nbar$.
    Indeed $01\nbar$ is the bottom label of a junction tile with right label $011$,
    and $011$ must be the left label of a yellow horizontal stripe tile with
    bottom label $11\nbar$; 
    see Figure~\ref{fig:horizontal_rauzy_graph_n=5_junction_to_junction_inOmegan}.
    \end{itemize}
    Restricting the sequences to those of lengths $n$ or $\nbar$, 
    we have that the sequence of bottom labels of tiles in the bottom row 
    of the return block $r$ (from left to right) belongs to the set
    \begingroup
    \allowdisplaybreaks
    \begin{align*}
        &    \{01\nbar\cdot (11\nbar)^{i} \mid n-1\leq i\leq n\}\\
        &\cup\{01n    \cdot (11n)^{i}(11\nbar)^{j} \mid i,j\geq0, n-1\leq i+j\leq n\}\\
        &\cup\{00n    \cdot (11n)^{i}(11\nbar)^{j} \mid i,j\geq0, i+j=n \}\\
       =\;&\{\tau_n(111), \tau_n(000)\}\\
          &\cup\{\tau_n(00i)\mid 1\leq i \leq n+1\}
           \cup\{\tau_n(11i)\mid 2\leq i \leq n+1\}\\
          &\cup\{\tau_n(01i)\mid 1\leq i \leq n+1\}\\
        =\;&\tau(V_n).\qedhere
    \end{align*}
    \endgroup
\end{proof}

\subsection{Desubstitution $\Omega_n\leftarrow\Omega'_n$}

In this section, we prove that every valid configuration with the tiles
$\Tcal_n$ can be desubstituted into a valid configuration over $\Tcal'_n$
using the substitution $\omega_n'$.
It is based on the following lemma which relates return blocks in $\Omega_n$
to tiles of $\Tcal_n'$.

\begin{lemma}\label{lem:return-word-description}
    Let $y\in\Omega_n$ be a configuration.
    For every return block $r$
    appearing in $y$, there exists a unique tile
    $t=\raisebox{-7mm}{
    \begin{tikzpicture}[auto,scale=.5]
        \tikzstyle{every node}=[font=\footnotesize]
        \tile{white}{0}{0}{\alpha}{\beta}{\gamma}{\delta}
    \end{tikzpicture}}
    \in\Tcal'_n$
    such that 
    $r=\omega_n'(t)$
    with external labels
    $\raisebox{-7mm}{
    \begin{tikzpicture}[auto,scale=.5]
        \tikzstyle{every node}=[font=\footnotesize]
        \tile{white}{0}{0}{\tau_n(\alpha)}{\tau_n(\beta)}{\tau_n(\gamma)}{\tau_n(\delta)}
    \end{tikzpicture}}$.
\end{lemma}

\begin{proof}
    Let $y\in\Omega_n$ be a configuration.
    From Proposition~\ref{prop:return-blocks},
    the configuration $y$
    can be divided into return blocks,
    that is,
    rectangular blocks of sizes
    $n\times n$, $n\times \nbar$, $\nbar\times n$ or $\nbar\times \nbar$ with
    a unique junction tiles at the bottom left corner;
    see Figure~\ref{fig:return-block-to-junction-tile}.

    Let $r$ be a return block appearing in $y$.
    From Lemma~\ref{lem:set-of-boundary-words},
    the sequences of bottom labels of tiles in the bottom row 
    of the return block $r$ (from left to right) belong to the set
    $\tau_n(V_n)$.
    By symmetry and since $r$ is surrounded by returns blocks, 
    this also holds for the right, top and left labels of $r$.
    Therefore, let $\alpha,\beta,\gamma,\delta\in V_n$
    such that
    the right, top, left and bottom labels
    of the return block $r$ are respectively
    $\tau_n(\alpha)$, $\tau_n(\beta)$, $\tau_n(\gamma)$ and $\tau_n(\delta)$.
    From Proposition~\ref{prop:there-exists-rect-pattern-iff},
    $t=\raisebox{-7mm}{
    \begin{tikzpicture}[auto,scale=.5]
        \tikzstyle{every node}=[font=\footnotesize]
        \tile{white}{0}{0}{\alpha}{\beta}{\gamma}{\delta}
    \end{tikzpicture}}
    \in\Tcal'_n$.
    From Lemma~\ref{lem:Tcaln'-there-exists-rect-pattern},
    there exists a unique rectangular pattern with these external labels.
    Thus, $r=\omega_n'(t)$.
\end{proof}

\begin{proposition}\label{prop:exists-unique-preimage}
    Let $n\geq1$ be an integer.
    For every configuration $y\in\Omega_n$,
    there exist a unique configuration $x\in\Omega'_n$
    and a unique vector $\bk\in\{0,1,\dots,n\}^2$
    such that
    $y=\sigma^\bk(\omega_n'(x))$.
\end{proposition}

\begin{proof}
    Let $\omega_n':\Omega'_n\to\Omega_n'$ 
    be the 2-dimensional substitution defined in
    \eqref{eq:defintion-omegan'}.

    Let $y\in\Omega_n$ be a configuration.
    From Lemma~\ref{lem:return-time-to-junction-tiles}
    there exist two strictly increasing sequences $A,B:\Z\to\Z$ such that
    the set of positions of junction tiles in the configuration $y$ is the Cartesian product 
    $A(\Z)\times B(\Z)$.
    Also the distance between two consecutive occurrences of junction tiles in
    the same row or the same column is $n$ or $n+1$, that is, 
    $A(\ell+1)-A(\ell)\in\{n,n+1\}$
    and $B(\ell+1)-B(\ell)\in\{n,n+1\}$ for every $\ell\in\Z$.
    We may suppose without loss of generality that the sequences $A$ and $B$ 
    are defined in such a way that the sequences take nonnegative values
    for nonnegative indices exclusively.
    In other words, 
    $A(\ell)\geq0$ if and only if $\ell\geq0$
    and
    $B(\ell)\geq0$ if and only if $\ell\geq0$.

    For every $\bl=(\ell_1,\ell_2)\in\Z^2$, consider the return block
    $y|_{S_\bl}$
    of support $S_\bl=[A(\ell_1),A(\ell_1+1)-1]\times [B(\ell_2),B(\ell_2+1)-1]$.
    From Lemma~\ref{lem:return-word-description}
    there exists a unique tile $x_\bl\in\Tcal_n'$ such that
    $y|_{S_\bl}=\omega_n'(x_\bl)$.
    Let $\bk=(-A(-1),-B(-1))$.
    The configuration $\sigma^{-\bk}(y)$ has a junction tile at the origin $(0,0)$.
    The configuration $x=(x_\bl)_{\bl\in\Z^2}$ 
    belongs to $\Omega_n'$ and
    satisfies
    that $\omega_n'(x)=\sigma^{-\bk}(y)$.
    Thus, $y=\sigma^{\bk}\omega_n'(x)$.
\end{proof}

\begin{proposition}\label{prop:desubstitute-to-Omegan'}
    For every integer $n\geq1$,
    the 2-dimensional substitution
    $\omega_n':\Omega'_n\to\Omega_n'$ 
    satisfies
    $\Omega_n\subseteq\overline{\omega_n'(\Omega'_n)}^\sigma$.
\end{proposition}

\begin{proof}
    From Proposition~\ref{prop:exists-unique-preimage},
    for every configuration $y\in\Omega_n$,
    there exist a unique configuration $x\in\Omega'_n$
    and a unique vector $\bk\in\{0,1,\dots,n\}^2$
    such that $y=\sigma^\bk(\omega_n'(x))$.
    Therefore, 
    $\Omega_n\subseteq\overline{\omega_n'(\Omega'_n)}^\sigma$.
\end{proof}

\section{Tiles in $\Tcal_n'\setminus\Tcal_n$ are illegal so that $\Omega_n'=\Omega_n$}
\label{sec:Omegan'subseteqOmegan}

By definition $\Tcal_n\subset\Tcal_n'$, so that $\Omega_n\subseteq\Omega_n'$.
In this section, we prove that in every configuration of the Wang shift
$\Omega'_n$ defined from the set $\Tcal'_n$, only the tiles from $\Tcal_n$
appear, that is,
$\Omega'_n\subseteq\Omega_n$.

\subsection{Illegal tiles}
Recall that the additional tiles are
\[
    \Tcal'_n\setminus \Tcal_n = A_n\cup\widehat{A_n} \cup 
        \{j_n^{0,0,1,1},
          j_n^{1,1,0,0}\}
    \cup \{b_n^n,\widehat{b_n^n}\}.
\]
The proof that these tiles do not appear in any configuration in $\Omega_n'$
follows from the following lemmas.
The easiest is to show that no configuration contain the last blue tile because
the argument is very local.

\begin{lemma}\label{lem:no-blue-tile}
    A valid configuration in $\Omega'_n$ contains no blue tile 
    in $\{b_n^n,\widehat{b_n^n}\}$.
\end{lemma}

\begin{proof}
    Let $c\in\Omega'_n$ be a valid configuration.
    The configuration $c$ does not contain the tile
    $b_n^n = 
    \raisebox{-9mm}{
    \begin{tikzpicture}[auto]
        \tikzstyle{every node}=[font=\footnotesize]
        \tileH{\ourColorBlue}{0}{0}{00\nbar}{111}{00n}{11n}
    \end{tikzpicture}}$, because
    no tile from $\Tcal'_n$ has left label $00\nbar$.
    Similarly,
    the configuration $c$ does not contain the tile
    $\widehat{b_n^n}$, because
    no tile from $\Tcal'_n$ has bottom label $00\nbar$.
\end{proof}

Then, we show to no configuration of $\Omega_n'$ contains any antigreen tile.
The argument is more difficult, because antigreen tiles admit large surroundings;
see Figure~\ref{fig:T4'-15x15}. As seen in the figure and proved in the next
lemma, the presence of an antigreen tile forces the presence of another
antigreen tile a few rows below that is closer to the left to a junction tile.

\begin{lemma}\label{lem:no-antigreen-tile}
    A valid configuration in $\Omega'_n$ contains no antigreen tile from the
    set $A_n\cup\widehat{A_n}$.
\end{lemma}

\begin{proof}
    Let $c\in\Omega'_n$ be a valid configuration.
    Recall that
    $a_n^i = 
    \raisebox{-9mm}{
    \begin{tikzpicture}[auto]
        \tikzstyle{every node}=[font=\footnotesize]
    \tileHantigreen{0}{0}{00\ibar}{112}{01i}{11n}
    \end{tikzpicture}}$.
    The configuration $c$ does not contain the tile
    $a_n^n$, because
    $a_n^n$ has left label $00\nbar$
    but no tile from $\Tcal'_n$ has left label $00\nbar$.
    Similarly,
    the configuration $c$ does not contain the tile
    $\widehat{a_n^n}$, because
    $\widehat{a_n^n}$ has top label $00\nbar$
    but no tile from $\Tcal'_n$ has top label $00\nbar$.

    Suppose by contradiction
    that $a_n^i$ appears in the configuration $c$
    for some integer $i$ such that $1\leq i\leq n-1$.
    Let $A,B:\Z\to\Z$ be the two increasing maps
    from Lemma~\ref{lem:return-time-to-junction-tiles} 
    such that $c^{-1}(J_n')=A(\Z)\times B(\Z)$.
    Suppose that $a_n^i$ appears at position $\bl=(\ell_1,\ell_2)\in\Z^2$.
    Let $\bk=(k_1,k_2)\in\Z^2$ be such that $A(k_1)\leq \ell_1<A(k_1+1)$ and
                                            $B(k_2)\leq \ell_2<B(k_2+1)$.
    Note that we must have $B(k_2)=\ell_2$.
    Suppose that the occurrence $\bl$ is chosen
    such that $\ell_1-A(k_1)$ is the minimum among all occurrences of the tile $a_n^i$ in $c$.
    In other words, such that the distance to the nearest junction tile to its left
    on the same row is minimal.
    Since the bottom and top labels of $a_n^i$ 
    start with $1$, the column $\ell_1$ in the configuration $c$ contains no junction tile,
    thus $A(k_1)\neq\ell_1$ and $\ell_1-A(k_1)\geq1$. There are two cases to consider.

    \textsc{Case} $\ell_1-A(k_1)=1$.
    In this case the tile at position $(A(k_1),B(k_2))$ is a junction tile
    with right label $011$ and bottom label $01\nbar$.
    Also the antigreen tile at position $(\ell_1,\ell_2)$ is $a_n^1$.
    Below the antigreen tile are white tiles and below the junction tile is
    a yellow or green tile that we show in gray in
    Figure~\ref{fig:presence_an1_leads_contradiction}.
    \begin{figure}[h]
    \begin{center}
    \begin{tikzpicture}[scale=1.3]
        \tikzstyle{every node}=[font=\footnotesize]
        \def\t{.1}
        \def\d{-.5}
        \draw[->] (\d,0) -- (\d,5.5);
        \draw (\d+\t,0.5) -- (\d-\t,0.5) node[left] {$B(k_2-1)$};
        \draw (\d+\t,1.5) -- (\d-\t,1.5);
        \draw (\d+\t,3.5) -- (\d-\t,3.5);
        \draw (\d+\t,4.5) -- (\d-\t,4.5) node[left] {$B(k_2)=\ell_2$};
        \draw[->] (0,\d) -- (2.5,\d);
        \draw (0.5,\d+\t) -- (0.5,\d-\t) node[below] {$A(k_1)$};
        \draw (1.5,\d+\t) -- (1.5,\d-\t) node[below] {$\ell_1$};
        \tileJunctionInsideIXXI{0}{4}{011}{*}{*}{01\nbar}{gray!50}{gray!50}
        \tilelabelinsidescopeHpink{1}{4}{002}{112}{011}{11n}
        \tilelabelinsidescopeV{gray!50}{0}{3}{11*}{01\nbar}{11\nbar}{0*n}
        \tilelabelinsidescope{white}{1}{3}{*}{11n}{*}{11(n-1)}

        \node at (0.5,2.5) {$\vdots$};
        \node at (1.5,2.5) {$\vdots$};

        \tilelabelinsidescopeV{gray!50}{0}{1}{*}{0*3}{*}{0{*}2}
        \tilelabelinsidescope{white}{1}{1}{*}{112}{*}{111}
        \tileJunctionInsideOXXO{0}{0}{00*}{0{*}2}{*}{*}{gray!50}{gray!50}
        \tilelabelinsidescopeH{gray!50}{1}{0}{*}{111}{00*}{*}
    \end{tikzpicture}
    \end{center}
        \caption{The presence of the antigreen $a_n^1$ leads to a contradiction.}
        \label{fig:presence_an1_leads_contradiction}
    \end{figure}
    So the unit parts of horizontal edge labels decrease by one at each level from top to bottom
    until we reach the white tile at position $(\ell_1,B(k_2-1)+1)$ with bottom label 111
    and a tile at position $(A(k_1),B(k_2-1)+1)$ with bottom label $0{*}2$.
    The tile at position $(\ell_1,B(k_2-1))$ must be a green or blue tile with
    left label $00{*}$. The tile at position $(A(k_1),B(k_2-1)$ must be a junction tile,
    but there are no junction tile with top label $0{*}2$.
    So this case leads to a contradiction.

    \textsc{Case} $\ell_1-A(k_1)>1$. This means that tiles in the column to the left of $a_n^i$
    do not contain junction tiles. 
    On Figure~\ref{fig:horizontal_rauzy_graph_n=5_junction_to_junction},
    we observe that
    only the yellow tile $y_n^{i-1}$ has right label $01i$.
    Thus, the tile to the left of $a_n^i$ at position $(\ell_1-1,\ell_2)$
    needs to be the yellow tile $y_n^{i-1}$.
    For every integer $j$ such that $B(k_2-1)<j<B(k_2)$ the tiles
    at positions $(\ell_1-1,j)$ and $(\ell_1,j)$ are white tiles.
    So the unit parts of the horizontal edge labels decrease by one at each level
    from top to bottom.
    Thus, the tile at position $(\ell_1-1,B(k_2-1))$ has top label $112$
    and the tile at position $(\ell_1,B(k_2-1))$ has top label $111$.
    The situation is illustrated in Figure~\ref{fig:presence_ani_leads_contradiction}.
    \begin{figure}[h]
    \begin{center}
    \begin{tikzpicture}[scale=1.3]
        \tikzstyle{every node}=[font=\footnotesize]
        \def\t{.1}
        \def\d{-.5}
        \draw[->] (\d,0) -- (\d,5.5);
        \draw (\d+\t,0.5) -- (\d-\t,0.5) node[left] {$B(k_2-1)$};
        \draw (\d+\t,1.5) -- (\d-\t,1.5);
        \draw (\d+\t,3.5) -- (\d-\t,3.5);
        \draw (\d+\t,4.5) -- (\d-\t,4.5) node[left] {$B(k_2)=\ell_2$};
        \draw[->] (0,\d) -- (4.5,\d);
        \draw (0.5,\d+\t) -- (0.5,\d-\t) node[below] {$A(k_1)$};
        \draw (2.5,\d+\t) -- (2.5,\d-\t) node[below] {$\ell_1-1$};
        \draw (3.5,\d+\t) -- (3.5,\d-\t) node[below] {$\ell_1$};
        \tilelabelinsidescope{white}{2}{3}{*}{11\nbar}{*}{11n}
        \tilelabelinsidescope{white}{3}{3}{*}{11n}{*}{11(n-1)}
        \tilelabelinsidescopeH{\ourColorYellow}{2}{4}{01i}{112}{01(i-1)}{11\nbar}
        \tilelabelinsidescopeHpink{3}{4}{00\ibar}{112}{01i}{11n}
        \tileJunctionInsideGRAY{0}{4}{*}{*}{*}{*}{gray!50}

        \node at (2.5,2.5) {$\vdots$};
        \node at (3.5,2.5) {$\vdots$};
        \node at (1.5,0.5) {$\dots$};
        \node at (1.5,4.5) {$\dots$};

        \tilelabelinsidescope{white}{2}{1}{*}{113}{*}{112}
        \tilelabelinsidescope{white}{3}{1}{*}{112}{*}{111}
        \tileJunctionInsideGRAY{0}{0}{*}{*}{*}{*}{gray!50}
        \tilelabelinsidescopeHpink{2}{0}{*}{112}{*}{11n}
    \end{tikzpicture}
    \end{center}
        \caption{The presence of the antigreen $a_n^i$ leads to a contradiction.}
        \label{fig:presence_ani_leads_contradiction}
    \end{figure}
    Since $112$ and $111$ are the labels of consecutive horizontal edges,
    we deduce from Figure~\ref{fig:horizontal_rauzy_graph_n=5_junction_to_junction} that
    the tile at position $(\ell_1-1,B(k_2-1))$ must be an antigreen tile as well.
    We observe that this antigreen tile is closer in distance to a junction tile to its
    left on the same row.
    This is a contradiction with the minimality of $\ell_1-A(k_1)$.
    Thus, the configuration $c$ does not contain the antigreen tile $a_n^i$.

    Finally, by contradiction, suppose that the tile $\widehat{a_n^i}$ appears in the
    configuration $c$.
    Since $\Tcal_n'$ is symmetric, that is $\widehat{\Tcal_n'}=\Tcal_n'$,
    the symmetric configuration $\widehat{c}$ is also a valid configuration in $\Omega_n'$.
    Thus, the configuration $c$ contains the tile $a_n^i$ which contradicts the
    conclusion of the previous paragraph.
\end{proof}

The previous lemma implies that the pattern shown in Figure~\ref{fig:T4'-15x15}
cannot be extended to a valid configuration in $\Omega_n'$.

\begin{lemma}\label{lem:no-orange-tile}
    A valid configuration in $\Omega'_n$ contains no junction tile from 
    the set $\{j_n^{0,0,1,1},j_n^{1,1,0,0}\}$.
\end{lemma}

\begin{proof}
    Recall that
\[
    j_n^{0,0,1,1} = 
    \raisebox{-9mm}{
    \begin{tikzpicture}[auto]
        \tikzstyle{every node}=[font=\footnotesize]
        \tileJunctionOIIO{0}{0}{000}{011}{01\nbar}{00n}
    \end{tikzpicture}}
\qquad
\text{ and }
\qquad
    \widehat{j_n^{0,0,1,1}} = 
    j_n^{1,1,0,0} = 
    \raisebox{-9mm}{
    \begin{tikzpicture}[auto]
        \tikzstyle{every node}=[font=\footnotesize]
        \tileJunctionIOOI{3}{0}{011}{000}{00n}{01\nbar}
    \end{tikzpicture}}.
\]
    Let $x\in\Omega'_n$ be a valid configuration.
    We first prove that $x$ does not contain the tile $j_n^{0,0,1,1}$.
    By contradiction, suppose that the tile $j_n^{0,0,1,1}$ appears in the configuration $x$ 
    at some position $\bl\in\Z^2$.
    Consider the return block containing this junction tile and let $W$ be its width
    and $H$ be its height.

    The bottom label of the junction tile $j_n^{0,0,1,1}$ is $00n$
    and its right label is $000$ with last digit $0$.
    From Figure~\ref{fig:horizontal_rauzy_graph_n=5_junction_to_junction},
    the width of the return block containing this junction tile must be $n+1$ or $n+2$.
    A return block of width $W=n+2$ is impossible from Proposition~\ref{prop:return-blocks}.
    Thus, the width of the return bock is $W=n+1$.

    If $n>1$, then
    we have $W=n$,
    which is a contradiction.
    Indeed,
    the tile appearing above the junction tile $j_n^{0,0,1,1}$ must be a 
    vertical stripe tile with right label $112$, either yellow or antigreen.
    From the observation made in Figure~\ref{fig:vertical-sequence-white-tiles},
    the width of this return block is $W=n$.

    If $n=1$, three different junction tile can appear on top of $j_n^{0,0,1,1}$.
    All of them have right label $001$.
    On the right of $j_n^{0,0,1,1}$, there may be a green or a blue tile,
    both of them having top label $111$.
    We get the following picture where we illustrate the blue or green tile in gray.
    \begin{center}
    \begin{tikzpicture}[scale=1.3]
        \tikzstyle{every node}=[font=\footnotesize]
        \tileJunctionInsideOXXI{0}{1}{001}{*}{*}{011}{gray!50}{gray!50}
        \tileJunctionInsideOIIO{0}{0}{000}{011}{01\nbar}{00n}
        \tilelabelinsidescopeH{gray!50}{1}{0}{*}{111}{000}{*}
    \end{tikzpicture}
    \end{center}
    But no tile from $\Tcal_1'$ have left label $001$
    and bottom label $111$; see Figure~\ref{fig:T'n-for-n=1}.
    Thus, no tile can be placed at position $\bl+(1,1)$.
    This is a contradiction.
\begin{figure}[h]
\begin{center}
\includegraphics{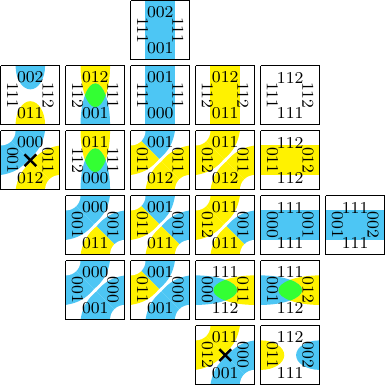}
\end{center}
    \caption{Extended metallic mean Wang tile sets $\Tcal'_n$ for $n=1$.}
    \label{fig:T'n-for-n=1}
\end{figure}

    Finally, by contradiction, suppose that the tile 
    $j_n^{1,1,0,0}=\widehat{j_n^{0,0,1,1}}$ appears in the
    configuration $x$.
    Since $\Tcal_n'$ is symmetric, that is $\widehat{\Tcal_n'}=\Tcal_n'$,
    the symmetric configuration $\widehat{x}$ is also a valid configuration in $\Omega_n'$.
    Thus, the configuration $x$ contains the tile $j_n^{0,0,1,1}$ which contradicts the
    first part of the proof.
\end{proof}

We may now prove the following result.

\begin{proposition}\label{prop:Omegan'=Omegan}
    For every integer $n\geq1$, $\Omega_n'=\Omega_n$.
\end{proposition}

\begin{proof}
    Since $\Tcal_n\subseteq\Tcal_n'$, we have $\Omega_n\subseteq\Omega_n'$.

    Let $c\in\Omega'_n$ be a valid configuration.
From Lemma~\ref{lem:no-blue-tile}, the configuration $c$ contains no blue
    tile in $\{b_n^n,\widehat{b_n^n}\}$.
From Lemma~\ref{lem:no-antigreen-tile},
    the configuration $c$ contains no antigreen tile from $A_n\cup\widehat{A_n}$.
From Lemma~\ref{lem:no-orange-tile}
the configuration $c$ contains no junction tile from
the set $\{j_n^{0,0,1,1},j_n^{1,1,0,0}\}$.
    Thus, the range of $c$ is $c(\Z^2)\subset\Tcal_n$.
    Thus, $c\in\Omega_n$, from which we conclude that
    $\Omega_n'\subseteq\Omega_n$.
\end{proof}

\section{$\Omega_n$ is self-similar and aperiodic}
\label{sec:self-similar-aperiodic}

In this section, we show that $\Omega_n$ is self-similar and aperiodic.
We prove Theorem~\ref{thm:similar-to-itself} below after recalling
its statement.

\begin{THEOREMI}
    \MainTheoremI
\end{THEOREMI}

\begin{proof} %
    Let $n\geq1$ be an integer.
    From Proposition~\ref{prop:desubstitute-to-Omegan'},
    the 2-dimensional substitution
    $\omega_n':\Omega_n'\to\Omega_n'$ 
    defined in \eqref{eq:defintion-omegan'}
    satisfies
    $\Omega_n\subseteq\overline{\omega_n'(\Omega'_n)}^\sigma$.
    From Proposition~\ref{prop:Omegan'=Omegan},
    we have $\Omega_n'=\Omega_n$.
    The restriction of $\omega_n'$
    to $\Omega_n$ is the 2-dimensional substitution $\omega_n:\Omega_n\to\Omega_n$
    defined in \eqref{eq:defintion-omegan}.
    From Lemma~\ref{lem:omegan-is-2-dim-substitution},
    $\omega_n(\Omega_n)\subset\Omega_n$.
    Therefore, we have
    \[
        \Omega_n
        \subseteq\overline{\omega_n'(\Omega'_n)}^\sigma
        =\overline{\omega_n'(\Omega_n)}^\sigma
        =\overline{\omega_n(\Omega_n)}^\sigma
        \subseteq\Omega_n.
    \]
    Therefore, $\omega_n$ is in fact a
    2-dimensional substitution
    $\Omega_n\to\Omega_n$
    satisfying
    $\Omega_n=\overline{\omega_n(\Omega_n)}^\sigma$.
    The 2-dimensional substitution $\omega_n$ is recognizable 
    following Proposition~\ref{prop:return-blocks}, since
    every configuration in $\Omega_n$ can be uniquely divided into return blocks.
    The 2-dimensional substitution $\omega_n$ is expansive (the image of every
    tile contains a junction tile and the image of every junction tile has
    a height and width at least 2).
    Hence the Wang shift $\Omega_n$ is self-similar with respect to the
    substitution $\omega_n$.
\end{proof}

\begin{proof}[Proof of Corollary~\ref{cor:aperiodic}]
    From Theorem~\ref{thm:similar-to-itself},
    we have that the Wang shift $\Omega_n$ is self-similar satisfying
    $\Omega_n=\overline{\omega_n(\Omega_n)}^\sigma$.
    Since the substitution $\omega_n$ is expansive and recognizable,
    it follows from Proposition~\ref{prop:expansive-recognizable-aperiodic}
    that $\Omega_n$ is aperiodic.
\end{proof}

\section{The self-similarity is primitive}
\label{sec:primitive}

Substitutive shifts obtained from expansive and primitive morphisms are
interesting for their properties.
As in the one-dimensional case, we say that $\omega$ is \defn{primitive}
if there exists $m\in\N$ such that
for every $a,b\in\Acal$ the letter $b$ occurs in $\omega^m(a)$.
In this section, we show that the 2-dimensional substitution $\omega_n$ is
primitive.

\begin{figure}[h]
\begin{center}
    \begin{tikzpicture}[>=latex,yscale=1.1]
    \def\x{1}
    \node (0) at (4.5,13) {$\{t\}$, for some $t\in\Tcal_n$};

    \node (A0) at (1.5,12) {$\{t\}$, for some $t\in\Tcal_n$};
    \node (A1) at (7,12)   {$\{t\}$, for some $t\in J_n$};

    \node (B0) at (1.5,11) {$\{t\}$, for some $t\in J_n$};
    \node (B1) at (8.5,11) {$\{w_n^{1,1}\}$};

    \node (C0) at (-2.5, 10)  {$\{w_n^{1,1}\}$};
    \node (C1) at (4.5,  10) {$\{t\}$, for some $t\in J_n$};
    \node (C2) at (8.5, 10)   {$\{j_n^{1,1,1,1}\}$};

    \node (D0) at (-2.5, 9)   {$\{j_n^{1,1,1,1}\}$};
    \node (D1) at (4.5,  9)   {$\{w_n^{1,1}\}$};
    \node (D2) at (8.5,  9)  {$\{b_n^0,\widehat{b_n^0}\}$};

    \node (E0) at (-4.5, 8)   {$\{w_n^{1,1}\}$};
    \node (E1) at (-0.5,  8)   {$\{b_n^0,\widehat{b_n^0}\}$};
    \node (E2) at (4.5,  8)   {$\{j_n^{1,1,1,1}\}$};
    \node (E3) at (8.5,  8)  {$\{\widehat{y_n^1},y_n^1\}$};

    \node (F0) at (-4.5, 7)    {$\{j_n^{1,1,1,1}\}$};
    \node (F1) at (-0.5,  7)    {$Y_n\cup\widehat{Y_n}$};
    \node (F2) at (2.8,  7)    {$\{j_n^{0,0,0,0}\}$};
    \node (F3) at (4.5,  7)  {$\{w_n^{1,1}\}$};
    \node (F4) at (6.0,  7)  {$\{b_n^0,\widehat{b_n^0}\}$};
    \node (F5) at (8.5,  7)  {$\{g_n^0,\widehat{g_n^0}\}$};

    \node (G0) at (-6,   6) {$\{j_n^{0,0,0,0}\}$};
    \node (G1) at (-4.8,   6) {$W_n$};
    \node (G2) at (-3.6, 6) {$B_n\cup\widehat{B_n}$};
    \node (G3) at (-1.2, 6) {$\{j_n^{0,1,0,0},j_n^{0,0,0,1}\}$};
    \node (G4) at (1.0,  6) {$G_n\cup\widehat{G_n}$};
    \node (G5) at (2.8,    6) {$\{j_n^{0,1,0,1}\}$};
    \node (G6) at (4.5,    6) {$\{j_n^{1,1,1,1}\}$};
    \node (G7) at (6.2,    6) {$Y_n\cup\widehat{Y_n}$};
    \node (G8) at (8.5,    6) {$\{j_n^{0,1,1,1},j_n^{1,1,0,1}\}$};
    \draw[->] (0) -- (A0);
    \draw[->] (0) -- (A1);
    \draw[->] (A0) -- (B0);
    \draw[->] (A1) -- (B1);
    \draw[->] (B0) -- (C0);
    \draw[->] (B0) -- (C1);
    \draw[->] (B1) -- (C2);
    \draw[->] (C0) -- (D0);
    \draw[->] (C1) -- (D1);
    \draw[->] (C2) -- (D2);
    \draw[->] (D0) -- (E0);
    \draw[->] (D0) -- (E1);
    \draw[->] (D1) -- (E2);
    \draw[->] (D2) -- (E3);
    \draw[->] (E0) -- (F0);
    \draw[->] (E1) -- (F1);
    \draw[->] (E2) -- (F2);
    \draw[->] (E2) -- (F3);
    \draw[->] (E2) -- (F4);
    \draw[->] (E3) -- (F5);
    \draw[->] (F0) -- (G0);
    \draw[->] (F0) -- (G1);
    \draw[->] (F0) -- (G2);
    \draw[->] (F1) -- (G3);
    \draw[->] (F1) -- (G4);
    \draw[->] (F2) -- (G4);
    \draw[->] (F2) -- (G5);
    \draw[->] (F3) -- (G6);
    \draw[->] (F4) -- (G7);
    \draw[->] (F5) -- (G8);
\end{tikzpicture}
\end{center}
\caption{
    When an arrow appears linking sets of tiles $S\to T$ and vertex $T$ has in-degree one, 
    it means that 
    $T\subseteq \bigcup_{s\in S}\{t\in\Tcal_n\mid t\text{ occurs in }\omega_n(s)\}$,
    that is,
    every tile $t\in T$ appears in the image of some tile $s\in
    S$ under the substitution $\omega_n$.
    When two arrows $S\to T$ and $S'\to T$ appear,
    it means that every tile $t\in T$ appears in the image of some tile $s\in
    S\cup S'$ under the substitution $\omega_n$.
    The figure illustrates that for every tile $t\in\Tcal_n$ 
    the pattern $(\omega_n)^7(t)$ contains every tile of $\Tcal_n$.
    This shows the primitivity of the substitution $\omega_n$.
}
\label{fig:primitivity}
\end{figure}

\begin{lemma}\label{lem:primitivity}
    For every integer $n\geq1$, the 2-dimensional substitution
    $\omega_n:\Omega_n\to\Omega_n$ is primitive.
\end{lemma}

\begin{proof}
    The proof follows from the following observations about the substitution $\omega_n$:
    \begin{itemize}
        \item in the image of every tile in $\Tcal_n$ under $\omega_n$, there is some junction tile;
        \item in the image of every junction tile, there is a white tile $w_{n}^{1,1}$;
\item in the image of the white tile $w_{n}^{1,1}$, there is
    the junction tile $j_n^{1,1,1,1}$;
\item in the image of the junction tile $j_n^{1,1,1,1}$, there are
    the junction tile $j_n^{0,0,0,0}$,
    all white tiles $W_n$ including the white tile $w_{n}^{1,1}$,
    and all blue tiles $B_n\cup\widehat{B_n}$ including the blue tiles
            $\{b_{n}^0,\widehat{b_{n}^0}\}$ 
            (all blue tiles appear in the image because 
            the left and bottom label of $j_n^{1,1,1,1}$ is $01\nbar$,
            see Lemma~\ref{lem:unique-horizontal-strip}
            and Figure~\ref{fig:bottom-row-if-bottom-labels-is-tau01i});
\item in the image of the blue tiles $\{b_{n}^0,\widehat{b_{n}^0}\}$,
    there are all yellow tiles $Y_n\cup\widehat{Y_n}$ including the yellow
            tiles $\{y_{n}^1,\widehat{y_{n}^1}\}$
            (all yellow tiles appear in the images because 
            the left label of $b_n^0$ is $000$
            and the bottom label of $\widehat{b_n^0}$ is $000$,
            see Lemma~\ref{lem:unique-horizontal-strip}
            and Figure~\ref{fig:bottom-row-if-bottom-labels-is-tau00i});
\item in the image of yellow tiles $Y_n\cup\widehat{Y_n}$,
    there are the junction tiles $\{j_n^{0,1,0,0},j_n^{0,0,0,1}\}$;
\item in the image of $Y_n\cup\widehat{Y_n}\cup\{j_n^{0,0,0,0}\}$,
    there are all green tiles $G_n\cup\widehat{G_n}$:
            \begin{itemize}
                \item green tiles $g_n^n$ and $\widehat{g_n^n}$ appear in the image of
            $j_n^{0,0,0,0}$ because 
            the left and bottom label of $j_n^{0,0,0,0}$ is $00n$,
            see Lemma~\ref{lem:unique-horizontal-strip}
            and Figure~\ref{fig:bottom-row-if-bottom-labels-is-tau00i};
                \item green tiles $g_n^i$ and $\widehat{g_n^i}$ for $0\leq i < n$
            appear in the images of the yellow tiles because 
            the bottom label of $\widehat{y_n^{\ibar}}$ is $01\ibar$,
            see Lemma~\ref{lem:unique-horizontal-strip}
            and Figure~\ref{fig:bottom-row-if-bottom-labels-is-tau01i});
            \end{itemize}
\item in the image of $j_n^{0,0,0,0}$,
    there is the junction tile $j_n^{0,1,0,1}$;
\item in the image of the blue tiles $\{y_{n}^1,\widehat{y_{n}^1}\}$,
    there are the green tiles $\{g_{n}^0,\widehat{g_{n}^0}\}$;
\item in the image of the green tiles $\{g_{n}^0,\widehat{g_{n}^0}\}$,
    there are the junction tiles $\{j_n^{0,1,1,1},j_n^{1,1,0,1}\}$.
\end{itemize}
    The tiles that can be obtained from the successive application of the
    substitution $\omega_n$ are shown in Figure~\ref{fig:primitivity}.
    The graph in the figure shows that every tile appears at distance 7 of
    every tile in $\Tcal_n$.
    Thus, for every tile $t\in\Tcal_n$ 
    the pattern $(\omega_n)^7(t)$ contains all tiles of $\Tcal_n$.
    Therefore, we conclude that $\omega_n$ is primitive.
\end{proof}

The exponent $7$ deduced in the previous proof is not sharp
as computations illustrate that for every integer $n\geq2$, the incidence
matrix of $(\omega_n)^4$ is already positive, while
the incidence matrix of $(\omega_1)^5$ is positive.

\begin{lemma}\label{lem:inflation-factor}
    The Perron--Frobenius dominant eigenvalue of the incidence matrix of
    $\omega_n$ is $\beta_n^2$, the square of the
    $n^{th}$ metallic mean number,
    and the inflation factor of $\omega_n$ is $\beta_n$.
\end{lemma}

\begin{proof} 
    We may deduce the dominant eigenvalue of the incidence matrix of
    $\omega_n$ from that of a simpler substitution.
    For every integer $n\geq1$, let $\rho_n$
    be the following 1-dimensional substitution:
    \[
    \begin{array}{rccl}
    \rho_n:&\{\texttt{a},\texttt{b}\}^* & \to & \{\texttt{a},\texttt{b}\}^*\\
        &  \texttt{a}&\mapsto&\texttt{a}\texttt{b}^n\\
        &  \texttt{b}&\mapsto&\texttt{a}\texttt{b}^{n-1}
    \end{array}.
    \]
    The incidence matrix of $\rho_n$ is
    \[
        \left(
    \begin{array}{cc}
        1 & 1\\
        n & n-1
    \end{array}
        \right)
    \]
    whose characteristic polynomial is $x^2-nx-1$.
    The Perron--Frobenius dominant eigenvalue of the incidence matrix of $\rho_n$
    is the positive root $\beta_n$ of the polynomial $x^2-nx-1$.
    Since $\rho_n$ is primitive, the growth rate of $|\rho_n^k(u)|$
    is independent of $u\in\{\texttt{a},\texttt{b}\}$ and is equal to $\beta_n$
    \cite[Corollary 5.2]{MR2590264}.
    In other words, for every $u\in\{\texttt{a},\texttt{b}\}$, we have
    \begin{equation}\label{eq:primitive-growth-rate}
        \lim_{k\to\infty} |\rho_n^k(u)|^{\frac{1}{k}} = \beta_n.
    \end{equation}

    We observe that the $2$-dimensional substitution $\omega_n$ 
    is an extension of the direct product $\rho_n\times \rho_n$ of the 
    one-dimensional substitution $\rho_n$ with itself.
    By extension, we mean the existence of a map 
    \[
        \begin{array}{rccl}
            \zeta:&\Tcal_n&\to&\{\texttt{a},\texttt{b}\}\times\{\texttt{a},\texttt{b}\}\\
            &t&\mapsto&
            \begin{cases}
                (\texttt{a}, \texttt{a}) & \text{ if } t \in J_n,\\
                (\texttt{b}, \texttt{a}) & \text{ if } t \in B_n\cup Y_n\cup G_n,\\
                (\texttt{a}, \texttt{b}) & \text{ if } t \in \widehat{B_n}\cup \widehat{Y_n}\cup \widehat{G_n},\\
                (\texttt{b}, \texttt{b}) & \text{ if } t \in W_n,\\
            \end{cases}
        \end{array}
    \]
    such that $(\rho_n\times \rho_n)\circ\zeta=\zeta\circ\omega_n$.

    Since $\omega_n$ is primitive,
    the dominant eigenvalue $\lambda$ of the incidence matrix of the substitution $\omega_n$
    is equal to the growth rate of $\scarea(\omega_n^k(t))$
    as $k\to\infty$, 
    where $t\in\Tcal_n$ is any tile
    and $\scarea(p)$ denotes
    the cardinality of the support of a rectangular pattern $p\in(\Tcal_n)^{*^2}$.
    Let $t\in\Tcal_n$ such that $\zeta(t)=(t_1,t_2)$ for some $t_1,t_2\in\{\texttt{a},\texttt{b}\}$.
    Since $\zeta$ is a tile to tile map, it preserves the area.
    Thus, we have
    \begin{align*}
        \lambda 
        &= \lim_{k\to\infty} \scarea(\omega_n^k(t))^{\frac{1}{k}}
        = \lim_{k\to\infty} \scarea(\zeta\circ\omega_n^k(t))^{\frac{1}{k}}
        = \lim_{k\to\infty} \scarea((\rho_n\times \rho_n)^k\circ\zeta(t))^{\frac{1}{k}}\\
        &= \lim_{k\to\infty} \scarea((\rho_n\times \rho_n)^k(\zeta(t)))^{\frac{1}{k}}
        = \lim_{k\to\infty} \scarea((\rho_n\times \rho_n)^k(t_1,t_2))^{\frac{1}{k}}\\
        &= \lim_{k\to\infty} \left( |\rho_n^k(t_1)|\cdot |\rho_n^k(t_2)| \right)^{\frac{1}{k}}
        = \lim_{k\to\infty} |\rho_n^k(t_1)|^{\frac{1}{k}} \cdot
          \lim_{k\to\infty} |\rho_n^k(t_2)|^{\frac{1}{k}}
        \overset{\eqref{eq:primitive-growth-rate}}{=} \beta_n \cdot \beta_n
        = \beta_n^2.
    \end{align*}
    Therefore, the incidence matrices of the substitutions $\omega_n$ and
    $\rho_n\times \rho_n$ have the same Perron-Frobenius dominant eigenvalue,
    and it is equal to $\beta_n^2$.

    The inflation factor is the factor of the homogeneous dilation associated with
    the stone inflation constructed from the direct product $\rho_n\times \rho_n$
    \cite[\S~5.6]{MR3136260} (for example, a stone inflation for $\rho_4\times
    \rho_4$ is shown in Figure~\ref{fig:global-stone-inflation} when $n=4$).
    The inflation factor 
    of the stone inflation of $\rho_n\times \rho_n$
    is $\beta_n$ as it multiplies distances between
    points by $\beta_n$ and the areas by $\beta_n^2$.
\end{proof}

\begin{THEOREMII}
    \MainTheoremII
\end{THEOREMII}

\begin{proof} 
    From Lemma~\ref{lem:primitivity}, $\omega_n$ is primitive.
    The Perron--Frobenius dominant eigenvalue of the incidence matrix of
    $\omega_n$ and its inflation factor are computed in
    Lemma~\ref{lem:inflation-factor}.
\end{proof} 

From Perron--Frobenius theorem, the primitivity of the substitution $\omega_n$
implies that every Wang tile in $\Tcal_n$ appears with positive frequency in a
configuration in the substitutive subshift $\Xcal_{\omega_n}$ generated by the
substitution $\omega_n$.  The frequencies of the tiles is given by the entries of 
the right-eigenvector of the incidence matrix of $\omega_n$ normalized so that
the sum of its entries is 1.

\section{$\Omega_n$ is minimal}
\label{sec:minimal}

The goal of this section is to prove that $\Omega_n$ is minimal.
To prove minimality, we need more notions.
We use the method proposed in \cite[\S 3.3]{labbe_three_2020}.

\subsection{A criterion for minimality of a self-similar subshift}

Recall that a subshift $X$ is self-similar if $X=\overline{\omega(X)}^\sigma$
for some expansive $d$-dimensional substitution; 
see Definition~\ref{def:expansive} and Definition~\ref{def:self-similar}.
First we recall Lemma 3.8 from \cite{labbe_three_2020}.

\begin{lemma}\label{lem:substitutive-contains-self-similar-part}
    Let $\omega:\Acal\to\Acal^{*^d}$ be an expansive and primitive $d$-dimensional
    morphism. Let $X\subseteq\Acal^{\Z^d}$ be a nonempty subshift 
    such that $X=\overline{\omega(X)}^{\sigma}$. Then 
    $\Xcal_\omega\subseteq X$.
\end{lemma}

\begin{proof}
    The language of $X$ is also self-similar satisfying
    $\Lcal(X)=\Lcal(\omega(\Lcal(X)))$.
    Recursively, $\Lcal(X)=\Lcal(\omega^m(\Lcal(X)))$
    for every $m\geq1$. Since $X$ is nonempty, there exists a letter
    $a\in\Acal$ such that for all $m\geq1$, the $d$-dimensional word $\omega^m(a)$ is in the
    language $\Lcal(X)$.
    From the primitivity of $\omega$, there exists $m\geq1$ such that
    $\omega^m(a)$ contains an occurrence of every letter of the alphabet $\Acal$.
    Therefore, every letter is in $\Lcal(X)$ and
    the $d$-dimensional word $\omega^m(a)$ is in the
    language $\Lcal(X)$ for all letters $a\in\Acal$ and all $m\geq1$.
    So we conclude that 
    $\Lcal(\Xcal_\omega)\subseteq\Lcal(X)$ and
    $\Xcal_\omega\subseteq X$.
\end{proof}

Proving that a self-similar $d$-dimensional subshift $X$ satisfying
$X=\overline{\omega(X)}^{\sigma}$ is equal to $\Xcal_\omega$ can be tricky.
As illustrated in the following example, it depends on the combinatorics of the
substitution.

\begin{example}\label{ex:sub2d-on-abc}
Consider the following 2-dimensional substitution $\nu$ over alphabet $\{a,b,c\}$:
\[
\nu:
a\mapsto \left(\begin{array}{ccccc}
				c&c&c&c&c\\
				c&c&c&c&c\\
				c&c&a&c&c\\
        \end{array}\right),\quad
b\mapsto \left(\begin{array}{ccccc}
				c&c&b&c&a\\
				c&c&c&c&c\\
				c&c&c&c&c\\
        \end{array}\right),\quad
c\mapsto \left(\begin{array}{ccccc}
				c&c&a&c&c\\
				c&c&c&b&c\\
				c&c&c&c&c\\
        \end{array}\right).
\]
We may observe that
the vertical domino $\left(\begin{smallmatrix}a\\b\end{smallmatrix}\right)$
does not belong to the language of the substitutive subshift $\Xcal_\nu$,
since it does not appear in any of the $k$-th image of any letter under the substitution.
But one can see that the vertical domino 
$\left(\begin{smallmatrix}a\\b\end{smallmatrix}\right)$
is preserved by the substitution. 
Therefore, there exists a configuration $x$ containing a single
vertical domino 
$\left(\begin{smallmatrix}a\\b\end{smallmatrix}\right)$
which is fixed by the substitution. 
Thus, we have
\[
    \varnothing\neq \Xcal_\nu\subsetneq 
    \Xcal_\nu\cup\{\sigma^n(x)\,|\,n\in\Z^2\}.
\] 
The subshift $\Xcal_\nu\cup\{\sigma^n(x)\,|\,n\in\Z^2\}$ is self-similar, but it is not
minimal because it contains a proper nonempty subshift.
\end{example}

Therefore, to conclude that we have the equality
$\Xcal_\omega=X$ for a self-similar subshift $X$, it is convenient to
consider 
the domino patterns of size $1\times 2$ and $2\times 1$
straddling the images of the two letters of a domino
as well as 
the $2\times 2$ patterns 
straddling the images of the four letters of $2\times 2$ pattern.
More precisely, we need to consider the following directed graphs:
\begin{itemize}
    \item
Let $G_\omega^{2\times2}=(V_\omega^{2\times2},E_\omega^{2\times2})$ be the
directed graph whose vertices and edges are
\begin{align*}
    V_\omega^{2\times2} &= \left\{
        \left(\begin{smallmatrix}
            a&b\\
            c&d
        \end{smallmatrix}\right)
        \in\Acal^{2\times 2}
        \mid
        a\equiv_1 b,
        c\equiv_1 d,
        a\equiv_2 c,
        b\equiv_2 d
        \right\},\\
    E_\omega^{2\times2} &= \left\{
        \left(\begin{smallmatrix}
            e&f\\
            g&h
        \end{smallmatrix}\right)
        \to
        \left(\begin{smallmatrix}
            a&b\\
            c&d
        \end{smallmatrix}\right)
        \middle|
        \begin{array}{l}
        a \text{ is the bottom right letter of } \omega(e),\\
        b \text{ is the bottom left letter of }  \omega(f),\\
        c \text{ is the top right letter of }   \omega(g),\\
        d \text{ is the top left letter of }    \omega(h)\\
        \end{array}
        \right\}.
\end{align*}
    \item
Let $G_\omega^{2\times1}=(V_\omega^{2\times1},E_\omega^{2\times1})$ be the
directed graph whose vertices and edges are 
\begin{align*}
    V_\omega^{2\times1} &= \left\{
        \left(\begin{smallmatrix}
            a&b
        \end{smallmatrix}\right)
        \in\Acal^{2\times1}
        \mid
        a\equiv_1 b
        \right\},\\
    E_\omega^{2\times1} &= \left\{
        \left(\begin{smallmatrix}
            e&f
        \end{smallmatrix}\right)
        \to
        \left(\begin{smallmatrix}
            a&b
        \end{smallmatrix}\right)
        \middle|
        \begin{array}{l}
        \text{there exists an integer $j$ such that $0\leq j < \height(\omega(e))$ and}\\
        a \text{ is the letter in the $j$-th row in the right-most column of } \omega(e),\\
        b \text{ is the letter in the $j$-th row in the left-most column of } \omega(f)\\
        \end{array}
        \right\}.
\end{align*}
    \item
Let $G_\omega^{1\times2}=(V_\omega^{1\times2},E_\omega^{1\times2})$ be the
directed graph whose vertices and edges are 
\begin{align*}
    V_\omega^{1\times2} &= \left\{
        \left(\begin{smallmatrix}
            a\\
            c
        \end{smallmatrix}\right)
        \in\Acal^{1\times2}
        \mid
        a\equiv_2 c
        \right\},\\
    E_\omega^{1\times2} &= \left\{
        \left(\begin{smallmatrix}
            e\\
            g
        \end{smallmatrix}\right)
        \to
        \left(\begin{smallmatrix}
            a\\
            c
        \end{smallmatrix}\right)
        \middle|
        \begin{array}{l}
        \text{there exists an integer $i$ such that $0\leq i < \width(\omega(e))$ and}\\
        a \text{ is the letter in the $i$-th column in the bottom-most row of } \omega(e),\\
        c \text{ is the letter in the $i$-th column in the top-most row of } \omega(g)\\
        \end{array}
        \right\}.
\end{align*}
\end{itemize}
    Finally, for every directed graph $G=(V,E)$, we define
    the set of \defn{recurrent} vertices, that is, those belonging to a cycle
    of the graph:
    \[
        \RecurrentVertices(G)
            =
            \{v\in V
            \mid
            v \text{ belongs to a cycle of }G\}.
    \]

    \begin{figure}[h]
	\setlength\tabcolsep{15pt}
        \begin{tabular}{ccc}
	\begin{tikzpicture}[>=latex,line join=bevel,]
    \tikzset{every loop/.style={min distance=5mm,looseness=5}}
		\node (node_0) at (0,2)  {$\left(\begin{array}{rr}c & c \\a & c\end{array}\right)$};
		\node (node_1) at (0,0)  {$\left(\begin{array}{rr}c & c \\c & c\end{array}\right)$};
            \draw[->] (node_1) edge [loop left] ();
        \draw [black,->] (node_0) -- (node_1);
	\end{tikzpicture}&
	\begin{tikzpicture}[>=latex,line join=bevel,]
    \tikzset{every loop/.style={min distance=5mm,looseness=5}}
		\node (node_0) at (0,2)  {$\left(\begin{array}{rr}a & c\end{array}\right)$};
		\node (node_1) at (0,0)  {$\left(\begin{array}{rr}c & c\end{array}\right)$};
		\draw[->] (node_1) edge [loop left] ();
        \draw [black,->] (node_0) -- (node_1);
	\end{tikzpicture}&
	\begin{tikzpicture}[>=latex,line join=bevel,scale=1.5]
    \tikzset{every loop/.style={min distance=5mm,looseness=5}}
	\node (ab) at (0,1.3)   {$\left(\begin{array}{c}a\\b \end{array}\right)$};
	\node (aa) at (4,1.3)   {$\left(\begin{array}{c}a\\a \end{array}\right)$};
	\node (ac) at (4,0)   {$\left(\begin{array}{c}a\\c \end{array}\right)$};
	\node (cb) at (0,0)     {$\left(\begin{array}{c}c\\b \end{array}\right)$};
	\node (cc) at (2,1.3)     {$\left(\begin{array}{c}c\\c \end{array}\right)$};
	\node (ca) at (2,0)     {$\left(\begin{array}{c}c\\a \end{array}\right)$};
    \draw[->] (ab) edge [loop left] ();
	\draw[->] (cb) edge [loop left] ();
	\draw[->] (cc) edge [loop above] ();
	\draw [black,->] (aa) -- (cc);
	\draw [black,->] (ab) -- (cc);
	\draw [black,->] (ab) -- (ca);
	\draw [black,->] (ac) -- (cc);
	\draw [black,->] (cb) -- (ca);
	\draw [black,->] (cb) -- (cc);
	\draw [black,->,bend left=10] (aa) to (ac);
	\draw [black,->,bend left=10] (ac) to (aa);
	\draw [black,->,bend left=10] (cc) to (ca);
	\draw [black,->,bend left=10] (ca) to  (cc);
	\end{tikzpicture}
            \\
            $G_\nu^{2\times2}$ &
            $G_\nu^{2\times1}$ &
            $G_\nu^{1\times2}$ \\
        \end{tabular}
\caption{The graphs $G_\nu^{2\times2}$, $G_\nu^{2\times1}$ and $G_\nu^{1\times2}$
    for the substitution $\nu$.}
\label{fig:the-three-seeds-graphs}
\end{figure}

\begin{example}
    The graphs
    $G_\nu^{2\times2}$,
    $G_\nu^{2\times1}$ and
    $G_\nu^{1\times2}$
    for the 2-dimensional substitution $\nu$ defined in
    Example~\ref{ex:sub2d-on-abc}
    are shown in Figure~\ref{fig:the-three-seeds-graphs}.
    The recurrent vertices of the graphs are:
\begin{align*}
        \RecurrentVertices(G_\nu^{2\times2}) &=
\left\{ \left(\begin{smallmatrix}c&c \\c&c \end{smallmatrix}\right)\right\}\\
        \RecurrentVertices(G_\nu^{2\times1}) &= 
\left\{ \left(\begin{smallmatrix}c&c \end{smallmatrix}\right)\right\}\\
        \RecurrentVertices(G_\nu^{1\times2}) &= 
\left\{
\left(\begin{smallmatrix}a\\b \end{smallmatrix}\right),
\left(\begin{smallmatrix}c\\c \end{smallmatrix}\right),
\left(\begin{smallmatrix}a\\a \end{smallmatrix}\right),
\left(\begin{smallmatrix}c\\b \end{smallmatrix}\right),
\left(\begin{smallmatrix}c\\a \end{smallmatrix}\right),
\left(\begin{smallmatrix}a\\c \end{smallmatrix}\right)\right\}
\end{align*}
In particular, we observe that
the vertical domino
$\left(\begin{smallmatrix}a\\b \end{smallmatrix}\right)$
    belongs to a cycle of $G_\nu^{1\times2}$,
    even though it is not in the language $\Lcal(\Xcal_\nu)$.
\end{example}

    The recurrent vertices of the three graphs
    $G_\omega^{2\times2}$,
    $G_\omega^{2\times1}$ and
    $G_\omega^{1\times2}$
    provide a criteria for the minimality of a self-similar subshift
    $X=\shiftclosure{\omega(X)}$.
    Lemma 3.7 and Lemma 3.9 from \cite{labbe_three_2020}
    gave hypothesis under which an expansive and primitive 2-dimensional
    substitution has a unique nonempty self-similar subshift.
    The following lemma is a relaxed version which
    allows to conclude that a self-similar subshift is minimal 
    even when the 2-dimensional substitution admits more than one self-similar
    subshift (some made of configurations which are not uniformly recurrent).

\begin{lemma}\label{lem:criterion-for-minimality}
    Let $X=\shiftclosure{\omega(X)}$ be a nonempty self-similar subshift
    where $\omega:\Acal\to\Acal^{*^d}$ is an expansive and primitive $2$-dimensional
    morphism.
    The following are equivalent:
    \begin{enumerate}[(i)]
        \item $\Lcal(X)\cap\RecurrentVertices(G^s_\omega)\subset\Lcal(\Xcal_\omega)$
              for every size $s\in\{2\times2, 2\times1, 1\times2\}$,
    \item $X=\Xcal_\omega$,
    \item $X$ is minimal.
    \end{enumerate}
\end{lemma}

An element $u\in\Acal^\bn$ is called a
\defn{$d$-dimensional word} of \defn{size} $\bn=(n_1,\dots,n_d)\in\N^d$
on the alphabet~$\Acal$.
We use the notation $\scsize(u)=\bn$ when necessary.

\begin{proof}
    Assume that $X=\overline{\omega(X)}^{\sigma}$ for some $\varnothing\neq X\subseteq\Acal^{\Z^d}$.

    (i) $\implies$ (ii)
    From Lemma~\ref{lem:substitutive-contains-self-similar-part},
    we have $\Xcal_\omega\subseteq X$.
    Let $z\in\Lcal(X)$. 
    We want to show that $z\in\Lcal(\Xcal_\omega)$. 
    Since $\omega$ is expansive,
    let $m\in\N$ such that the image of every letter
    $a\in\Acal$ by $\omega^m$ is larger than $z$, that is, 
    $\scsize(\omega^m(a))\geq\scsize(z)$ for all $a\in\Acal$.
    We have $z\in\Lcal(X)=\Lcal\left(\omega^m(\Lcal(X))\right)$.
    By the choice of $m$, $z$ cannot overlap more than two blocks
    $\omega^m(a)$ in the same direction. Thus, there exists a word $u\in\Lcal(X)$ of
    size 
    $1\times 1$,
    $2\times 1$,
    $1\times 2$ or
    $2\times 2$ such that $z$ is a subword of $\omega^m(u)$.
    If $u$ is of size $1\times 1$, then $z\in\Lcal(\Xcal_\omega)$.
    We may assume that the word $u$ has the smallest possible rectangular size
    $s\in \{2\times 1, 1\times 2, 2\times 2\}$.

    We have $u\in V_\omega^{s}$. Since $u\in\Lcal(X)$ and $X$ is self-similar,
    there exists a sequence $(u_k)_{k\in\N}$ with $u_k\in V_\omega^{s}\cap\Lcal(X)$ 
    for all $k\in\N$
    such that
    \[
        \cdots\rightarrow
        u_{k+1} \rightarrow 
        u_k \rightarrow 
        \cdots\rightarrow
        u_1 \rightarrow u_0 = u
    \]
    is a left-infinite path in the graph $G_\omega^{s}$.
    Since $V_\omega^{s}$ is finite, there exist some $k,k'\in\N$ with $k<k'$
    such that $u_k=u_{k'}$.
    Thus, $u_k\in\RecurrentVertices(G^s_\omega)$
    and $u$ is a subword of $\omega^k(u_k)$.
    From the hypothesis, we have $u_k\in\Lcal(\Xcal_\omega)$.
    Since $\omega$ is primitive, there exists $\ell$ such that 
    $u_k$ is a subword of $\omega^\ell(a)$ for every $a\in\Acal$.
    Therefore, $z$ is a subword of $\omega^{m+k+\ell}(a)$ for every $a\in\Acal$.
    Then $z\in\Lcal(\Xcal_\omega)$ and $\Lcal(X)\subseteq\Lcal(\Xcal_\omega)$.
    Thus, $X\subseteq\Xcal_\omega$ and $X=\Xcal_\omega$.

    (ii) $\implies$ (i)
    If $X=\Xcal_\omega$,
    then $\Lcal(X)=\Lcal(\Xcal_\omega)$.
    Thus, $\Lcal(X)\cap\RecurrentVertices(G^s_\omega)\subset\Lcal(X)=\Lcal(\Xcal_\omega)$
    for every size $s\in\{2\times2, 2\times1, 1\times2\}$.

    (ii) $\implies$ (iii)
    The substitutive shift of $\omega$ is well-defined since $\omega$ is expansive
    and it is minimal since $\omega$ is primitive 
    using standard arguments \cite[\S 5.2]{MR2590264}.

    (iii) $\implies$ (ii)
    From Lemma~\ref{lem:substitutive-contains-self-similar-part},
    we have $\Xcal_\omega\subseteq X$.
    Since $X$ is minimal, we conclude that $\Xcal_\omega=X$.
\end{proof}

\subsection{The Wang shift $\Omega_n$ is minimal when $n\geq2$}

The proof that the Wang shift $\Omega_n$ is minimal needs to be split into two cases.
When $n=1$, configurations in $\Omega_1$ have consecutive rows containing junction tiles
whereas this does not happen when $n\geq2$. This affects the language of patterns
of vertical domino support. In particular, a vertical domino made of two
junction tiles may appear in the language of $\Omega_n$ when $n=1$.
In this section, we consider the case $n\geq2$.

\begin{lemma}\label{lem:language-vertical-dominoes}
    Let $n\geq 2$ be an integer.
    The following vertical dominoes appear in the language of the substitutive
    subshift $\Xcal_{\omega_n}$:
\begingroup
\allowdisplaybreaks
    \begin{align*}
        \Lcal_{1\times 2}(\Xcal_{\omega_n})
    &\supseteq
    \left\{
        \left(\begin{array}{c} j_n^{0,1,0,0}\\ \widehat{g_n^{n-1}} \end{array}\right),
        \left(\begin{array}{c} j_n^{0,1,0,1}\\ \widehat{g_n^{n-1}} \end{array}\right),
        \left(\begin{array}{c} j_n^{0,1,0,0}\\ \widehat{y_n^{n-1}} \end{array}\right),
        \left(\begin{array}{c} j_n^{0,1,0,1}\\ \widehat{y_n^{n-1}} \end{array}\right),
        \left(\begin{array}{c} j_n^{0,1,1,1}\\ \widehat{y_n^{n-1}} \end{array}\right)
    \right\}\\
    &\quad\cup
    \left\{
        \left(\begin{array}{c} j_n^{1,1,0,1}\\ \widehat{g_n^{n}} \end{array}\right),
        \left(\begin{array}{c} j_n^{1,1,0,1}\\ \widehat{y_n^{n}} \end{array}\right),
        \left(\begin{array}{c} j_n^{1,1,1,1}\\ \widehat{y_n^{n}} \end{array}\right),
        \left(\begin{array}{c} j_n^{0,0,0,0}\\ \widehat{b_n^{n-1}} \end{array}\right),
        \left(\begin{array}{c} j_n^{0,0,0,1}\\ \widehat{b_n^{n-1}} \end{array}\right)
    \right\}\\
    &\quad\cup
    \left\{
        \left(\begin{array}{c}
            g_n^{i-1}\\
            w_n^{i,n}
        \end{array}\right),
        \left(\begin{array}{c}
            g_n^{i}\\
            w_n^{i,n}
        \end{array}\right)
        \middle|
        \;
        1\leq i\leq n
    \right\}\\
    &\quad\cup
    \left\{
        \left(\begin{array}{c}
            b_n^{i-1}\\
            w_n^{i,n-1}
        \end{array}\right)
        \middle|
        \;
        1\leq i\leq n
    \right\}
    \cup
    \left\{
        \left(\begin{array}{c}
            b_n^{i}\\
            w_n^{i,n-1}
        \end{array}\right)
        \middle|
        \;
        1\leq i\leq n-1
    \right\}\\
    &\quad\cup
    \left\{
        \left(\begin{array}{c}
            y_n^{i-1}\\
            w_n^{i,n}
        \end{array}\right)
        \middle|
        \;
        2\leq i\leq n
    \right\}
    \cup
    \left\{
        \left(\begin{array}{c}
            y_n^{i}\\
            w_n^{i,n}
        \end{array}\right)
        \middle|
        \;
        1\leq i\leq n
    \right\}.
    \end{align*}
\endgroup
\end{lemma}

\begin{proof}
    We show that every vertical domino listed above appear in the image of some
    tile under the application of the 2-dimensional substitution $\omega_n$.
    Below, we use the notation $p \xrightarrow{\omega_n} q$ to denote that $q$ is a
    pattern appearing in the image $\omega_n(p)$.
    We have
\begingroup
\allowdisplaybreaks
    \begin{align*}
        &j_n^{0,1,0,1}
        \xrightarrow{\omega_n}
        \left(\begin{array}{c}\widehat{b_n^{0}}\\ j_n^{0,0,0,0} \end{array}\right)
        \xrightarrow{\omega_n}
        \left(\begin{array}{ccccc}
            {\bf j_n^{1,1,0,1}}     & y_n^1 & y_n^2 & \dots & y_n^n\\
            {\bf \widehat{g_n^{n}}} & w_n^{1,n} & w_n^{2,n} & \dots & w_n^{n,n}
        \end{array}\right),
        \\
        &j_n^{0,1,0,1}
        \xrightarrow{\omega_n}
        \left(\begin{array}{c}\widehat{w_n^{1,1}}\\ b_n^0 \end{array}\right)
        \xrightarrow{\omega_n}
        \left(\begin{array}{ccccc}
            {\bf j_n^{1,1,1,1}}     & y_n^1 & y_n^2 & \dots & y_n^{n-1}\\
            {\bf \widehat{y_n^{n}}} & w_n^{2,n} & w_n^{3,n} & \dots & w_n^{n,n}
        \end{array}\right),
        \\
        &j_n^{0,1,0,1}
        \xrightarrow{\omega_n}
        \left(\begin{array}{c}\widehat{w_n^{1,1}}\\ b_n^0 \end{array}\right)
        \xrightarrow{\omega_n}
        \left(\begin{array}{c}j_n^{1,1,1,1}\\ \widehat{y_n^{n}} \end{array}\right)
        \xrightarrow{\omega_n}
        \left(\begin{array}{ccccc}
            {\bf j_n^{0,0,0,0}}       & b_n^0 & b_n^1 & \dots & b_n^{n-1}\\
            {\bf \widehat{b_n^{n-1}}} & w_n^{1,n-1} & w_n^{2,n-1} & \dots & w_n^{n,n-1}
        \end{array}\right),
        \\
        &g_n^0
        \xrightarrow{\omega_n}
        \left(\begin{array}{c}\widehat{y_n^1}\\ j_n^{0,1,1,1} \end{array}\right)
        \xrightarrow{\omega_n}
        \left(\begin{array}{c}\bf j_n^{0,0,0,1}\\ \bf \widehat{b_n^{n-1}} \end{array}\right)
        \xrightarrow{\omega_n}
        \left(\begin{array}{ccccc}
            {\bf j_n^{0,1,0,0}}       & b_n^1 & b_n^2 & \dots & b_n^{n-1}\\
            {\bf \widehat{g_n^{n-1}}} & w_n^{1,n-1} & w_n^{2,n-1} & \dots & w_n^{n-1,n-1}
        \end{array}\right).
    \end{align*}
\endgroup
    Also,
    \begin{align*}
        &g_n^1
        \xrightarrow{\omega_n}
        \left(\begin{array}{c}\widehat{g_n^1}\\ j_n^{0,1,0,1}\\ \end{array}\right)
        \xrightarrow{\omega_n}
        \left(\begin{array}{c}\bf j_n^{0,1,0,1}\\ \bf \widehat{g_n^{n-1}} \end{array}\right).
    \end{align*}
    Since $n\geq 2$, we have
\begingroup
\allowdisplaybreaks
    \begin{align*}
        &j_n^{1,1,1,1}
        \xrightarrow{\omega_n}
        \left(\begin{array}{c}\widehat{b_n^0}\\ j_n^{0,0,0,0} \end{array}\right)
        \xrightarrow{\omega_n}
        \left(\begin{array}{c}y_n^1\\ w_n^{1,n} \end{array}\right)
        \xrightarrow{\omega_n}
        \left(\begin{array}{c}\bf j_n^{0,1,0,0}\\ \bf\widehat{y_n^{n-1}} \end{array}\right),
        \\
        &w_n^{1,1}
        \xrightarrow{\omega_n}
        \left(\begin{array}{c}w_n^{2,2}\\ y_n^1 \end{array}\right)
        \xrightarrow{\omega_n}
        \left(\begin{array}{c}\bf j_n^{0,1,0,1}\\ \bf\widehat{y_n^{n-1}} \end{array}\right),
        \\
        &g_n^0
        \xrightarrow{\omega_n}
        \left(\begin{array}{c}w_n^{2,1}\\ \widehat{b_n^1} \end{array}\right)
        \xrightarrow{\omega_n}
        \left(\begin{array}{c}\bf j_n^{1,1,0,1}\\ \bf \widehat{y_n^{n}} \end{array}\right),
        \\
        &w_n^{1,1}
        \xrightarrow{\omega_n}
        \left(\begin{array}{c}\widehat{y_n^1}\\ j_n^{1,1,1,1} \end{array}\right)
        \xrightarrow{\omega_n}
        \left(\begin{array}{c}g_n^0\\ w_n^{1,n} \end{array}\right)
        \xrightarrow{\omega_n}
        \left(\begin{array}{c}\bf j_n^{0,1,1,1}\\ \bf \widehat{y_n^{n-1}} \end{array}\right).
    \end{align*}
\endgroup
\end{proof}

\begin{lemma}\label{lem:language-2x2-seeds}
    The following four $2\times2$ patterns belong
    to the language of the substitutive subshift 
    $\Lcal(\Xcal_{\omega_n})$:
    \begin{align*}
    &\left\{
        \left(\begin{array}{cc}
            b_n^{n-1}     & j_n^{0,0,0,0}\\
            w_n^{n-1,n-1} & \widehat{b_n^{n-1}}
        \end{array}\right),
        \left(\begin{array}{cc}
            g_n^{n-1}     & j_n^{0,1,0,1}\\
            w_n^{n,n} & \widehat{g_n^{n-1}}
        \end{array}\right),
        \left(\begin{array}{cc}
            b_n^{n-1}     & j_n^{0,1,0,0}\\
            w_n^{n,n-1} & \widehat{g_n^{n-1}}
        \end{array}\right),
        \left(\begin{array}{cc}
            g_n^{n-1}     & j_n^{0,0,0,1}\\
            w_n^{n-1,n} & \widehat{b_n^{n-1}}
        \end{array}\right)
    \right\}\\
    &=
    \left\{
        \raisebox{-10mm}{
        \begin{tikzpicture}
        \tikzstyle{every node}=[font=\scriptsize]
            \tilelabelinsideH{\ourColorBlue}{0}{1}{00n}{111}{00\nunder}{11n}
        \tilelabelinside{white}  {0}{0}{11n}{11n}{11\nunder}{11\nunder}
        \tileJunctionInsideOOOO  {1}{1}{000}{000}{00n}{00n}
            \tilelabelinsideV{\ourColorBlue}{1}{0}{111}{00n}{11n}{00\nunder}
        \end{tikzpicture}
        },
        \raisebox{-10mm}{
        \begin{tikzpicture}
        \tikzstyle{every node}=[font=\scriptsize]
        \tileHgreenInside{0}{1}{01n}{111}{00\nunder}{11\nbar}
        \tilelabelinside{white}  {0}{0}{11\nbar}{11\nbar}{11n}{11n}
        \tileJunctionInsideOOII  {1}{1}{001}{001}{01n}{01n}
        \tileVgreenInside{1}{0}{111}{01n}{11\nbar}{00\nunder}
        \end{tikzpicture}
        },
        \raisebox{-10mm}{
        \begin{tikzpicture}
        \tikzstyle{every node}=[font=\scriptsize]
            \tilelabelinsideH{\ourColorBlue}{0}{1}{00n}{111}{00\nunder}{11n}
        \tilelabelinside{white}  {0}{0}{11\nbar}{11n}{11n}{11\nunder}
        \tileJunctionInsideOOOI  {1}{1}{001}{000}{00n}{01n}
        \tileVgreenInside{1}{0}{111}{01n}{11\nbar}{00\nunder}
        \end{tikzpicture}
        },
        \raisebox{-10mm}{
        \begin{tikzpicture}
        \tikzstyle{every node}=[font=\scriptsize]
        \tileHgreenInside{0}{1}{01n}{111}{00\nunder}{11\nbar}
        \tilelabelinside{white}  {0}{0}{11n}{11\nbar}{11\nunder}{11n}
        \tileJunctionInsideOOIO  {1}{1}{000}{001}{01n}{00n}
        \tilelabelinsideV{\ourColorBlue}{1}{0}{111}{00n}{11n}{00\nunder}
        \end{tikzpicture}
        }
    \right\}
    \subset\Lcal_{2\times 2}(\Xcal_{\omega_n}).
    \end{align*}
\end{lemma}

\begin{proof}
    We show that every pattern listed above appear in the image of some
    tile under some repeated application of the 2-dimensional substitution $\omega_n$.
    Below, we use the notation $p \xrightarrow{\omega_n} q$ to denote that $q$ is a
    pattern appearing in the image $\omega_n(p)$.
    The four patterns can be obtain in few steps when applying the substitution $\omega_n$
    on the tiles $j_n^{1,1,1,1}$ and $y_n^1$.
    We have
    \begin{center}
        \includegraphics[width=.8\linewidth]{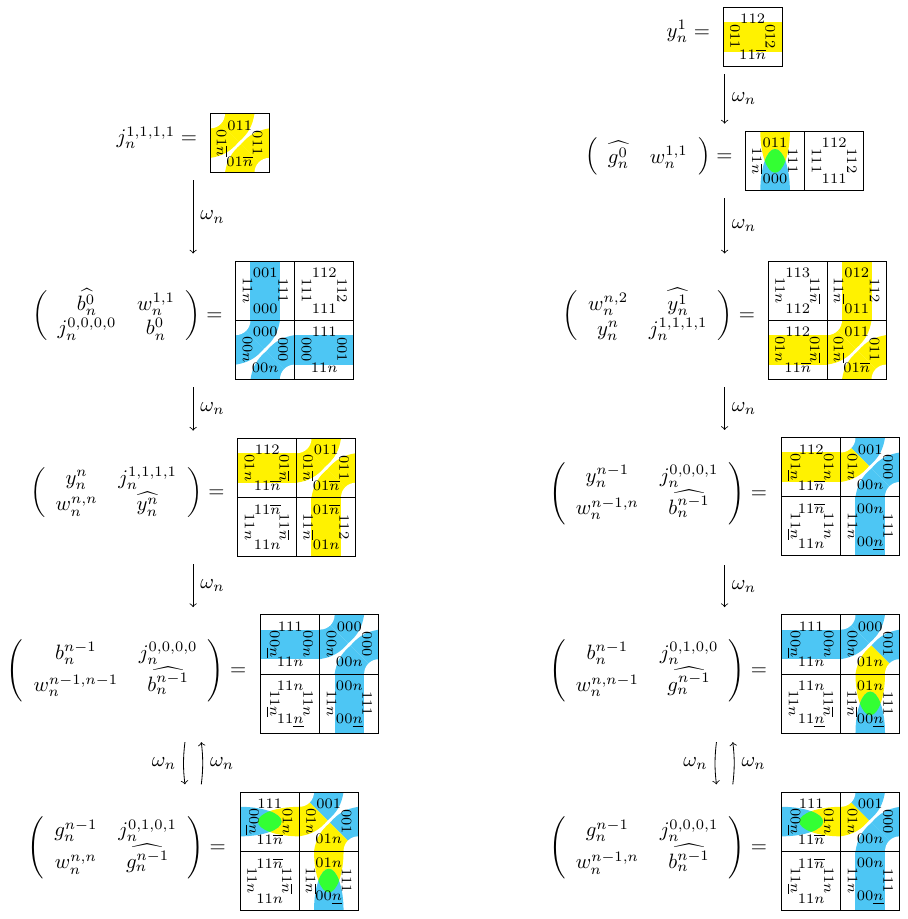}
    \end{center}
\end{proof}

\begin{lemma}\label{lem:two-illegal-vertical-dominoes}
    Let $n\geq 2$ be an integer.
    The following two vertical dominoes are illegal in $\Omega_n$:
    \[
        \raisebox{-9mm}{
        \begin{tikzpicture}[auto,scale=1.2]
            \tikzstyle{every node}=[font=\footnotesize]
            \tileJunctionInsideIIII{0}{1}{011}{011}{01\nbar}{01\nbar}
            \tileVgreenInside{0}{0}{111}{01\nbar}{11\nbar}{00n}
        \end{tikzpicture}}
        =
        \left(\begin{array}{c}
            j_n^{1,1,1,1}\\
            \widehat{g_n^{n}}
        \end{array}\right)
        \notin\Lcal(\Omega_n)
        \qquad
        \text{ and }
        \qquad
        \raisebox{-9mm}{
        \begin{tikzpicture}[auto,scale=1.2]
            \tikzstyle{every node}=[font=\footnotesize]
            \tileJunctionInsideOIII{0}{1}{001}{011}{01\nbar}{01n}
            \tileVgreenInside{0}{0}{111}{01n}{11\nbar}{00\nunder}
        \end{tikzpicture}}
        =
        \left(\begin{array}{c}
            j_n^{0,1,1,1}\\
            \widehat{g_n^{n-1}}
        \end{array}\right)
        \notin\Lcal(\Omega_n).
    \]
\end{lemma}

\begin{proof}
    Let $c\in\Omega_n$ be a valid configuration.
    Let $A,B:\Z\to\Z$ be the two increasing maps
    from Lemma~\ref{lem:return-time-to-junction-tiles} 
    such that $c^{-1}(J_n)=A(\Z)\times B(\Z)$.

    Suppose that $j_n^{1,1,1,1}$ appears at position $\bl=(\ell_1,\ell_2)\in\Z^2$
    and that $\widehat{g_n^n}$ appears at position $(\ell_1,\ell_2-1)$ in $c$.
    Let $\bk=(k_1,k_2)\in\Z^2$ be such that 
    $A(k_1)\leq \ell_1<A(k_1+1)$ and
    $B(k_2)\leq \ell_2<B(k_2+1)$.
    Since $j_n^{1,1,1,1}$ is a junction tile, we must have 
    $A(k_1)=\ell_1$ and $B(k_2)=\ell_2$.
    At position $(\ell_1,\ell_2-2)$ there must be
    a blue tile $\widehat{b_n^{n-1}}$, since only this tile
    has top label $00n$ when $n\geq2$.
    The current situation is illustrated below.
    \begin{center}
        \includegraphics{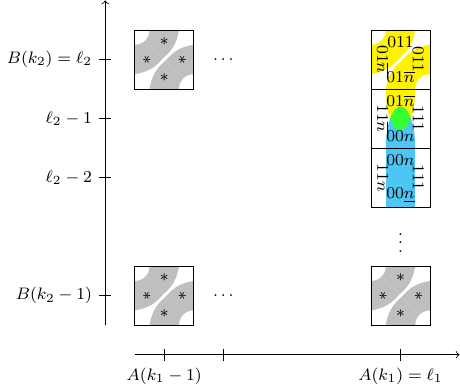}
    \end{center}
    Consider the return blocks with support 
    $[A(k_1-1),A(k_1))\times[B(k_2),B(k_2+1))$.
    It has label $01\nbar$ at the far right of its bottom row.
    From Lemma~\ref{lem:unique-horizontal-strip},
    the width of this return block cannot be $n$, so it
    has to be
    \[
        A(k_1)-A(k_1-1)=n+1.
    \]
    Now consider the return blocks with support 
    $[A(k_1-1),A(k_1))\times[B(k_2-1),B(k_2))$.
    The white tile at position $(A(k_1)-1,\ell_2-2)$ 
    has right label $11n$.
    From the observation made in Figure~\ref{fig:vertical-sequence-white-tiles},
    the width of this return block is
    \[
        A(k_1)-A(k_1-1)=n.
    \]
    This is a contradiction.
    Thus, $\left(\begin{array}{c}
            j_n^{1,1,1,1}\\
            \widehat{g_n^{n}}
        \end{array}\right)
        \notin\Lcal(\Omega_n)$.

    The same contradiction is obtained if we suppose that
    $j_n^{0,1,1,1}$ appears at position $\bl=(\ell_1,\ell_2)\in\Z^2$
    and that $\widehat{g_n^{n-1}}$ appears at position $(\ell_1,\ell_2-1)$ in $c$.
    Indeed a blue tile with left label $11n$ is also forced to appear at
    position $(\ell_1,\ell_2-2)$.
\end{proof}

Note that Lemma~\ref{lem:two-illegal-vertical-dominoes} 
cannot be extended to the case $n=1$.

\begin{proposition}\label{prop:minimality-when-n>=2}
    For every integer $n\geq2$, the Wang shift $\Omega_n$ is minimal
    and is equal to the substitutive subshift $\Omega_n=\Xcal_{\omega_n}$.
\end{proposition}

\begin{proof}
    Let $n\geq2$ be an integer.
    From Theorem~\ref{thm:primitivity},
    the 2-dimensional substitution 
    $\omega_n$ is primitive. Also $\omega_n$ is expansive.
    From Theorem~\ref{thm:similar-to-itself},
    the Wang shift $\Omega_n$ is self-similar
    satisfying $\Omega_n=\shiftclosure{\omega_n(\Omega_n)}$.
    Therefore, we may use Lemma~\ref{lem:criterion-for-minimality}
    to show that the Wang shift $\Omega_n$ is minimal
    and $\Xcal_{\omega_n}=\Omega_n$.
    From Lemma~\ref{lem:criterion-for-minimality},
    our goal is show that
    \[
        \Lcal(\Omega_n)\cap\RecurrentVertices(G^s_{\omega_n})\subset\Lcal(\Xcal_{\omega_n})
    \]
    for every size
    $s\in\{2\times2, 2\times1, 1\times2\}$.

    \textsc{Case} $s=1\times2$.
    We have
\begingroup
\allowdisplaybreaks
\begin{align*}
    &\RecurrentVertices(G^{1\times2}_{\omega_n})\\
    &\subseteq
    \left\{
        \left(\begin{array}{c}
            a\\
            c
        \end{array}\right)
        \middle|
        \begin{array}{l}
        \text{there exists $e,g\in\Acal$ such that $\width(\omega_n(g))=\width(\omega_n(e))$}\\
        \text{there exists an integer $i$ such that $0\leq i < \width(\omega_n(e))$ and}\\
        a \text{ is the letter in the $i$-th column in the bottom-most row of } \omega_n(e),\\
        c \text{ is the letter in the $i$-th column in the top-most row of } \omega_n(g)\\
        \end{array}
        \right\}\\
    &=
    \left\{
        \left(\begin{array}{c}
            a\\
            c
        \end{array}\right)
        \middle|
        \begin{array}{l}
            a\in\{
                  j_n^{0,0,0,0},
                  j_n^{0,0,0,1},
                  j_n^{0,1,0,0},
                  j_n^{0,1,0,1},
                  j_n^{1,1,0,1}
              \}\\
            c\in\{\widehat{b_n^{n-1}},\widehat{g_n^{n-1}},\widehat{g_n^{n}}\}
        \end{array}
    \right\}\\
    &\quad\cup
    \left\{
        \left(\begin{array}{c}
            a\\
            c
        \end{array}\right)
        \middle|
        \begin{array}{l}
            a\in\{
                  j_n^{0,1,0,0},
                  j_n^{0,1,0,1},
                  j_n^{0,1,1,1},
                  j_n^{1,1,0,1},
                  j_n^{1,1,1,1}
              \}\\
            c\in\{\widehat{y_n^{n-1}},\widehat{y_n^{n}},
                  \widehat{g_n^{n-1}},\widehat{g_n^{n}}\}
        \end{array}
    \right\}\\
    &\quad\cup
    \left\{
        \left(\begin{array}{c}
            g_n^{i-1}\\
            w_n^{i,n}
        \end{array}\right),
        \left(\begin{array}{c}
            g_n^{i-1}\\
            w_n^{i,n-1}
        \end{array}\right),
        \left(\begin{array}{c}
            g_n^{i}\\
            w_n^{i,n}
        \end{array}\right),
        \left(\begin{array}{c}
            g_n^{i}\\
            w_n^{i,n-1}
        \end{array}\right)
        \middle|
        \;
        1\leq i\leq n
    \right\}\\
    &\quad\cup
    \left\{
        \left(\begin{array}{c}
            b_n^{i-1}\\
            w_n^{i,n}
        \end{array}\right),
        \left(\begin{array}{c}
            b_n^{i-1}\\
            w_n^{i,n-1}
        \end{array}\right)
        \middle|
        \;
        1\leq i\leq n
    \right\}
    \cup
    \left\{
        \left(\begin{array}{c}
            b_n^{i}\\
            w_n^{i,n}
        \end{array}\right),
        \left(\begin{array}{c}
            b_n^{i}\\
            w_n^{i,n-1}
        \end{array}\right)
        \middle|
        \;
        1\leq i\leq n-1
    \right\}\\
    &\quad\cup
    \left\{
        \left(\begin{array}{c}
            y_n^{i-1}\\
            w_n^{i,n}
        \end{array}\right),
        \left(\begin{array}{c}
            y_n^{i-1}\\
            w_n^{i,n-1}
        \end{array}\right)
        \middle|
        \;
        2\leq i\leq n
    \right\}
    \cup
    \left\{
        \left(\begin{array}{c}
            y_n^{i}\\
            w_n^{i,n}
        \end{array}\right),
        \left(\begin{array}{c}
            y_n^{i}\\
            w_n^{i,n-1}
        \end{array}\right)
        \middle|
        \;
        1\leq i\leq n
    \right\}.
\end{align*}
\endgroup
On the other hand, we can estimate the set of vertical dominoes in
$\Lcal(\Omega_n)$ by the pair of tiles sharing the same label on the common
horizontal edge excluding the two illegal dominoes from
Lemma~\ref{lem:two-illegal-vertical-dominoes}:
\begingroup
\allowdisplaybreaks
\begin{align*}
    &\Lcal_{1\times 2}(\Omega_n)
    \cap 
    \left\{
        \left(\begin{array}{c}
            a\\
            c
        \end{array}\right)
        \middle|
        \; a \text{ is a junction tile or a horizontal stripe tile }
        \right\}\\
    &\subseteq
    \left\{
        \left(\begin{array}{c}
            a\\
            c
        \end{array}\right)
        \middle|
        \begin{array}{l}
            a\in\{
                  j_n^{0,1,0,0},
                  j_n^{0,1,0,1},
                  j_n^{0,1,1,1}
              \}\\
            c\in\{\widehat{g_n^{n-1}},\widehat{y_n^{n-1}}\}
        \end{array}
    \right\} 
    \setminus
    \left\{
        \left(\begin{array}{c}
            j_n^{0,1,1,1}\\
            \widehat{g_n^{n-1}}
        \end{array}\right)
    \right\}
    \hspace{10mm} \text{(tiles sharing edge label $01n$)} \\
    &\quad\cup
    \left\{
        \left(\begin{array}{c}
            a\\
            c
        \end{array}\right)
        \middle|
        \begin{array}{l}
            a\in\{
                  j_n^{1,1,0,1},
                  j_n^{1,1,1,1}
              \}\\
            c\in\{\widehat{g_n^{n}},\widehat{y_n^{n}}\}
        \end{array}
    \right\} 
    \setminus
    \left\{
        \left(\begin{array}{c}
            j_n^{1,1,1,1}\\
            \widehat{g_n^{n}}
        \end{array}\right)
    \right\}
    \hspace{10mm} \text{(tiles sharing edge label $01\nbar$)} \\
    &\quad\cup
    \left\{
        \left(\begin{array}{c}
            a\\
            c
        \end{array}\right)
        \middle|
        \begin{array}{l}
            a\in\{
                  j_n^{0,0,0,0},
                  j_n^{0,0,0,1}
              \}\\
            c\in\{\widehat{b_n^{n-1}}\}
        \end{array}
    \right\} \hspace{15mm} \text{(tiles sharing edge label $00n$)} \\
    &\quad\cup
    \left\{
        \left(\begin{array}{c}
            b_n^i\\
            w_n^{k,n-1}
        \end{array}\right)
        \middle|
        \begin{array}{l}
            0\leq i\leq n-1,\\
            1\leq k\leq n
        \end{array}
    \right\} \hspace{15mm} \text{(tiles sharing edge label $11n$)} \\
    &\quad\cup
    \left\{
        \left(\begin{array}{c}
            y_n^i\\
            w_n^{k,n}
        \end{array}\right)
        \middle|
        \begin{array}{l}
            1\leq i\leq n,\\
            1\leq k\leq n
        \end{array}
    \right\} \hspace{15mm} \text{(tiles sharing edge label $11\nbar$)} \\
    &\quad\cup
    \left\{
        \left(\begin{array}{c}
            g_n^i\\
            w_n^{k,n}
        \end{array}\right)
        \middle|
        \begin{array}{l}
            0\leq i\leq n,\\
            1\leq k\leq n
        \end{array}
    \right\} \hspace{15mm} \text{(tiles sharing edge label $11\nbar$).}
\end{align*}
\endgroup

Note that
    \[
        \RecurrentVertices(G^{1\times2}_{\omega_n})
    \subset
    \left\{
        \left(\begin{array}{c}
            a\\
            c
        \end{array}\right)
        \middle|
        \; a \text{ is a junction tile or a horizontal stripe tile }
        \right\}.\]
Thus, we can compute the intersection of the two sets
and using Lemma~\ref{lem:language-vertical-dominoes},
we obtain
\begingroup
\allowdisplaybreaks
\begin{align*}
    &\RecurrentVertices(G^{1\times 2}_{\omega_n})\cap\Lcal(\Omega_n)\\
    &=\RecurrentVertices(G^{1\times 2}_{\omega_n})\cap\Lcal_{1\times 2}(\Omega_n)\\
    &=
    \left\{
        \left(\begin{array}{c}
            a\\
            c
        \end{array}\right)
        \middle|
        \begin{array}{l}
            a\in\{
                  j_n^{0,1,0,0},
                  j_n^{0,1,0,1},
                  j_n^{0,1,1,1}
              \}\\
            c\in\{\widehat{g_n^{n-1}},\widehat{y_n^{n-1}}\}
        \end{array}
    \right\} %
    \setminus
    \left\{
        \left(\begin{array}{c}
            j_n^{0,1,1,1}\\
            \widehat{g_n^{n-1}}
        \end{array}\right)
    \right\}
    \\
    &\quad\cup
    \left\{
        \left(\begin{array}{c}
            a\\
            c
        \end{array}\right)
        \middle|
        \begin{array}{l}
            a\in\{
                  j_n^{1,1,0,1},
                  j_n^{1,1,1,1}
              \}\\
            c\in\{\widehat{g_n^{n}},\widehat{y_n^{n}}\}
        \end{array}
    \right\} %
    \setminus
    \left\{
        \left(\begin{array}{c}
            j_n^{1,1,1,1}\\
            \widehat{g_n^{n}}
        \end{array}\right)
    \right\}
    \\
    &\quad\cup
    \left\{
        \left(\begin{array}{c}
            a\\
            c
        \end{array}\right)
        \middle|
        \begin{array}{l}
            a\in\{
                  j_n^{0,0,0,0},
                  j_n^{0,0,0,1}
              \}\\
            c\in\{\widehat{b_n^{n-1}}\}
        \end{array}
    \right\} %
    \\
    &\quad\cup
    \left\{
        \left(\begin{array}{c}
            g_n^{i-1}\\
            w_n^{i,n}
        \end{array}\right),
        \left(\begin{array}{c}
            g_n^{i}\\
            w_n^{i,n}
        \end{array}\right)
        \middle|
        \;
        1\leq i\leq n
    \right\}\\
    &\quad\cup
    \left\{
        \left(\begin{array}{c}
            b_n^{i-1}\\
            w_n^{i,n-1}
        \end{array}\right)
        \middle|
        \;
        1\leq i\leq n
    \right\}
    \cup
    \left\{
        \left(\begin{array}{c}
            b_n^{i}\\
            w_n^{i,n-1}
        \end{array}\right)
        \middle|
        \;
        1\leq i\leq n-1
    \right\}\\
    &\quad\cup
    \left\{
        \left(\begin{array}{c}
            y_n^{i-1}\\
            w_n^{i,n}
        \end{array}\right)
        \middle|
        \;
        2\leq i\leq n
    \right\}
    \cup
    \left\{
        \left(\begin{array}{c}
            y_n^{i}\\
            w_n^{i,n}
        \end{array}\right)
        \middle|
        \;
        1\leq i\leq n
    \right\}\\
&\subset\Lcal_{1\times 2}(\Xcal_{\omega_n})
    \subset\Lcal(\Xcal_{\omega_n}).
\end{align*}
\endgroup

    \textsc{Case} $s=2\times1$.
    The condition is satisfied because this case is symmetric to the case $s=1\times 2$.

    \textsc{Case} $s=2\times2$.
    The tiles appearing on the corners of images of letters under $\omega_n$ is quite restricted.
    Therefore, we have the following inclusion
\begingroup
\allowdisplaybreaks
\begin{align*}
    &\RecurrentVertices(G^{2\times2}_{\omega_n})\\
    &\subseteq
    \left\{
        \left(\begin{array}{cc}
            a&b\\
            c&d
        \end{array}\right)
        \middle|
        \begin{array}{l}
        \text{there exist $e,f,g,h\in\Tcal_n$ such that }\\
        a \text{ is the bottom right letter of } \omega(e),\\
        b \text{ is the bottom left letter of }  \omega(f),\\
        c \text{ is the top right letter of }   \omega(g),\\
        d \text{ is the top left letter of }    \omega(h)\\
        \end{array}
        \right\}\\
    &=
    \left\{
        \left(\begin{array}{cc}
            a&b\\
            c&d
        \end{array}\right)
        \middle|
        \begin{array}{l}
            a \in \{ b_n^{n-1}, g_n^{n-1}\}\\
            b \in \{ j_n^{0,1,0,0}, j_n^{0,1,0,1}, j_n^{0,0,0,1}, j_n^{0,0,0,0} \}\\
            c \in \{ w_n^{n-1,n-1}, w_n^{n,n-1}, w_n^{n-1,n}, w_n^{n,n} \}\\
            d \in \{ \widehat{b_n^{n-1}}, \widehat{g_n^{n-1}}\}\\
        \end{array}
        \right\}.
\end{align*}
\endgroup
The above set has size $2\times4\times4\times2=64$.
Of those, only four belong to $\Lcal(\Omega_n)$ because the choice made for the tile $b$
imposes a unique choice for the tiles $a$, $d$ and $c$.
Thus, using Lemma~\ref{lem:language-2x2-seeds}, we obtain
\begingroup
\allowdisplaybreaks
\begin{align*}
    &\Lcal(\Omega_n)
    \cap \RecurrentVertices(G^{2\times2}_{\omega_n})\\
    &=\Lcal_{2\times 2}(\Omega_n)
    \cap \RecurrentVertices(G^{2\times2}_{\omega_n})\\
    &\subseteq
    \left\{
        \left(\begin{array}{cc}
            b_n^{n-1}     & j_n^{0,0,0,0}\\
            w_n^{n-1,n-1} & \widehat{b_n^{n-1}}
        \end{array}\right),
        \left(\begin{array}{cc}
            g_n^{n-1}     & j_n^{0,1,0,1}\\
            w_n^{n,n} & \widehat{g_n^{n-1}}
        \end{array}\right),
        \left(\begin{array}{cc}
            b_n^{n-1}     & j_n^{0,1,0,0}\\
            w_n^{n,n-1} & \widehat{g_n^{n-1}}
        \end{array}\right),
        \left(\begin{array}{cc}
            g_n^{n-1}     & j_n^{0,0,0,1}\\
            w_n^{n-1,n} & \widehat{b_n^{n-1}}
        \end{array}\right)
    \right\}\\
    &=
    \left\{
        \raisebox{-10mm}{
        \begin{tikzpicture}
        \tikzstyle{every node}=[font=\scriptsize]
        \tilelabelinsideH{\ourColorBlue}{0}{1}{00n}{111}{00\nunder}{11n}
        \tilelabelinside{white}  {0}{0}{11n}{11n}{11\nunder}{11\nunder}
        \tileJunctionInsideOOOO  {1}{1}{000}{000}{00n}{00n}
        \tilelabelinsideV{\ourColorBlue}{1}{0}{111}{00n}{11n}{00\nunder}
        \end{tikzpicture}
        },
        \raisebox{-10mm}{
        \begin{tikzpicture}
        \tikzstyle{every node}=[font=\scriptsize]
        \tileHgreenInside{0}{1}{01n}{111}{00\nunder}{11\nbar}
        \tilelabelinside{white}  {0}{0}{11\nbar}{11\nbar}{11n}{11n}
        \tileJunctionInsideOOII  {1}{1}{001}{001}{01n}{01n}
        \tileVgreenInside{1}{0}{111}{01n}{11\nbar}{00\nunder}
        \end{tikzpicture}
        },
        \raisebox{-10mm}{
        \begin{tikzpicture}
        \tikzstyle{every node}=[font=\scriptsize]
        \tilelabelinsideH{\ourColorBlue}{0}{1}{00n}{111}{00\nunder}{11n}
        \tilelabelinside{white}  {0}{0}{11\nbar}{11n}{11n}{11\nunder}
        \tileJunctionInsideOOOI  {1}{1}{001}{000}{00n}{01n}
        \tileVgreenInside{1}{0}{111}{01n}{11\nbar}{00\nunder}
        \end{tikzpicture}
        },
        \raisebox{-10mm}{
        \begin{tikzpicture}
        \tikzstyle{every node}=[font=\scriptsize]
        \tileHgreenInside{0}{1}{01n}{111}{00\nunder}{11\nbar}
        \tilelabelinside{white}  {0}{0}{11n}{11\nbar}{11\nunder}{11n}
        \tileJunctionInsideOOIO  {1}{1}{000}{001}{01n}{00n}
        \tilelabelinsideV{\ourColorBlue}{1}{0}{111}{00n}{11n}{00\nunder}
        \end{tikzpicture}
        }
    \right\}\\
    &\subset\Lcal_{2\times 2}(\Xcal_{\omega_n})
     \subset\Lcal(\Xcal_{\omega_n}).
\end{align*}
\endgroup
    From Lemma~\ref{lem:criterion-for-minimality},
    we conclude that
    the Wang shift $\Omega_n$ is minimal
    and $\Omega_n=\Xcal_{\omega_n}$.
\end{proof}

\subsection{The Wang shift $\Omega_n$ is minimal when $n=1$}

From Theorem~\ref{thm:n=1-equiv-Ammann}, $\Tcal_n$ is equivalent to the 
16 Ammann Wang tiles when $n=1$.
We know from \cite{MR857454}
that the 16 Ammann Wang tiles are self-similar and that the self-similarity
is recognizable (the decomposition of every configuration into the 16
supertiles shown in \cite[Figure~11.1.6]{MR857454} is unique).
This corresponds to the case $n=1$ of Theorem~\ref{thm:similar-to-itself} proved here.
Therefore, from Lemma~\ref{lem:substitutive-contains-self-similar-part},
we have $\Xcal_{\omega_1}\subseteq\Omega_1$. The goal of this section is to
prove that the equality holds and therefore that $\Omega_1$ is minimal.
Note that minimality of $\Omega_1$ was not proved in \cite{MR857454},
neither in the more recent works about Ammann A2 tilings
\cite{akiyama_note_2012,zbMATH07187340}.

The proof made in the previous section for $n\geq2$ does not directly work
for $n=1$ because it is not true anymore that next to a junction tile is never
a junction tile. Indeed, when $n=1$, two junction tiles can be adjacent
horizontally or vertically. This observation changes the description of
vertical and horizontal dominoes that appear in the language.

Adapting the proof made above for $n\geq2$ to the case $n=1$ is possible. 
But, instead of doing this, we have chosen
to provide a proof based on computer experiments
in order to check that the criterion provided in
Lemma~\ref{lem:criterion-for-minimality} is satisfied.
We hope that it may be useful to study other examples.

\begin{lemma}\label{lem:minimality-when-n=1}
    The Wang shift $\Omega_1$ is minimal and $\Omega_1=\Xcal_{\omega_1}$.
\end{lemma}

\begin{proof}
    From Theorem~\ref{thm:primitivity},
    the 2-dimensional substitution 
    $\omega_1$ is primitive. Also $\omega_1$ is expansive.
    From Theorem~\ref{thm:similar-to-itself},
    the Wang shift $\Omega_1$ is self-similar
    satisfying $\Omega_1=\shiftclosure{\omega_1(\Omega_1)}$.
    Therefore, we may use Lemma~\ref{lem:criterion-for-minimality}
    to show the minimality of $\Omega_1$.

    We compute below the patterns in 
    $\Lcal_s(\Omega_1)$ and $\Lcal_s(\Xcal_{\omega_1})$
    for every size $s\in\{2\times2, 2\times1, 1\times2\}$.
    As we observe below, these sets are equal.
    Therefore, it is not necessary to compute $\RecurrentVertices(G^s_{\omega_1})$.
    We define $\omega_1$ as a $2$-dimensional substitution
    over the alphabet $\{0,1,2,\dots,15\}$ according to the labeling of the tiles
    shown in Figure~\ref{fig:W1_sub}.
We compute the patterns of size $s\in\{2\times2, 2\times1, 1\times2\}$ in the
substitutive subshift $\Xcal_{\omega_1}$:
\begin{sagecommandline}
sage: from slabbe import Substitution2d
sage: omega1 = Substitution2d({0: [[9], [15]], 1: [[6], [7]], 2: [[13], [14]], 3: [[6]], 4: [[5], [7]], 5: [[12, 4], [11, 3]], 6: [[12, 1], [11, 3]], 7: [[8, 4]], 8: [[13, 0], [14, 3]], 9: [[12, 4], [14, 3]], 10: [[12, 1], [14, 3]], 11: [[6, 2]], 12: [[9, 0], [15, 3]], 13: [[8, 4], [15, 3]], 14: [[10, 2]], 15: [[9, 0]]})
sage: patterns_1x2_in_subst_shift = set((a,b) for [[a,b]] in omega1.list_dominoes(direction="vertical", output_format="list_of_lists"))
sage: len(patterns_1x2_in_subst_shift)
30
sage: min(patterns_1x2_in_subst_shift)   # show some minimal element
(0, 5)
sage: patterns_2x1_in_subst_shift = set((a,b) for [[a],[b]] in omega1.list_dominoes(direction="horizontal", output_format="list_of_lists"))
sage: len(patterns_2x1_in_subst_shift)
30
sage: min(patterns_2x1_in_subst_shift)   # show some minimal element
(0, 1)
sage: patterns_2x2_in_subst_shift = sorted(omega1.list_2x2_factors())
sage: len(patterns_2x2_in_subst_shift)
51
sage: min(patterns_2x2_in_subst_shift)   # show some minimal element
[[0, 5], [3, 7]]
\end{sagecommandline}

We choose a solver to compute the dominoes and $2\times 2$ patterns below.
Three reductions are available: to a mixed-integer linear program, 
to a SAT instance or
to an exact cover problem solved with Knuth's dancing links algorithm
\cite{knuth_dancing_2000}. We use Knuth's algorithm because it performs well
and it is in SageMath by default.

\begin{sagecommandline}
sage: solver = "dancing_links" # other options are: solver="gurobi" or solver="kissat"
\end{sagecommandline}

    We define the set $\Tcal_1$ of Wang tiles in an order consistent with
    the labeling of the tiles with the indices in the set $\{0,1,2,\dots,15\}$
    as shown in Figure~\ref{fig:W1_sub}.
We compute the patterns of size $s\in\{2\times2, 2\times1, 1\times2\}$ in the
Wang shift $\Omega_1$:
\begin{sagecommandline}
sage: from slabbe import WangTileSet
sage: tiles = [("111", "012", "112", "001"), ("111", "001", "111", "000"), ("112", "012", "112", "011"), ("112", "112", "111", "111"), ("111", "011", "112", "000"), ("011", "001", "011", "012"), ("011", "011", "012", "012"), ("012", "112", "011", "112"), ("001", "000", "001", "011"), ("001", "001", "011", "011"), ("001", "011", "012", "011"), ("001", "111", "000", "111"), ("000", "000", "001", "001"), ("000", "001", "011", "001"), ("011", "111", "000", "112"), ("012", "111", "001", "112")]
sage: T1 = WangTileSet(tiles)
sage: T1
Wang tile set of cardinality 16
sage: patterns_1x2_in_sft = T1.dominoes_with_surrounding(i=2, radius=1, solver=solver)
sage: len(patterns_1x2_in_sft)
30
sage: min(patterns_1x2_in_sft)    # show some minimal element
(0, 5)
sage: patterns_2x1_in_sft = T1.dominoes_with_surrounding(i=1, radius=1, solver=solver)
sage: len(patterns_2x1_in_sft)
30
sage: min(patterns_2x1_in_sft)    # show some minimal element
(0, 1)
sage: patterns_2x2_in_sft = T1.tilings_with_surrounding(2,2, radius=3, solver=solver)
sage: patterns_2x2_in_sft = sorted(pattern.table() for pattern in patterns_2x2_in_sft)
sage: len(patterns_2x2_in_sft)
51
sage: min(patterns_2x2_in_sft)    # show some minimal element
[[0, 5], [3, 7]]
\end{sagecommandline}

We compare the sets of horizontal dominoes, vertical dominoes and $2\times 2$ patterns
computed above within the language of the substitutive subshift $\Xcal_{\omega_1}$ and 
within the language of the Wang shift $\Omega_1$.
We observe their equality:
\begin{sagecommandline}
sage: patterns_1x2_in_subst_shift == patterns_1x2_in_sft
True
sage: patterns_2x1_in_subst_shift == patterns_2x1_in_sft
True
sage: patterns_2x2_in_subst_shift == patterns_2x2_in_sft
True
\end{sagecommandline}
Therefore, the above computations prove that we have the following equality
\[
    \Lcal_s(\Omega_1) = \Lcal_s(\Xcal_{\omega_1})
\]
for every size $s\in\{2\times2, 2\times1, 1\times2\}$.
Thus, for every size
$s\in\{2\times2, 2\times1, 1\times2\}$, we have
    \[
        \Lcal(\Omega_1)\cap\RecurrentVertices(G^s_{\omega_1})
        \subset\Lcal_s(\Omega_1)
        = \Lcal_s(\Xcal_{\omega_1})
        \subset\Lcal(\Xcal_{\omega_1}).
    \]
    From Lemma~\ref{lem:criterion-for-minimality}, we conclude that
    $\Omega_1$ is minimal and $\Omega_1=\Xcal_{\omega_1}$.
\end{proof}

\subsection{Proof of Theorem~\ref{thm:minimal}}

\begin{THEOREMIII}
    \MainTheoremIII
\end{THEOREMIII}

\begin{proof}%
    If $n=1$, then $\Omega_1$ is minimal 
    and $\Omega_1=\Xcal_{\omega_1}$
    from Lemma~\ref{lem:minimality-when-n=1}.
    If $n\geq2$, then $\Omega_n$ is minimal 
    and $\Omega_n=\Xcal_{\omega_n}$
    from Proposition~\ref{prop:minimality-when-n>=2}.
\end{proof}

\section{Open questions}
\label{sec:question}

Note that the $n^{th}$ metallic mean is a quadratic Pisot unit, that is, 
it is an algebraic unit of degree two and all its
algebraic conjugates have modulus strictly less than one.
The other quadratic Pisot units are the positive roots of $x^2-nx+1$ for
$n\geq3$. The family of quadratic Pisot units has nice properties
\cite{MR1934431,MR1906191,MR3304612}; see also \cite{MR2994362}.
The continued fraction expansion of the positive root of
$x^2-nx+1$ is $[n-1; (1, n-2)^\infty]$. In particular, it is not purely
periodic. 

\begin{mainquestion}
    Let $\beta$ be a positive quadratic Pisot unit which is not a metallic mean.
    Can we construct a self-similar set of Wang tiles
    whose inflation factor is $\beta$?
\end{mainquestion}

An alternative question is about those quadratic integers whose
continued fraction expansion is purely periodic.

\begin{mainquestion}
    Let $\beta$ be a positive quadratic integer whose continued fraction
    expansion is purely periodic. Does there exist a set of Wang tiles such
    that the shift is self-similar with inflation factor equal to
    $\beta$?
\end{mainquestion}

The procedure explained in \cite[p.594--598]{MR857454} starts from 
the Ammann A2 shapes shown in Figure~\ref{fig:Ammann-A2-with-bars} and constructs
a set of 16 Wang tiles which we show in Theorem~\ref{thm:n=1-equiv-Ammann} 
to be equivalent to the set $\Tcal_1$.
A question we can ask is whether this construction can be inverted.
More precisely, starting from the Ammann set of 16 Wang tiles, can we recover the
two Ammann shapes shown in Figure~\ref{fig:Ammann-A2-with-bars} with their Ammann bars?
In general, we ask the following question.

\begin{mainquestion}\label{question:Ammann-bars}
    For every integer $n\geq1$,
    can we find geometrical shapes 
    with Ammann bars on them such that
    encoding their tilings by rhombi along a pair of Ammann bars is equivalent
    to the tiles $\Tcal_n$?
\end{mainquestion}

Theorem~\ref{thm:n=1-equiv-Ammann} together with the discussion
\cite[p.594--598]{MR857454} is an answer
to Question~\ref{question:Ammann-bars} when $n=1$.
An answer to Question~\ref{question:Ammann-bars} 
would shed light on Mr.~Ammann's remarkable insights \cite{MR2104463}.

\subsection*{Relation to the work of Mozes}

Let $n\geq1$ be an integer and recall the 1-dimensional substitution
\[
\rho_n = 
\begin{cases}
    \texttt{a} \mapsto \texttt{ab}^{n}\\
    \texttt{b} \mapsto \texttt{ab}^{n-1}
\end{cases}
\]
over alphabet $\{\texttt{a},\texttt{b}\}$
defined in the proof of Lemma~\ref{lem:inflation-factor}.
The incidence matrix of $\rho_n$ is
$\left(\begin{smallmatrix} 1&1\\ n&n-1 \end{smallmatrix}\right)$
whose characteristic polynomial is $x^2-nx-1$,
and whose Perron--Frobenius dominant eigenvalue is the $n^{th}$ metallic mean.
A right dominant eigenvector is 
$\left(\begin{smallmatrix} 1\\ \beta_n-1 \end{smallmatrix}\right)$
    and a left dominant eigenvector is 
$\left(\begin{smallmatrix} n & \beta_n-1 \end{smallmatrix}\right)$.
    Following the theory on inflation tilings \cite[\S~6]{MR3136260},
    a stone inflation associated with the substitution $\rho_n$ gives a volume of $n$ to the
    letter $\texttt{a}$ and a volume of $\beta_n-1$ to the letter $\texttt{b}$.
The stone inflation induced by the direct product $\rho_n\times\rho_n$ 
of the substitution $\rho_n$ with itself
in the sense of 
\cite[\S~6]{MR1014984}
is shown in Figure~\ref{fig:global-stone-inflation};
see also \cite[Example~5.9]{MR3136260}.
Note that another substitution with same inflation factor 
and often used in examples illustrating metallic means is 
$\texttt{a}\mapsto \texttt{a}^n\texttt{b}, \texttt{b}\mapsto \texttt{a}$
\cite[Remark~4.7]{MR3136260}.

\begin{figure}[h]
\begin{center}
    \includegraphics[width=.6\linewidth]{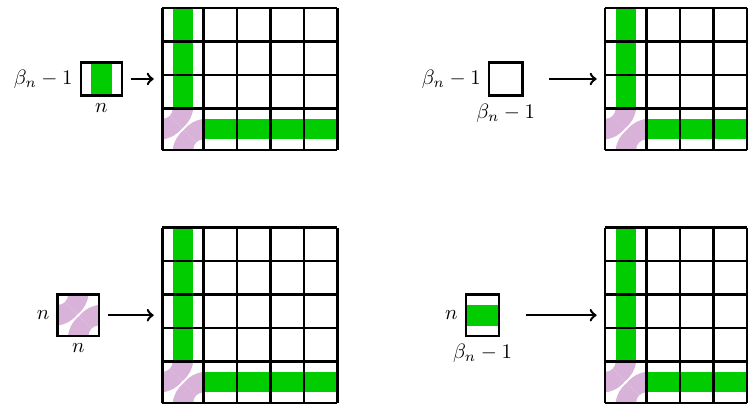}
\end{center}
\caption{Stone inflation associated with the direct product of the substitution
$\rho_n$ with itself with inflation factor equal to $\beta_n$, the $n^{th}$ metallic mean.
The size of the rectangles are given by the entries of a Perron--Frobenius dominant left-eigenvector
of the incidence matrix of $\rho_n$. The figure is drawn with parameter $n=4$.
Color is added to the tiles to differentiate them and visually link them to the
tiles in $\Tcal_n$.}
\label{fig:global-stone-inflation}
\end{figure}

From the work of Mozes \cite{MR1014984}, we know that there exists a 
tiling system given by a finite set of tiles and a finite set of matching rules
such that the tiling system is a symbolic extension
of the substitutive dynamical system generated by the 2-dimensional substitution
$\rho_n\times\rho_n$ over a four-letter alphabet. Since the substitution
$\rho_n\times\rho_n$ is recognizable (or has ``unique derivation'', using the
vocabulary of Mozes), the tiling system constructed by Mozes is even
measure-theoretically isomorphic to the substitutive dynamical system.
Note that the construction of an equivalent tiling system out of a substitution
was extended to geometric substitutions \cite{MR1609510}.

In this contribution, we provide an explicit construction of a tiling system
$\Omega_n$ which is a symbolic extension
of the 2-dimensional substitutive subshift defined by $\rho_n\times\rho_n$.
The set of Wang tiles deduced from \cite{MR1014984} 
when applied on $\rho_n\times\rho_n$ would be much larger than $(n+3)^2$.
This raises a question about the optimality of a tiling system
for 2-dimensional substitutions.

\begin{mainquestion}
    Is the size of $\Tcal_n$ optimal? 
    In other words, does there exist a set $\Tcal$ of Wang tiles of cardinality $\#\Tcal<(n+3)^2$
    such that the Wang shift $\Omega_{\Tcal}$ is isomorphic to
    the 2-dimensional substitutive subshift
    $\Xcal_{\rho_n\times\rho_n}$?
\end{mainquestion}

\section{Appendix A: The substitutions $\omega_n$ for $1\leq n\leq5$}
\label{sec:Appendix-substitutions-1ot5}

\begin{figure}[h]
\begin{center}
\includegraphics[scale=.88]{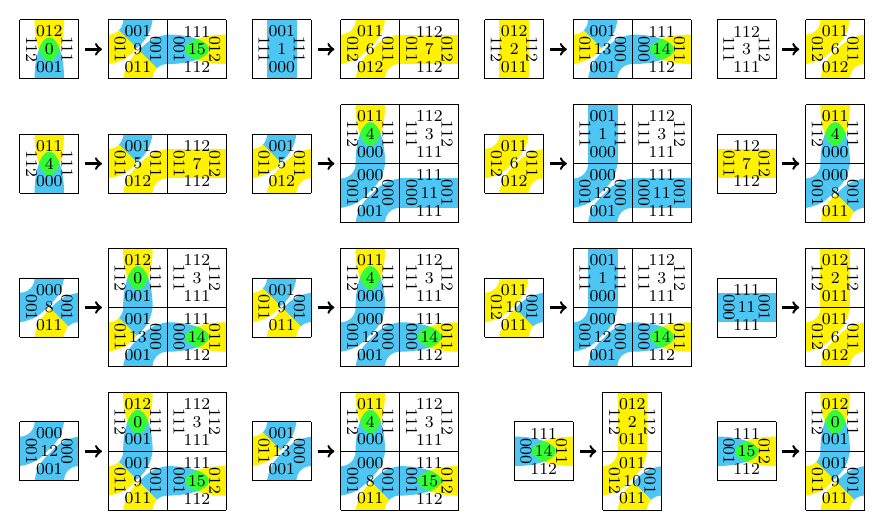}
\end{center}
    \caption{Substitution $\omega_1$}
    \label{fig:W1_sub}
\end{figure}

\begin{figure}[h]
\begin{center}
    \includegraphics[width=.99\linewidth]{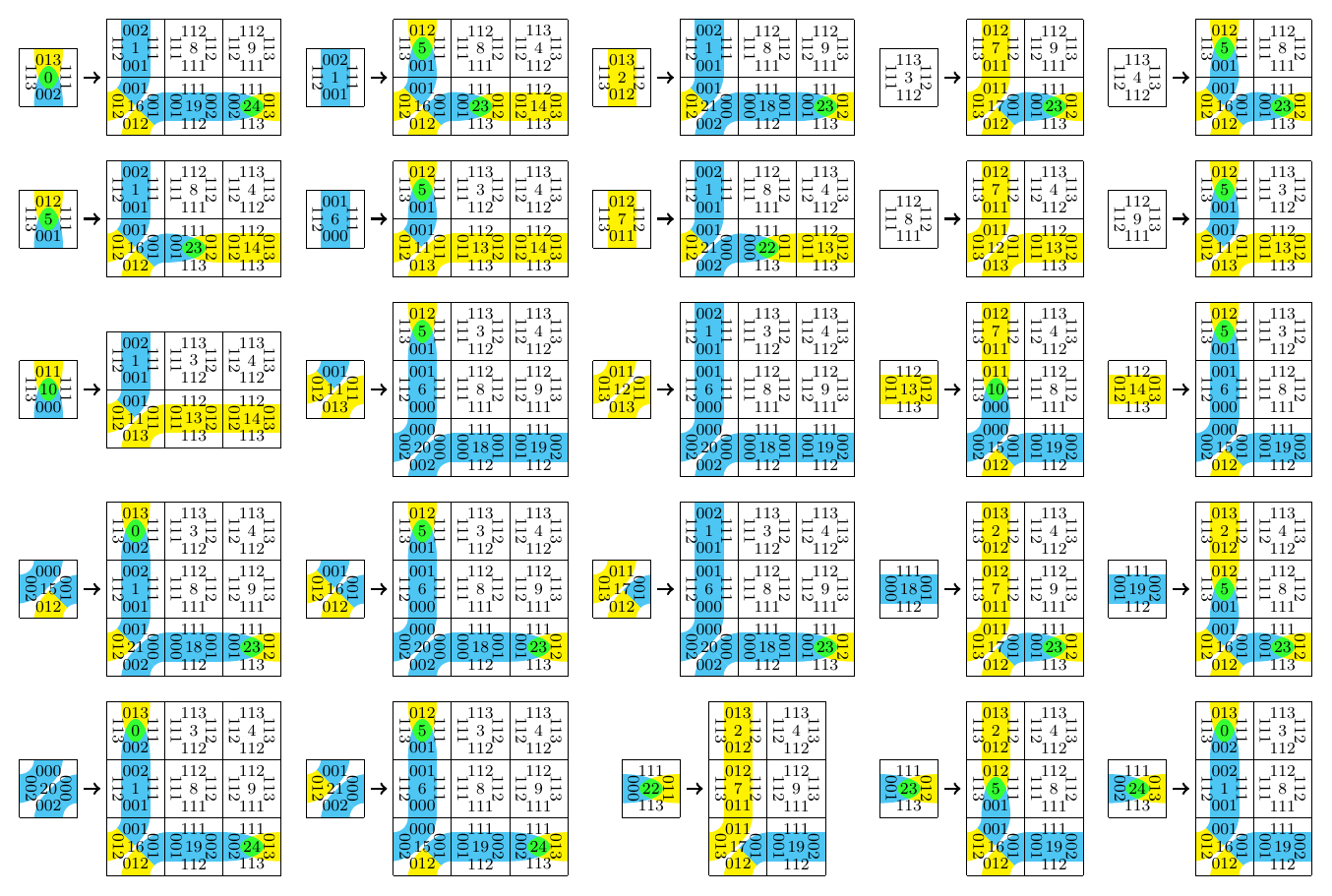}
\end{center}
    \caption{Substitution $\omega_2$}
    \label{fig:W2_sub}
\end{figure}

\begin{figure}[h]
\begin{center}
\includegraphics[width=22cm,angle=90]{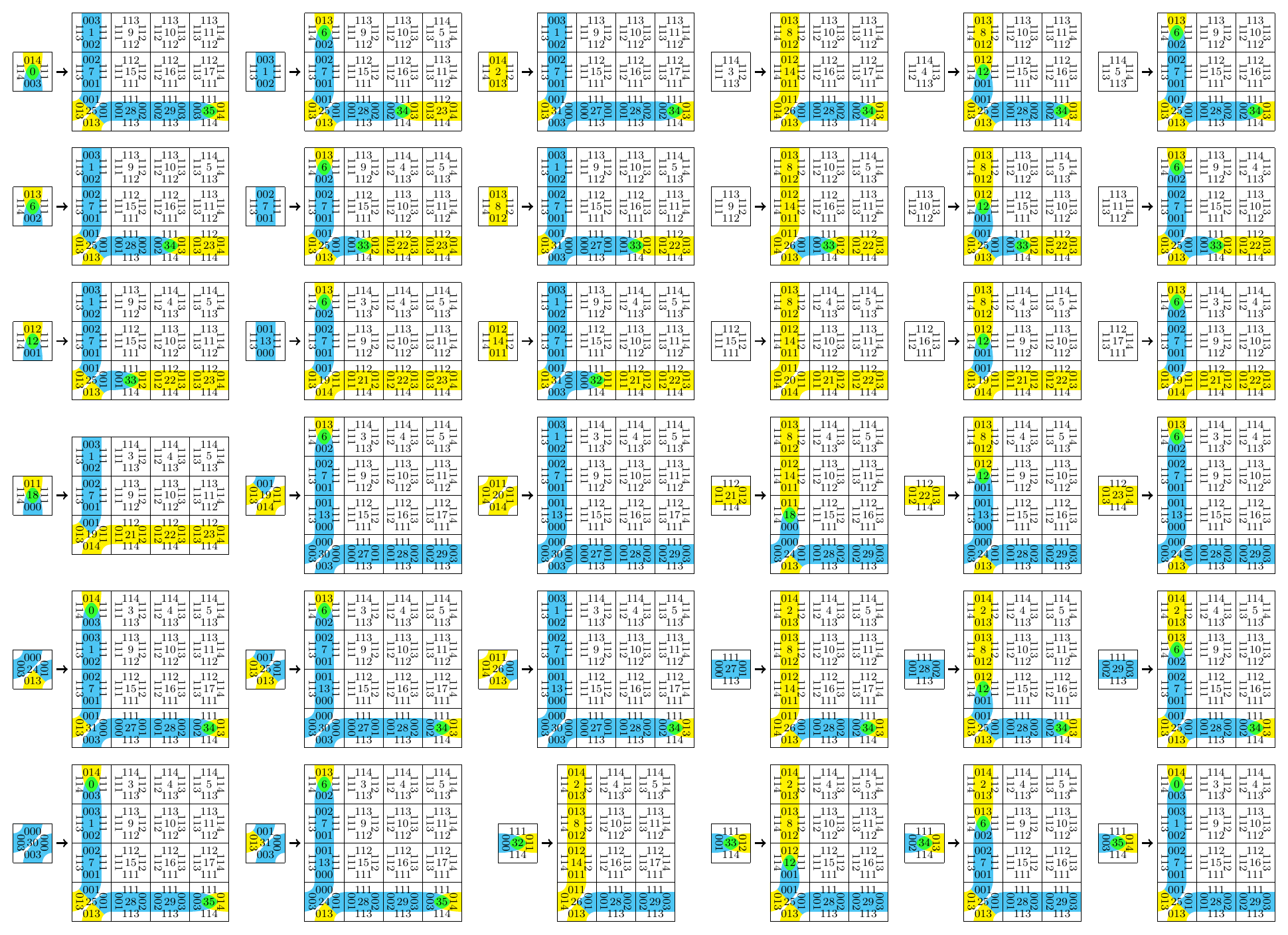}
\end{center}
    \caption{Substitution $\omega_3$ (rotated 90 degrees counterclockwise)}
    \label{fig:W3_sub}
\end{figure}

\begin{figure}[h]
\begin{center}
\includegraphics[width=22cm,angle=90]{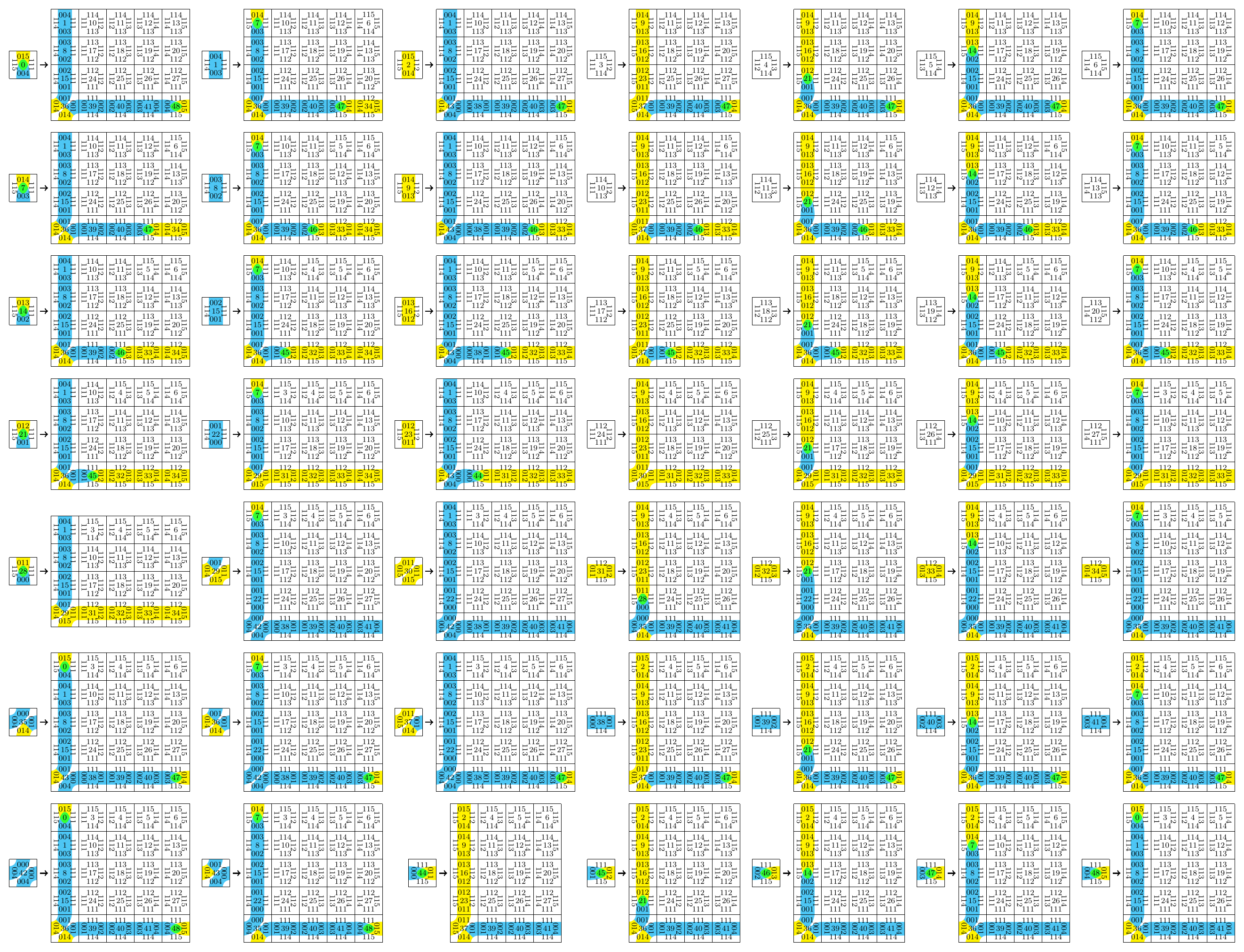}
\end{center}
    \caption{Substitution $\omega_4$ (rotated 90 degrees counterclockwise)}
    \label{fig:W4_sub}
\end{figure}

\begin{figure}[h]
\begin{center}
\includegraphics[width=22cm,angle=90]{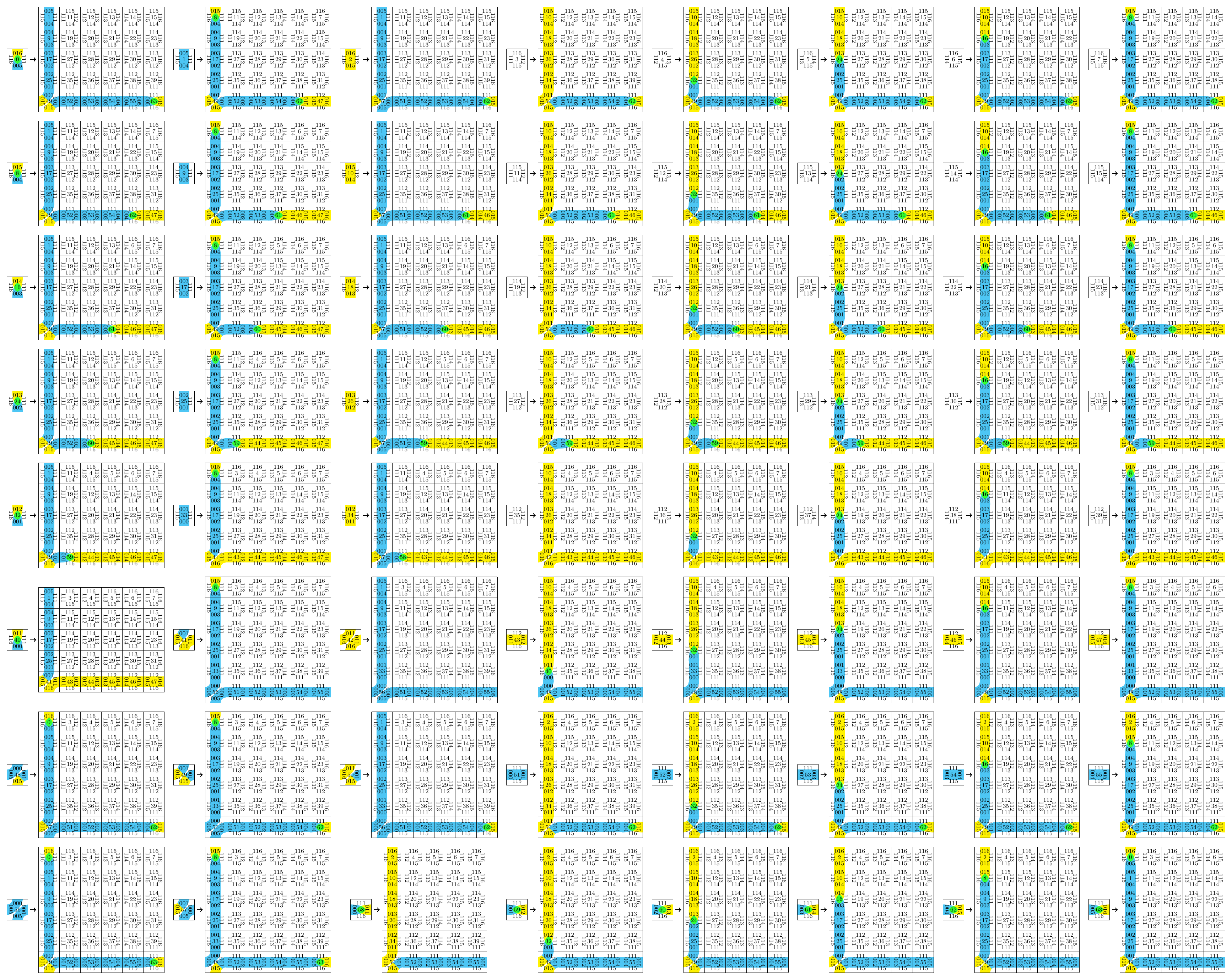}
\end{center}
    \caption{Substitution $\omega_5$ (rotated 90 degrees counterclockwise)}
    \label{fig:W5_sub}
\end{figure}

\clearpage

\section{Appendix B: Proving the self-similarity of $\Omega_2$ in SageMath}
\label{sec:Appendix-self-similarity-Omega2}

In this section, we illustrate how
Theorem~\ref{thm:similar-to-itself} can be proved in SageMath for a specific
but not too big integer $n\geq1$. 
Since the proof of Theorem~\ref{thm:similar-to-itself} given in
this article was deduced from such computer experiments
performed for small values of $n$,
we hope that the approach shown below can be used to study and show the
self-similarity of other aperiodic set of Wang tiles.

We use here a method proposed in \cite{MR4226493} to study the substitutive
structure of the Jeandel--Rao Wang shift \cite{zbMATH07421483}. 
The method is based on the notion of marker tiles (not to be confused with the
notion of marker used in Lemma 10.1.8 from \cite{MR1369092}). A nonempty subset
$M\subset\Acal$ is called \defn{markers for the direction $\be_2$} within a
subshift $X\subset\Acal^{\Z^2}$ if for every configuration $x\in X$ the
positions of the markers are nonadjacent rows, that is, $x^{-1}(M)=\Z\times P$
for some set $P\subset\Z$ such that $1\notin P-P$. A symmetric definition holds
for \defn{markers for the direction $\be_1$}. It was proved
that the existence of marker tiles allow to decompose uniquely a Wang shift.
Informally, marker tiles are merged with the tiles that appear just on top of
(or just below) them. Remaining tiles are kept unchanged.
The search for markers and the construction of the substitution is performed by
two algorithms \textsc{FindMarkers} and \textsc{FindSubstitution}. 
Their pseudocode can be found in \cite{MR4226493}; see also the chapter
\cite{labbe_three_2020} where a simpler example is considered.

Below we prove the self-similarity of $\Omega_n$ when $n=2$ using SageMath
\cite{sagemathv10.5} with optional package \texttt{slabbe} \cite{labbe_slabbe_0_7_7_2024}.
The algorithms
\textsc{FindMarkers} and \textsc{FindSubstitution} are used twice horizontally
and then twice vertically. 
The computations show that every configuration in $\Omega_2$ can be decomposed
uniquely into 25 supertiles.
The 25 supertiles are equivalent to the original set of 25 tiles. 
Thus, the Wang shift $\Omega_2$ is self-similar and we compute the self-similarity.

We choose a solver to search for markers and desubstitutions below.
\begin{sagecommandline}
sage: solver = "dancing_links" # other options are: solver="gurobi" or solver="kissat"
\end{sagecommandline}

First, we define the set $\Tcal_2$ of Wang tiles.
\begin{sagecommandline}
sage: from slabbe import WangTileSet
sage: tiles = [("111", "013", "113", "002"), ("111", "002", "112", "001"), ("112", "013", "113", "012"), ("112", "113", "111", "112"), ("113", "113", "112", "112"), ("111", "012", "113", "001"), ("111", "001", "112", "000"), ("112", "012", "113", "011"), ("112", "112", "111", "111"), ("113", "112", "112", "111"), ("111", "011", "113", "000"), ("011", "001", "012", "013"), ("011", "011", "013", "013"), ("012", "112", "011", "113"), ("013", "112", "012", "113"), ("001", "000", "002", "012"), ("001", "001", "012", "012"), ("001", "011", "013", "012"), ("001", "111", "000", "112"), ("002", "111", "001", "112"), ("000", "000", "002", "002"), ("000", "001", "012", "002"), ("011", "111", "000", "113"), ("012", "111", "001", "113"), ("013", "111", "002", "113")]
sage: T2 = WangTileSet(tiles)
sage: T2_tikz = T2.tikz(ncolumns=10, scale=1.2, label_shift=.15)
\end{sagecommandline}

\begin{center}
    \sageplot[width=.75\linewidth][pdf]{T2_tikz}
\end{center}

Then, we search for markers for the direction $\be_1$ (such markers appear on
nonadjacent columns). We fusion the markers with the possible tiles appearing
on their right (thus the marker appear on the left side of each pair).
\begin{sagecommandline}
sage: T2.find_markers(i=1, radius=1, solver=solver)
[[0, 1, 2, 5, 6, 7, 10, 11, 12, 15, 16, 17, 20, 21]]
sage: M = [0, 1, 2, 5, 6, 7, 10, 11, 12, 15, 16, 17, 20, 21]
sage: U1, s1 = T2.find_substitution(M=M, i=1, radius=2, solver=solver, side="left")
sage: s1_tikz = s1.wang_tikz(domain_tiles=U1, codomain_tiles=T2, ncolumns=5, scale=1.2, label_shift=.15, direction="left", extra_space=1.2)
\end{sagecommandline}

\begin{center}
    \sageplot[width=.95\linewidth][pdf]{s1_tikz}
\end{center}

The resulting set of Wang tiles (shown above at the source of the arrows) is
obtained by concatenating the top and bottom labels of the merged pairs:
\begin{sagecommandline}
sage: U1_tikz = U1.tikz(scale=1.4, label_shift=0.15)
\end{sagecommandline}

\begin{center}
    \sageplot[][pdf]{U1_tikz}
\end{center}

\begin{sagecommandline}
sage: U1.find_markers(i=1, radius=1, solver=solver)
[[0, 1, 2, 3, 4, 5, 6]]
sage: M = [0, 1, 2, 3, 4, 5, 6]
sage: U2, s2 = U1.find_substitution(M=M, i=1, radius=1, solver=solver)
sage: U2_tikz = U2.tikz(scale=1.7, label_shift=0.15, ncolumns=12)
\end{sagecommandline}

\begin{center}
    \sageplot[width=.95\linewidth][pdf]{U2_tikz}
\end{center}

\begin{sagecommandline}
sage: U2.find_markers(i=2, radius=1, solver=solver)
[[9, 10, 11, 12, 13, 14, 15, 16, 24, 25, 27, 28, 29, 30, 31, 32, 33]]
sage: M = [9, 10, 11, 12, 13, 14, 15, 16, 24, 25, 27, 28, 29, 30, 31, 32, 33]
sage: U3, s3 = U2.find_substitution(M=M, i=2, radius=1, solver=solver, side="left")
sage: U3_tikz = U3.tikz(scale=1.9, label_shift=0.1)
\end{sagecommandline}

\begin{center}
    \sageplot[width=.85\linewidth][pdf]{U3_tikz}
\end{center}

\begin{sagecommandline}
sage: U3.find_markers(i=2, radius=1, solver=solver)
[[0, 1, 2, 3, 4, 5, 6]]
sage: M = [0, 1, 2, 3, 4, 5, 6]
sage: U4, s4 = U3.find_substitution(M=M, i=2, radius=1, solver=solver)
sage: U4_tikz = U4.tikz(scale=2.2, label_shift=.1)
\end{sagecommandline}

\begin{center}
    \sageplot[width=.85\linewidth][pdf]{U4_tikz}
\end{center}

It turns out that tiles with indices 11, 14, 20, 27 are not needed within 
the above set of tiles as they do not have a surrounding of radius 2
as confirmed by the following computation.
Thus, they cannot appear in any tiling. In fact, they correspond to
antigreen tiles and other tiles proved to be illegal in Section~\ref{sec:Omegan'subseteqOmegan}.
We compute the remaining twenty five tiles below.

\begin{sagecommandline}
sage: U5 = U4.tiles_allowing_surrounding(radius=2, solver=solver)
sage: U5_tikz = U5.tikz(scale=2.1, label_shift=.1)
\end{sagecommandline}

\begin{center}
    \sageplot[width=.85\linewidth][pdf]{U5_tikz}
\end{center}

\begin{sagecommandline}
sage: U4_tiles = U4.tiles()
sage: U5_tiles = U5.tiles()
sage: d = {i:U4_tiles.index(U5_tiles[i]) for i in range(len(U5))}
sage: from slabbe import Substitution2d
sage: s5 = Substitution2d.from_permutation(d)
\end{sagecommandline}

We confirm that the set $U_5$ is equivalent to the set $\Tcal_n$ of Wang tiles we started with.
We extract the bijection \texttt{s6} between the indices of the tiles.
Also, it gives a bijection for the horizontal edge labels and vertical edge
labels. Both are equal. This bijection corresponds to the map $\tau_n$ when $n=2$ defined
in Section~\ref{sec:return-blocks}.

\begin{sagecommandline}
sage: T2.is_equivalent(U5)
True
sage: _,vert_bijection,horiz_bijection,s6 = T2.is_equivalent(U5, certificate=True)
sage: vert_bijection == horiz_bijection
True
sage: vert_bijection #@\label{vert_bijection}
{'000': '013113113',
 '001': '012113113',
 '002': '012112113',
 '011': '002113113',
 '012': '002112113',
 '013': '002112112',
 '111': '013113',
 '112': '012113',
 '113': '012112'}
\end{sagecommandline}

One may compare the bijection computed above with the map $\tau_n$
defined in Section~\ref{sec:substitution}.
The only difference is that the image of the label $003$ does not appear in the
computed bijection above because it is does not appear as an edge label in the
set $\Tcal_2$.

The self-similarity is: 
\renewcommand{\sagecommandlinetextoutput}{False} %
\begin{sagecommandline}
sage: self_similarity = s1*s2*s3*s4*s5*s6
sage: self_similarity
Substitution 2d: {0: [[16, 1], [19, 8], [24, 9]], 1: [[16, 5], [23, 8], [14, 4]], 2: [[21, 1], [18, 8], [23, 9]], 3: [[17, 7], [23, 9]], 4: [[16, 5], [23, 8]], 5: [[16, 1], [23, 8], [14, 4]], 6: [[11, 5], [13, 3], [14, 4]], 7: [[21, 1], [22, 8], [13, 4]], 8: [[12, 7], [13, 4]], 9: [[11, 5], [13, 3]], 10: [[11, 1], [13, 3], [14, 4]], 11: [[20, 6, 5], [18, 8, 3], [19, 9, 4]], 12: [[20, 6, 1], [18, 8, 3], [19, 9, 4]], 13: [[15, 10, 7], [19, 8, 4]], 14: [[15, 6, 5], [19, 8, 3]], 15: [[21, 1, 0], [18, 8, 3], [23, 9, 4]], 16: [[20, 6, 5], [18, 8, 3], [23, 9, 4]], 17: [[20, 6, 1], [18, 8, 3], [23, 9, 4]], 18: [[17, 7, 2], [23, 9, 4]], 19: [[16, 5, 2], [23, 8, 4]], 20: [[16, 1, 0], [19, 8, 3], [24, 9, 4]], 21: [[15, 6, 5], [19, 8, 3], [24, 9, 4]], 22: [[17, 7, 2], [19, 9, 4]], 23: [[16, 5, 2], [19, 8, 4]], 24: [[16, 1, 0], [19, 8, 3]]}
\end{sagecommandline}

The characteristic polynomial of the incidence matrix of the self-similarity is:
\begin{sagecommandline}
sage: matrix(self_similarity).charpoly().factor()
(x - 1)^3 * (x + 1)^5 * x^11 * (x^2 - 6*x + 1) * (x^2 + 2*x - 1)^2
\end{sagecommandline}
\renewcommand{\sagecommandlinetextoutput}{True} %

The self-similarity shown with the associated Wang tiles:
\begin{sagecommandline}
sage: sim_tikz = self_similarity.wang_tikz(domain_tiles=T2, codomain_tiles=T2, ncolumns=5, scale=1.2, label_shift=.15)
\end{sagecommandline}

\begin{center}
    \sageplot[width=.95\linewidth][pdf]{sim_tikz}
\end{center}

\bibliographystyle{myalpha} %
\bibliography{biblio}

\end{document}